%% file: main.tex
\documentclass[a4paper]{article}
\input{preamble}

\title{Unstable \'etale motives}

\begin{document}

\maketitle 

\begin{abstract}
    We prove a rigidity result for certain $p$-complete étale $\AffSpc{1}$-invariant sheaves 
    of anima over a qcqs finite-dimensional base scheme $S$ of bounded étale cohomological dimension with $p$ invertible on $S$.
    This generalizes results of 
    Suslin--Voevodsky \cite{suslin1996singular}, Ayoub \cite{ayoub2014realisation}, Cisinski--Déglise \cite{Cisinski_2015}, and Bachmann \cite{Bachmann2021etalerigidity,bachmann2021remarksetalemotivicstable}
    to the unstable setting. Over a perfect field we exhibit a large class of sheaves to which our main theorem applies,
    in particular the $p$-completion of the étale sheafification of any 
    $2$-effective $2$-connective motivic space,
    as well as the $p$-completion of any $4$-connective $\AffSpc{1}$-invariant étale sheaf.
    We use this rigidity result to prove (a weaker version of) an étale analog of Morel's 
    theorem stating that for a Nisnevich sheaf of abelian groups, strong $\AffSpc{1}$-invariance implies strict $\AffSpc{1}$-invariance.
    Moreover, this allows us to construct an unstable étale realization functor on $2$-effective $2$-connective motivic spaces.
\end{abstract}

\hypersetup{pdfborder=0 0 0}
\tableofcontents
\hypersetup{pdfborder=1 1 1}
\newpage

\input{intro}

\section{Preliminaries}         \label{sec:prelim}
\input{highly-connected}        \label{sec:highly-connected}
\input{p-comp}                  \label{sec:p-comp}
\input{effective}               \label{sec:effective}
\input{etale-topos}             \label{sec:etale-topos}
\input{canonical}               \label{sec:canonical}

\input{etale-motivic-htpy}      \label{sec:etale-A1-inv}
\input{rational}                \label{sec:rational}
\input{at-char}                 \label{sec:at-char}
\input{rigidity}                \label{sec:rigidity}
\input{retract}                 \label{sec:retract}
\input{examples}                \label{sec:examples}
\input{application-strict}      \label{sec:application:strict}

\appendix
\input{nilpotent-morphisms}     \label{sec:nilpotent-morphisms}

\bibliographystyle{alpha}
\bibliography{bibliography}

\end{document}

%% file: preamble.tex
\usepackage[utf8]{inputenc}
\usepackage[T1]{fontenc}
\usepackage[english]{babel}

\usepackage{tocloft}

\usepackage{amsmath, amssymb, amsfonts}
\usepackage{mathtools}

\usepackage{comment}

\usepackage{tikz-cd}
\usepackage{quiver} 
    
\usepackage{enumitem}

\usepackage{amsthm}
\usepackage{thmtools}
\usepackage{hyperref}
\usepackage[capitalize,noabbrev,english]{cleveref}

\cftsetindents{section}{0em}{2em}
\cftsetindents{subsection}{0em}{2em}
\theoremstyle{definition}
\newtheorem{defn}{Definition}[section]
\newtheorem{constr}[defn]{Construction}
\newtheorem*{defn*}{Definition}

\theoremstyle{plain}
\newtheorem{lem}[defn]{Lemma}
\newtheorem{cor}[defn]{Corollary}
\newtheorem{prop}[defn]{Proposition}
\newtheorem{thm}[defn]{Theorem}

\newtheorem*{thm*}{Theorem}
\newtheorem*{lem*}{Lemma}

\theoremstyle{plain}
\newtheorem{introthm}{Theorem}

\newtheorem{introcor}[introthm]{Corollary}
\newtheorem{introquestion}[defn]{Question}

\theoremstyle{remark}
\newtheorem{rmk}[defn]{Remark}

\crefname{lem}{Lemma}{Lemmas}
\crefname{defn}{Definition}{Definitions}
\crefname{cor}{Corollary}{Corollaries}
\crefname{introcor}{Corollary}{Corollaries}
\crefname{prop}{Proposition}{Propositions}
\crefname{thm}{Theorem}{Theorems}
\crefname{introthm}{Theorem}{Theorems}
\crefname{rmk}{Remark}{Remarks}
\crefname{section}{Section}{Sections}
\crefname{equation}{Equivalence}{Equivalences}

\newcommand{\Z}{{\mathbb{Z}}}
\newcommand{\N}{{\mathbb{N}}}
\newcommand{\Q}{{\mathbb{Q}}}

\renewcommand{\S}{{\mathbb S}}
\newcommand{\finfld}[1]{\mathbb{F}_{#1}}

\newcommand{\colim}[1]{{\underset{\substack{#1}}{\operatorname{colim}}\,}}
\newcommand{\colimil}[1]{{{\operatorname{colim}}_{#1}\,}}
\renewcommand{\lim}[1]{{\underset{\substack{#1}}{\operatorname{lim}}\,}}
\newcommand{\limil}[1]{{{\operatorname{lim}}_{#1}\,}}

\newcommand{\Map}[3]{{\operatorname{Map}_{#1}({#2},{#3})}}
\newcommand{\Spc}[1]{\operatorname{Spc}(#1)}
\newcommand{\SpcEt}[1]{\operatorname{Spc}_{\et}(#1)}
\newcommand{\Sp}{{\operatorname{Sp}}}
\newcommand{\SH}[1]{\mathcal{SH}(#1)}
\newcommand{\SHone}[1]{\mathcal{SH}^{S^1}(#1)}
\newcommand{\SHet}[1]{\mathcal{SH}_{\et}(#1)}
\newcommand{\SHoneet}[1]{\mathcal{SH}^{S^1}_{\et}(#1)}

\newcommand{\op}[1]{#1^{\operatorname{op}}}

\newcommand{\Sus}{{\Sigma^\infty}}
\newcommand{\SusP}{{\Sigma^\infty_{\Pp^1}}}
\newcommand{\pSus}{{\Sigma^\infty_+}}
\newcommand{\Loop}{{\Omega^\infty}}
\newcommand{\pLoop}{{\Omega_*^\infty}}
\newcommand{\pLoopP}{{\Omega_{\Pp^1,*}^\infty}}

\newcommand{\Fib}[1]{{\operatorname{fib}\mathopen{}\left(#1\right)\mathclose{}}}

\newcommand{\topos}[1]{{\mathcal {#1}}}
\newcommand{\Cat}[1]{{\mathcal {#1}}}
\newcommand{\spectra}[1]{{\operatorname{Sp}(\topos {#1})}}

\newcommand{\Stab}[1]{{\operatorname{Sp}(#1)}}

\newcommand{\An}{{\mathcal An}}

\newcommand{\Ab}{{\mathcal{A}b}}
\newcommand{\AbObj}[1]{{\mathcal{A}b(#1)}}

\newcommand{\Disc}[1]{{\operatorname{Disc}(#1)}}

\newcommand{\ShvTop}[2]{\operatorname{Shv}_{#1}({#2})}
\newcommand{\ShvTopH}[2]{\operatorname{Shv}^{\operatorname{h}}_{#1}({#2})}

\newcommand{\id}[1]{\operatorname{id}_{#1}}

\newcommand{\Mod}[2]{\operatorname{Mod}_{#1}(#2)}
\newcommand{\fgt}[1]{\operatorname{fgt}_{#1}}

\newcommand{\loccitdot}{\textit{loc.~cit..}}

\newcommand{\et}{\operatorname{\acute{e}t}}
\newcommand{\nis}{\operatorname{nis}}
\newcommand{\smet}[1]{{#1}_{\et}}
\newcommand{\Sm}[1]{\operatorname{Sm}_{#1}}

\newcommand{\nilop}{\operatorname{nil}}

\newcommand{\complete}[1]{{#1}\phantom{}^{\wedge}_{p}}

\newcommand{\compnil}[1]{{#1}\phantom{}^{\wedge}_{p,\nilop}}

\newcommand{\AffSpc}[1]{\mathbb{A}^{#1}}
\newcommand{\Pp}{\mathbb{P}}

\newcommand{\Gm}{\mathbb{G}_m}
\newcommand{\Ga}{\mathbb{G}_a}

\newcommand{\sh}{\operatorname{sh}}

\DeclareMathOperator{\Spec}{Spec}

\newcommand{\sslash}{/\mkern-6mu/}

\newcommand{\cartsymb}{\arrow[dr, phantom,"\scalebox{1.5}{\color{black}$\lrcorner$}", near start, color=black]}

\newcommand{\cd}[2]{\operatorname{cd}_{#1}(#2)}

\newcommand{\shens}[2]{{#1}_{\overline{#2}}^{\operatorname{sh}}}

\newcommand\noloc{%
  \nobreak
  \mspace{6mu plus 1mu}
  {:}
  \nonscript\mkern-\thinmuskip
  \mathpunct{}
  \mspace{2mu}
}

\author{Klaus Mattis\footnote{\href{www.klaus-mattis.com}{www.klaus-mattis.com}}}
\date{\today}

%% file: intro.tex
\section{Introduction}
It is well-known that étale cohomology with locally constant $\finfld{p}$-coefficients 
is $\AffSpc{1}$-invariant, cf.\ \cite[Corollaire XV.2.2]{SGA4}. 
One can ask whether the converse is true, resulting in rigidity results 
of Suslin--Voevodsky \cite{suslin1996singular}, Ayoub \cite{ayoub2014realisation}
or Cisinski--Déglise \cite{Cisinski_2015}.
The most general of such results is the following version for spectral coefficients,
due to Bachmann.
\begin{thm*}[{\cite{Bachmann2021etalerigidity,bachmann2021remarksetalemotivicstable}}]
    Let $S$ be a scheme, and $p$ a prime invertible on $S$.
    There are canonical equivalences
    \begin{equation*}
        \complete{\ShvTopH{\et}{\smet{S}, \Sp}} \xrightarrow{\simeq} \complete{\SHoneet{S}} \xrightarrow{\simeq} \complete{\SHet{S}}
    \end{equation*}
    between $p$-complete hypersheaves of spectra on the small étale site of $S$,
    $p$-complete $\AffSpc{1}$-invariant étale hypersheaves of spectra on $\Sm{S}$,
    and its ($p$-complete) $\Gm$-stabilization.
\end{thm*}

The main goal of this article is a generalization of Bachmann's result to the unstable setting.
We write $\cd{}{k}$ for the étale cohomological dimension of a field $k$.
Our main theorem is the following.
\begin{introthm} [{\cref{lem:etale-nilpotent:4-connective-main-thm}}] \label{lem:intro:main-thm-cor}
    Let $k$ be a perfect field with $\cd{}{k} < \infty$, and $p \neq \operatorname{char}(k)$ be a prime.
    There is a canonical equivalence
    \begin{equation*}
        L_p \iota_! \colon \complete{\left(\ShvTopH{\et}{\smet{k}}_{*,\ge 4}\right)} \xrightarrow{\simeq} \complete{\left(\ShvTopH{\et}{\Sm{k}}_{*,\AffSpc{1},\ge 4}\right)}.
    \end{equation*}
    Here, $\complete{(\ShvTopH{\et}{\smet{k}}_{*,\ge 4})}$ denotes the full subcategory 
    of $\ShvTopH{\et}{\smet{k}}_{*}$ (the $\infty$-topos of étale hypersheaves on $\smet{k}$) consisting of those sheaves that are the $p$-completion 
    of a $4$-connective sheaf, and similarly $\complete{(\ShvTopH{\et}{\Sm{k}}_{*,\AffSpc{1},\ge 4})}$ 
    denotes the full subcategory of $\ShvTopH{\et}{\Sm{k}}_{*}$ consisting of those sheaves that are the $p$-completion 
    of an $\AffSpc{1}$-invariant $4$-connective sheaf.
\end{introthm}

\begin{rmk}
    The above result is likely true for integers $2 \le m < 4$, but our methods 
    are not strong enough to show this.
    Note however that the functor 
    \begin{equation*}
        L_p \iota_! \colon \complete{\left(\ShvTopH{\et}{\smet{k}}_*\right)} \to \complete{\left(\ShvTopH{\et}{\Sm{k}}_{*,\AffSpc{1}}\right)}
    \end{equation*}
    is \emph{not} an equivalence. Indeed, any sheaf of sets is automatically $p$-complete (cf.\ \cite[Lemma 3.13]{mattis2024unstable}),
    but there are $\AffSpc{1}$-invariant étale sheaves of sets that are not coming from the small étale site,
    e.g.\ $\Gm$.
\end{rmk}

We will deduce this theorem from a related version that holds over any base scheme 
with some bound on the étale cohomological dimension of its residue fields.
To state the result, we need the following definition:
\begin{defn*}[\cref{def:etale-nilpotent:etale-nilpotent}]
    Let $S$ be a scheme, and $X \in \ShvTopH{\et}{\Sm{S}}$.
    We say that $X$ is $p$-completely étale $\AffSpc{1}$-nilpotent
    if there exists a highly connected tower $(X_n)_n$ under $X$ (i.e.\ the connectivity of the transition maps $X_{n+1} \to X_n$ goes to infinity as $n \to \infty$) 
    such that the following holds:
    \begin{enumerate}
        \item $X_0 \cong *$.
        \item $X \cong \limil{n} X_n$.
        \item Each of the morphisms $X_{n+1} \to X_n$ is part of a fiber sequence $X_{n+1} \to X_n \to K_n$
            where $K_n \cong \pLoop E_n$ is a connected infinite loop sheaf
            for some $E_n \in \Stab{\ShvTopH{\et}{\Sm{S}}}_{\ge 1}$.
        \item For every $n$ the object $L_p E_n = \limil{k} E_n \sslash p^k$ is $\AffSpc{1}$-invariant.
    \end{enumerate}
    In other words, an sheaf $X$ is $p$-completely étale $\AffSpc{1}$-nilpotent
    if it admits a principal refinement of (a version of) the Postnikov-tower of $X$,
    such that the layers are $\AffSpc{1}$-invariant after $p$-completion.
\end{defn*}

\begin{introthm}[{\cref{lem:rigidity:main-thm,lem:etale-nilpotent:nilpotent-A1-inv}}] \label{lem:intro:main-thm}
    Let $S$ be a qcqs scheme of finite Krull-dimension with $\sup_{s \in S} \operatorname{cd}(s) < \infty$,
    and let $p$ be a prime invertible on $S$.
    The functor 
    \begin{equation*}
        \iota_!^p \colon \compnil{\ShvTopH{\et}{\smet{S}}} \to \complete{\ShvTopH{\et}{\Sm{S}}}
    \end{equation*}
    is fully faithful,
    with essential image exactly the $p$-completions of $p$-completely étale $\AffSpc{1}$-nilpotent sheaves.
    In particular, every sheaf in the essential image is $\AffSpc{1}$-invariant.
\end{introthm}
We will deduce \cref{lem:intro:main-thm-cor} from \cref{lem:intro:main-thm} 
by doing a careful analysis of which étale sheaves are actually $p$-completely étale $\AffSpc{1}$-nilpotent,
which we will now explain.
\subsection*{Étale \texorpdfstring{$\AffSpc{1}$}{A1}-nilpotent sheaves over a perfect field}
Over a perfect field $k$, by a result of Asok--Fasel--Hopkins,
in the Nisnevich local world there is no difference between Nisnevich nilpotence and Nisnevich $\AffSpc{1}$-nilpotence:
\begin{lem*}[{\cite[Proposition 3.2.3]{asok2022localization} and \cite[Lemma 5.19]{mattis2024unstable}}]
    Let $k$ be a perfect field and $X \in \ShvTop{\nis}{\Sm{k}}_*$ be connected and $\AffSpc{1}$-invariant.
    Then $X$ is nilpotent (in the topos theoretic sense,
    i.e.\ there exists a principal refinement of the Postnikov tower of $X$) 
    if and only if $X$ is $\AffSpc{1}$-nilpotent (i.e.\ there exists a 
    principal refinement of the Postnikov tower of $X$ with $\AffSpc{1}$-invariant layers).
\end{lem*}
This naturally leads to the following analogous question:
\begin{introquestion} \label{question:intro:etale-nilpotent}
    Let $k$ be a perfect field, $p \neq \operatorname{char}(k)$ be a prime,
    and $X \in \ShvTopH{\et}{\Sm{k}}_*$ be connected $\AffSpc{1}$-invariant.
    If $X$ is étale nilpotent, is it true that $X$ is ($p$-completely) étale $\AffSpc{1}$-nilpotent?
\end{introquestion}
We are able to give a partial answer to this question in form of the following result, which is interesting on its own:
\begin{introthm}[{\cref{lem:etale-nilpotent:2-effective-nilpotent}}] \label{lem:intro:2-effective}
    Let $k$ be a perfect field of exponential characteristic $e$ with $\cd{}{k} < \infty$, $p \neq e$ a prime,
    and $X \in \Spc{k}_*$ be a motivic space that is $2$-effective, nilpotent (as a Nisnevich sheaf)
    and $\Z[\frac{1}{e}]$-local.
    Then $L_{\et} X$ is étale $\AffSpc{1}$-nilpotent (so in particular $\AffSpc{1}$-invariant
    and $p$-completely étale $\AffSpc{1}$-nilpotent).
\end{introthm}
Note that this is a priori surprising, as, in general,
the étale sheafification of an $\AffSpc{1}$-invariant presheaf is no longer $\AffSpc{1}$-invariant,
and simultaneous étale sheafification and $\AffSpc{1}$-localization 
is computed by the countable colimit of the alternating application of the two localization functors.
The above theorem says that if we start with a nilpotent and $2$-effective motivic space,
then this procedure is unnecessary, as the étale sheafification is already $\AffSpc{1}$-invariant.
The proof makes heavy use of the theory of $\Pp^1$-Postnikov towers as developed in \cite{asok2023freudenthal}.

Using this, we can give more examples of étale sheaves for which rigidity holds.
Bachmann showed that in $p$-complete étale stable motivic homotopy theory over an algebraically closed field,
there is an equivalence $L_p \Sus \Gm \cong L_p \S^1$, cf.\ \cite[Theorem 6.5]{Bachmann2021etalerigidity} and \cite[proof of Theorem 3.1]{bachmann2021remarksetalemotivicstable}
(note that if the field is algebraically closed, the twisting spectrum is trivialized, yielding the above claim).
A similar statement is true unstably:
\begin{introthm}[{\cref{lem:retract:main-thm,lem:retract:double-suspension}}] \label{lem:intro:retract}
    Let $k$ be an algebraically closed field with $p \neq \operatorname{char}(k)$.
    Then there is a retract diagram (in $\ShvTopH{\et}{\Sm{k}}$)
    \begin{equation*}
        L_p S^2 \xrightarrow{\tau} L_p L_{\et, \AffSpc{1}} \Pp^1 \xrightarrow{\sigma} L_p S^2.
    \end{equation*}
    Moreover, after (the $p$-completion of) a twofold suspension,
    $\tau$ becomes an equivalence,
    i.e.\ we have 
    \begin{equation*}
        L_p S^4 \cong L_p L_{\et, \AffSpc{1}} \Sigma^2 \Pp^1.
    \end{equation*}
\end{introthm}
Using this retract, we can prove that in fact every $4$-connective $\AffSpc{1}$-invariant étale sheaf over
an algebraically closed field is ($p$-completely) a retract 
of a $2$-effective and $2$-connective étale motivic space.
Using étale descent, this result suggests the following corollary.
\begin{introcor} [{\cref{lem:etale-nilpotent:4-connective-nilpotent}}] \label{lem:intro:4-connective}
    Let $k$ be a perfect field with $\cd{}{k} < \infty$ and $p \neq \operatorname{char}(k)$ a prime,
    and $X \in \ShvTopH{\et}{\Sm{k}}_*$ be $4$-connective and 
    $\AffSpc{1}$-invariant.
    Then $X$ is $p$-completely small, i.e.\ $L_p \iota_! \iota^* X \cong L_p X$.
\end{introcor}

\subsection*{Étale motivic spaces at the characteristic}
If $k$ is a field of characteristic $p > 0$, then the Artin--Schreier sequence $\finfld{p} \to \AffSpc{1} \xrightarrow{1 - F} \AffSpc{1}$
shows that the $\AffSpc{1}$-localization of the constant sheaf of abelian groups $\finfld{p}$ is $0$.
From this, one can for example deduce the vanishing of the whole category of motives with ${p}$-torsion coefficients,
cf.\ \cite[Proposition A.3.1]{Cisinski_2015}.
The strongest such vanishing result is due to Bachmann and Hoyois:
\begin{thm*}[{\cite[Theorem A.1]{bachmann2021remarksetalemotivicstable}}]
    $\complete{\SHet{S}} \cong \complete{\SHoneet{S}} \cong 0$ for $S$ any scheme over $\finfld{p}$.
\end{thm*}
Unstably, we can show the following version:
\begin{introthm}[{\cref{lem:at-char:p-completely-null}}] \label{lem:intro:vanishing-at-char-easy}
    Let $S$ be a qcqs $\finfld{p}$-scheme of finite Krull-dimension with $\sup_{s \in S} \operatorname{cd}(s) < \infty$,
    and $X \in \ShvTopH{\et}{\Sm{S}}$ 
    be $p$-completely étale $\AffSpc{1}$-nilpotent.
    Then $L_p X \cong *$.
\end{introthm}
Related is the following result:
\begin{introthm}[{\cref{lem:at-char:EM-vanishing}}] \label{lem:intro:vanishing-at-char-hard}
    Let $k$ be a perfect field of characteristic $p > 0$ with $\cd{}{k} < \infty$.
    Let $A \in \ShvTop{\et}{\Sm{k}, \Ab}$ be a sheaf of abelian groups, and $k \ge 3$.
    If $L_p K(A, k)$ is $\AffSpc{1}$-invariant, then $L_p K(A, k) \cong *$.
\end{introthm}
\begin{rmk}
    It is not true that unstably the whole $p$-complete category vanishes:
    Indeed, étale sheaves of sets are always $p$-complete by \cite[Lemma 3.13]{mattis2024unstable},
    but there are $\AffSpc{1}$-invariant sheaves of sets that are nontrivial, e.g.\ $\Gm$.
\end{rmk}

\subsection*{Étale hyperdescent for rational motivic spaces} 
Let $k$ be a perfect field. Recall from e.g.\ \cite[§16]{cisinski2019triangulated}
that there is a splitting $\SH{k}_\Q \cong \SH{k}_\Q^+ \times \SH{k}_\Q^-$.
Every object in $\SH{k}_\Q^+$ satisfies étale hyperdescent,
whereas if the cohomological dimension of $k$ is finite, $\SH{k}_\Q^- = 0$.
In particular, one obtains the following result:
\begin{thm*}[{\cite[§16]{cisinski2019triangulated}}]
    Let $k$ be a perfect field with $\cd{}{k} < \infty$.
    Then any rational motivic spectrum $E \in \SH{k}_\Q$ 
    satisfies étale hyperdescent.
\end{thm*}
Using similar techniques as in the proof of \cref{lem:intro:2-effective}, 
we can prove an unstable version of this result.
\begin{introthm}[\cref{lem:rational:main-thm}] \label{lem:intro:rational}
    Let $k$ be a perfect field with $\cd{}{k} < \infty$.
    Any rational nilpotent and $2$-effective motivic space $X \in \Spc{k}_*$ satisfies étale hyperdescent. 
\end{introthm}

\subsection*{Strict \texorpdfstring{$\AffSpc{1}$}{A1}-invariance for étale sheaves of abelian groups}
Let $k$ be a perfect field.
Recall that a Nisnevich sheaf of abelian groups $A$ is $1$-strictly 
(resp.\ strictly) Nisnevich $\AffSpc{1}$-invariant
if $K_{\nis}(A, 1) \in \ShvTop{\nis}{\Sm{k}}$ is $\AffSpc{1}$-invariant 
(resp.\ $K_{\nis}(A, n) \in \ShvTop{\nis}{\Sm{k}}$ is $\AffSpc{1}$-invariant for all $n \ge 0$).
Morel showed in \cite[Theorem 4.46]{morel2012a1} that
1-strictly Nisnevich $\AffSpc{1}$-invariant sheaves of abelian groups are 
already strictly Nisnevich $\AffSpc{1}$-invariant.
This naturally leads to the analogous question for étale sheaves.
Let us say that an étale sheaf of abelian groups $A$ on $\Sm{k}$
is \emph{$m$-strictly étale $\AffSpc{1}$-invariant} for $m \ge 1$
if $K(A, m) \in \ShvTopH{\et}{\Sm{k}}$ is $\AffSpc{1}$-invariant, and that it is \emph{strictly 
étale $\AffSpc{1}$-invariant} if $K(A, n)$ is $\AffSpc{1}$-invariant for all $n \ge 0$.
\begin{introquestion}
    Let $k$ be a perfect field with $\cd{}{k} < \infty$ and $A$
    be an étale sheaf of abelian groups on $\Sm{k}$.
    Suppose that $A$ is $1$-strictly étale $\AffSpc{1}$-invariant.
    Is it true that $A$ is strictly étale $\AffSpc{1}$-invariant?
\end{introquestion}

Using the rigidity result, we can partially answer this question.
\begin{introthm}[{\cref{lem:strictly:main-thm}}] \label{lem:intro:strictly}
    Let $k$ be a perfect field with $\cd{}{k} < \infty$ and $A \in \ShvTopH{\et}{\Sm{k}, \Ab}$
    be an étale sheaf of abelian groups.
    If $A$ is $4$-strictly étale $\AffSpc{1}$-invariant,
    then $A$ is strictly étale $\AffSpc{1}$-invariant.
\end{introthm}
In other words, this is a weak version of Morel's theorem in the étale world.

\begin{rmk}
    It is unknown to the author if for $m \in \{1, 2, 3\}$
    we still have the implication that $m$-strictly étale $\AffSpc{1}$-invariant sheaves 
    are already strictly étale $\AffSpc{1}$-invariant.
    This is closely related to whether \cref{lem:intro:main-thm-cor} holds for $m < 4$.
\end{rmk}
Using the equivalence $H^i_{\et}(U, A) \cong \pi_{n - i}(K(A, n)(U))$
for every smooth $k$-scheme $U$, integers $n \ge i \ge 0$ and étale sheaf 
of abelian groups $A$, we can reformulate the above theorem as the following.
\begin{introcor} \label{lem:intro:strictly-cor}
    Let $k$ be a perfect field with $\cd{}{k} < \infty$ and $A \in \ShvTopH{\et}{\Sm{S}, \Ab}$
    be an étale sheaf of abelian groups.
    If $H^{i}_{\et}(-, A)$ is $\AffSpc{1}$-invariant for every $i \in \{ 0, \dots, 4 \}$,
    then $H^{i}_{\et}(-, A)$ is $\AffSpc{1}$-invariant for every $i \ge 0$.
\end{introcor}

\subsection*{Étale realization functor}
As another application of this theory, we obtain an unstable étale realization functor:

\begin{introthm} \label{lem:intro:realization-functor}
    Let $k$ be a perfect field with $\cd{}{k} < \infty$, and $p \neq \operatorname{char}(k)$ a prime.
    There exists an \emph{étale realization functor} 
    \begin{equation*}
        Re_p \colon \Spc{k}_{*,\mathrm{nil},2-\mathrm{eff}} \to \complete{\ShvTopH{\et}{\smet{k}}_{*,}}.
    \end{equation*}
    Here, the left-hand side is the $\infty$-category of nilpotent and $2$-effective pointed motivic spaces.
\end{introthm}
\begin{proof}
    Write $e$ for the exponential characteristic of $k$.
    Let $X \in \Spc{k}_{*,\mathrm{nil},2-\mathrm{eff}}$.
    Then also $L_{\Z[\frac{1}{e}]} X$ is nilpotent and $2$-effective:
    Indeed, nilpotence was shown in \cite[Proposition 4.3.8]{asok2022localization}
    (they only show that $L_{\Z[\frac{1}{e}]} X$ is weakly $\AffSpc{1}$-nilpotent, 
    but under the assumption that $X$ is nilpotent their proof in fact shows that $L_{\Z[\frac{1}{e}]} X$ is nilpotent).
    That it is $2$-effective was shown in \cite[Proposition 4.3.4]{asok2023freudenthal}.
    Hence, $Y \coloneqq L_{\et} L_{\Z[\frac{1}{e}]} X$ is étale $\AffSpc{1}$-nilpotent by \cref{lem:intro:2-effective},
    so in particular $p$-completely étale $\AffSpc{1}$-nilpotent by \cref{lem:etale-nilpotent:etale-nilpotent-implies-p-comp-rationally}.
    We can therefore define $Re_p X$ to be the inverse of the equivalence $\iota_!^p$ from \cref{lem:intro:main-thm},
    applied to $L_p Y$.
\end{proof}

\subsection*{Improving the bounds}
Let $k$ be a perfect field and $\epsilon \ge 0$ be an integer.
Recall that $\SHone{k}(\epsilon)$, the category of $\epsilon$-effective 
motivic $S^1$-spectra, and $\SH{k}(\epsilon)$, the category of $\epsilon$-effective motivic spectra 
both possess a t-structure \cite[§6.1]{bachmannyakerson}.
In particular, by \cite[Lemma 6.2]{bachmannyakerson}, there is an induced functor on the hearts 
\begin{equation*}
    \omega^\infty \colon \SH{k}(\epsilon)^\heartsuit \to \SHone{k}(\epsilon)^\heartsuit.
\end{equation*}
By \cite{bachmannyakerson,bachmann2021zerothP1,feld2021milnor}, this functor is an equivalence for $\epsilon \ge 2$
(see also \cite[Theorem 2.2.30 (2)]{asok2023freudenthal} for a combined proof of this fact).
This equivalence lets us resolve any $\epsilon$-effective motivic space 
by a tower, where the layers are $\Pp^1$-infinite loop spaces, cf.\ \cref{lem:etale-nilpotent:effective-Postnikov-refinement}.
All the bounds presented in the above introduction are derived from the fact that this works for $\epsilon = 2$.
If it turns out that also for $\epsilon = 1$ the above functor is an equivalence, one can improve the above results as follows:
\begin{itemize}
    \item In \cref{lem:intro:main-thm-cor}, one can replace $4$-connective by $3$-connective.
    \item In \cref{lem:intro:2-effective}, one can replace $2$-effective by $1$-effective.
    \item In \cref{lem:intro:4-connective}, one can replace $4$-connective by $3$-connective.
    \item In \cref{lem:intro:rational}, one can replace $2$-effective by $1$-effective.
    \item In \cref{lem:intro:strictly}, one can replace $4$-strictly by $3$-strictly.
    \item In \cref{lem:intro:strictly-cor}, one can replace $\{ 0, \dots, 4 \}$ by $\{ 0, \dots, 3 \}$. 
    \item In \cref{lem:intro:realization-functor}, one can replace $2$-effective by $1$-effective.
\end{itemize}

\subsection*{Linear Outline}
We begin in \cref{sec:prelim} with background material.

\Cref{sec:highly-connected} establishes a criterion for when a geometric morphism of $\infty$-topoi preserves limits along certain highly connected towers.

In \cref{sec:p-comp}, we strengthen some results on unstable $p$-completion from \cite{mattis2024unstable,mattis2024fracture}.

In \cref{sec:effective}, we collect key results on $\Pp^1$-Postnikov towers from the work of Asok–Bachmann–Hopkins \cite{asok2023freudenthal}, which serve as an important tool in our later arguments.

\Cref{sec:etale-topos} collects key facts about hypercomplete étale $\infty$-topoi.

Finally, in \cref{sec:canonical}, we show that any $n$-connective pointed sheaf in an $\infty$-topos $\topos X$ has a canonical resolution by ``free $n$-connective pointed sheaves'',
i.e.\ sheaves of the form $S^n \wedge T$ for $T \in \topos X_*$.

\Cref{sec:etale-A1-inv} examines the connection between étale sheaves and $\AffSpc{1}$-invariant presheaves. In particular, we introduce the notion of ($p$-completely) étale $\AffSpc{1}$-nilpotent sheaves, which play a central role in the rigidity theorem.

\Cref{sec:rational} shows that, after rationalization, étale and Nisnevich motivic spaces become closely related, culminating in \cref{lem:intro:rational}.

\cref{sec:at-char} provides some vanishing results of $p$-complete étale motivic spaces, if $p$ is the characteristic of the base scheme, 
proving \cref{lem:intro:vanishing-at-char-easy,lem:intro:vanishing-at-char-hard}.

The unstable rigidity theorem is proven in \cref{sec:rigidity}, where we establish \cref{lem:intro:main-thm}.

In \cref{sec:retract}, we construct a retract $S^2 \to \Gm \to S^2$ in the $p$-complete étale unstable motivic setting, yielding \cref{lem:intro:retract}.

\Cref{sec:examples} presents concrete examples of étale $\AffSpc{1}$-nilpotent sheaves, thereby verifying \cref{lem:intro:2-effective,lem:intro:4-connective,lem:intro:main-thm-cor}.

Finally, \cref{sec:application:strict} proves \cref{lem:intro:strictly}, an étale analog of Morel’s theorem on $1$-strictly and strictly $\AffSpc{1}$-invariant sheaves of abelian groups.

Furthermore, in Appendix \ref{sec:nilpotent-morphisms}, we collect some results 
about Moore--Postnikov towers for nilpotent morphisms in the setting of $\infty$-topoi.

\subsection*{Notation}
This article is written in the language of $\infty$-categories, as developed e.g.\ by 
Lurie in \cite{highertopoi}.
We will use the word \emph{anima} for an object of the $\infty$-category of anima/spaces/homotopy types,
whereas we will use the word \emph{sheaf} for an object of an $\infty$-topos $\topos X$.
We will furthermore employ the following notation:
\begin{center}
    \begin{tabular}{l|ll}
        $\An$ &$\infty$-category of anima \\
        $\Sp$ &$\infty$-category of spectra &\cite[§1.4.3]{higheralgebra}\\
        $\Stab{-}$ &Stabilization of an $\infty$-category &\cite[§1.4.2]{higheralgebra}\\
        $\ShvTop{\tau}{\Cat C}$ &$\infty$-topos of sheaves on a site $(\Cat C, \tau)$ \\
        $\ShvTopH{\tau}{\Cat C}$ &$\infty$-topos of hypersheaves on a site $(\Cat C, \tau)$, \\ &i.e.\ the hypercompletion of $\ShvTop{\tau}{\Cat C}$ \\
        $\Ab$ &Category of abelian groups \\
        $L_\Q$ &(Un)stable rationalization functor &\cite[§§2, 3]{mattis2024fracture} \\
        $L_p$ &(Un)stable $p$-completion functor &\cite[§§2, 3]{mattis2024unstable} \\
        $\complete{\Cat C}$ &$p$-complete objects in an $\infty$-category $\Cat C$, &\cite[§§2, 3]{mattis2024unstable} \\ &i.e.\ the essential image of $L_p$ \\
        $\Sm{S}$ &Category of quasicompact smooth $S$-schemes \\
        $\smet{S}$ &Category of quasicompact étale $S$-schemes \\
        $\Spc{S}$ &$\infty$-category of motivic spaces &\cite[§2.2]{bachmann2020norms} \\
        $\SHone{S}$ &$\infty$-category of motivic $S^1$-spectra &\cite[§4]{morel2003trieste} \\
        $\SH{S}$ &$\infty$-category of motivic spectra &\cite[§4.1]{bachmann2020norms} \\
        $\SHoneet{S}$ &$\infty$-category of étale motivic $S^1$-spectra &\cite[§5]{Bachmann2021etalerigidity} \\
        $\SHet{S}$ &$\infty$-category of étale motivic spectra &\cite[§5]{Bachmann2021etalerigidity} \\
        $\mathbf{K}^{\mathrm{MW}}_n$ &$n$-th Milnor--Witt K-theory group &\cite[§2]{morel2012a1}\\
    \end{tabular}
\end{center}
We will use without mention that there is a canonical equivalence 
\begin{equation*}
    \Stab{\ShvTop{\tau}{\Cat C}} \cong \ShvTop{\tau}{\Cat C, \Sp}
\end{equation*}
between the stabilization of the $\infty$-category of sheaves on $(\Cat C, \tau)$
and the $\infty$-category of sheaves of spectra on $(\Cat C, \tau)$,
cf.\ \cite[Remark 1.3.2.2 and Proposition 1.3.1.7]{sag},
and similarly in the hypercomplete case.
In particular, an object $E \in \Stab{\ShvTop{\tau}{\Cat C}}$ will be called 
a \emph{sheaf of spectra}.

Recall that there are by construction adjunctions 
\begin{equation*}
    \Sus \colon \Spc{S}_* \rightleftarrows \SHone{S} \noloc \pLoop
\end{equation*}
and 
\begin{equation*}
    \sigma^\infty \colon \SHone{S} \rightleftarrows \SH{S} \noloc \omega^\infty.
\end{equation*}
We also write $\SusP \coloneqq \sigma^\infty \Sus$ and $\pLoopP \coloneqq \pLoop \omega^\infty$
for the composed adjunction.

If $k$ is a field, then we define the \emph{exponential characteristic of $k$} as
\begin{equation*}
    e = \begin{cases}
        1 &\text{if } \mathrm{char}(k) = 0 \\
        p &\text{if } \mathrm{char}(k) = p > 0. 
    \end{cases}
\end{equation*}

\subsection*{Acknowledgments}
I thank Tom Bachmann, Elden Elmanto, Lorenzo Mantovani and Swann Tubach for helpful discussions about the topic.
Furthermore, I thank Tom Bachmann, Julie Bannwart and Timo Weiß for helpful comments on a draft of this article.

The author acknowledges support by the Deutsche Forschungsgemeinschaft 
(DFG, German Research Foundation) through the Collaborative Research 
Centre TRR 326 \textit{Geometry and Arithmetic of Uniformized 
Structures}, project number 444845124.

%% file: highly-connected.tex
\subsection{Limits of highly connected towers}
Recall from \cite[Definition 5.1 and Lemma 5.2]{mattis2024fracture} that a locally finite dimensional cover of an $\infty$-topos $\topos X$
consists of a jointly conservative collection of limit-preserving geometric morphisms $\{p_i^* \colon \topos X \to \topos U_i\}_{i \in I}$,
such that moreover $\topos U_i$ has enough points and is locally of homotopy dimension $\le n_i$ for some $n_i \ge 0$ (depending on $i$).
Moreover, any $\infty$-topos that admits such a cover is automatically Postnikov-complete, cf.\ \cite[Lemma 5.3]{mattis2024fracture}.
Recall that a tower $(X_n)_n$ is called \emph{locally highly connected (subordinate to the cover $\{\topos U_i\}_i)$} \cite[Definition 6.1 (3)]{mattis2024fracture}
if $(p_i^* X_n)_n$ is a \emph{highly connected tower} in $\topos U_i$ for all $i \in I$,
i.e.\ all the $p_i^* X_n$ are connected, and for every $i \in I$ and $k \ge 1$ there exists an $N_{i,k} \ge 0$ 
such that for all $n \ge N_{i,k}$ 
the map $\pi_k(p_i^* X_n) \to \pi_k(p_i^* X_{N_{i,k}})$ is an isomorphism.
In this section, we will show that under suitable finiteness conditions, geometric morphisms preserve limits of (locally) highly connected towers.

\begin{prop} \label{lem:highly-connected:main-thm}
    Let $f^* \colon \topos X \to \topos Y$ be a geometric morphism of $\infty$-topoi.
    Suppose that $\topos X$ is locally of homotopy dimension $\le N$ for some $N$
    and that $\topos Y$ admits a locally finite dimensional cover.
    Let $(X_n)_n$ be a highly connected tower in $\topos X$.
    Then the canonical map 
    \begin{equation*}
        f^* \limil{n} X_n \to \limil{n} f^* X_n
    \end{equation*}
    is an equivalence.
\end{prop}
\begin{proof}
    Since $\topos Y$ admits a locally finite dimensional cover,
    there exists a jointly conservative set of points $\Cat S$ of $\topos Y$
    such that each point $s^* \in \Cat S$ factors through an $\infty$-topos 
    which is locally of homotopy dimension $\le N_s$ for some $N_s$,
    cf.\ \cite[Definition 5.1 (3)]{mattis2024fracture}.
    It thus suffices to check the equivalence after applying $s^* \in \Cat S$.
    Note that we have a commutative diagram of the form 
    \begin{center}
        \begin{tikzcd}
            s^* f^* \limil{n} X_n\ar[r] \ar[dr] \ar[dr] &s^* \limil{n} f^* X_n \ar[d] \\
            &\limil{n} s^* f^* X_n
        \end{tikzcd}
    \end{center}
    where all maps are limit-assembly maps.
    Since $s^* f^*$ is a point of $\topos X$,
    the diagonal map is an equivalence by \cite[Lemma 6.5]{mattis2024fracture}.
    By \cite[Example 6.3]{mattis2024fracture},
    we see that $(f^* X_n)_n$ is highly connected (and thus in particular locally highly connected).
    Hence, since $s^* \in \Cat S$,
    the vertical map is an equivalence by \cite[Corollary 6.6]{mattis2024fracture}.
    Therefore, also the horizontal map is an equivalence, which is what we wanted to show.
\end{proof}

We want a version of the above proposition where we relax the assumption that $\topos X$ is locally of 
homotopy dimension $\le N$. Instead, we only want both $\topos X$ and $\topos Y$ to admit locally finite dimensional covers.
In order to prove such a result one needs a certain compatibility of locally finite dimensional covers 
which is captured by the following straightforward definition.
\begin{defn}
    Let $\topos X$ be an $\infty$-topos with a locally finite dimensional cover $\{ p_i^* \colon \topos X \to \topos U_i \}_{i \in I}$,
    and similarly $\topos Y$ be an $\infty$-topos with a locally finite dimensional cover 
    $\{ q_j^* \colon \topos Y \to \topos V_j \}_{j \in J}$.

    A \emph{morphism of $\infty$-topoi with locally finite dimensional covers}
    consists of a geometric morphism $f^* \colon \topos X \to \topos Y$,
    a map $\kappa \colon J \to I$,
    and geometric morphisms $f_j^* \colon \topos U_{\kappa j} \to \topos V_j$,
    together with the datum of a commutative square 
    \begin{center}
        \begin{tikzcd}
            \topos X \ar[r, "f^*"] \ar[d, "p_{\kappa j}^*"] &\topos Y \ar[d, "q_j^*"]\\ 
            \topos U_{\kappa j} \ar[r, "f_j^*"] &\topos V_j
        \end{tikzcd}
    \end{center}
    for every $j \in J$.
\end{defn}

\begin{prop} \label{lem:highly-connected:morphism-of-covers-preserves-towers}
    Let $\topos X$ be an $\infty$-topos with a locally finite dimensional cover $\{ p_i^* \colon \topos X \to \topos U_i \}_{i \in I}$,
    and similarly $\topos Y$ be an $\infty$-topos with a locally finite dimensional cover 
    $\{ q_j^* \colon \topos Y \to \topos V_j \}_{j \in J}$.
    Let $(f^* \colon \topos X \to \topos Y, \kappa \colon J \to I, (f_j^* \colon \topos U_{\kappa j} \to \topos V_j)_j)$ 
    be a morphism of $\infty$-topoi with locally finite dimensional covers.
    
    If $(X_n)_n$ is a locally highly connected tower in $\topos X$ (subordinate to $\{ \topos U_i \}_i$),
    then the canonical map 
    \begin{equation*}
        f^* \limil{n} X_n \to \limil{n} f^* X_n.
    \end{equation*}
    is an equivalence.
\end{prop}
\begin{proof}
    Since the $q_j^*$ are jointly conservative, it suffices to show that the canonical map 
    $q_j^* f^* \limil{n} X_n \to q_j^* \limil{n} f^* X_n$ is an equivalence for all $j \in J$.
    So fix $j \in J$ and consider the following commutative diagram:
    \begin{center}
        \begin{tikzcd}
            q_j^* f^* \limil{n} X_n \ar[d, "\alpha"]\ar[r, "\simeq"] & f_j^* p_{\kappa j}^* \limil{n} X_n \ar[d, "\gamma"] \\
            q_j^* \limil{n} f^* X_n \ar[d, "\beta"] & f_j^* \limil{n} p_{\kappa j}^* X_n \ar[d, "\delta"]\\
            \limil{n} q_j^* f^* X_n \ar[r, "\simeq"] &\limil{n} f_j^* p_{\kappa j}^* X_n
        \end{tikzcd}
    \end{center}
    The two horizontal arrows are given by the morphism of $\infty$-topoi with finite dimensional covers, and are thus equivalences by definition.
    All the other maps are given by limit-assembly maps.
    The maps $\beta$ and $\gamma$ are equivalences since $p_{\kappa j}^*$ and $q_j^*$ commute with all limits by definition.
    Moreover, the map $\delta$ is an equivalence by \cref{lem:highly-connected:main-thm},
    since $(p_{\kappa j}^* X_n)_n$ is highly connected and $\topos U_{\kappa j}$ is locally 
    of homotopy dimension $\le N$ for some $N$ by assumption.
    Hence, also the map $\alpha$ is an equivalence which is precisely what we wanted to show.
\end{proof}

\begin{lem} \label{lem:highly-connected:tower-limit-htpy-groups}
    Let $\topos X$ be an $\infty$-topos locally of homotopy dimension $\le N$ with enough points,
    and $(X_n)_n$ be a highly connected tower in $\topos X$.
    Write $X \coloneqq \limil{n} X_n$
    Then for all $k \ge 0$ 
    the canonical map $\pi_k(X) \to \pi_k(X_n)$ is an isomorphism for all $n \gg 0$.
\end{lem}
\begin{proof}
    Let $k \ge 0$.
    By assumption, there exists $N$ such that $\pi_k(X_n) \cong \pi_k(X_N)$ for all $n \ge N$.
    We will show that $\pi_k(X) \cong \pi_k(X_n)$ for all $n \ge N$.
    Since there are enough points, and since $s^* \limil{n} X_n \cong \limil{n} s^* X_n$
    for all points $s^*$ by \cite[Lemma 6.5]{mattis2024fracture},
    and $s^* \pi_k (-) \cong \pi_k(s^* -)$ since $s^*$ is a geometric morphism,
    we may assume that $\topos X \cong \An$.
    The result is now an immediate consequence of 
    the computation of homotopy groups of the limit as 
    e.g.\ done in \cite[Proposition 2.2.9]{MCAT},
    as the system $(\pi_{k+1}(X_n))_n$ is Mittag--Leffler (it consists of isomorphisms for $n \gg 0$).
\end{proof}

%% file: p-comp.tex
\subsection{Unstable \texorpdfstring{$p$}{p}-completion}
Let $\topos X$ be a Postnikov-complete $\infty$-topos with enough points.
Write $L_p \colon \topos X \to \topos X$ for the unstable $p$-completion functor,
the Bousfield localization at the class of $p$-equivalences,
cf.\ \cite[Section 3]{mattis2024unstable},
and $\complete{\topos X}$ for its essential image, i.e.\ the subcategory of $p$-complete sheaves.
The goal of this section is to strengthen some of the results about $L_p$ 
from \cite[Section 3]{mattis2024unstable} and \cite[Section 6]{mattis2024fracture}.

\begin{lem} \label{lem:p-comp:loop-detects-peq}
    Let $f \colon E \to F$ be a morphism in $\spectra X$ 
    between $1$-connective sheaves of spectra.
    Then $f$ is a $p$-equivalence if and only if $\pLoop f$ is a $p$-equivalence.
\end{lem}
\begin{proof}
    If $f$ is a $p$-equivalence, then $\pLoop f$ is a $p$-equivalence by \cite[Lemma 3.16]{mattis2024unstable}.
    Suppose on the other hand that $\pLoop f$ is a $p$-equivalence.
    Hence, we see that 
    $L_p \pLoop f \colon L_p \pLoop E \to L_p \pLoop F$ is an equivalence.
    Note that by \cite[Lemma 3.17]{mattis2024unstable} this morphism is equivalent to the morphism 
    \begin{equation*}
        \pLoop \tau_{\ge 1} L_p E \to \pLoop \tau_{\ge 1} L_p F.
    \end{equation*}
    Since $\pLoop$ is conservative on connective objects (this can e.g.\ be checked on stalks,
    where it reduces to the same claim about $\pLoop \colon \Sp_{\ge 0} \to \An_*$,
    which is conservative by \cite[Corollary 5.2.6.27]{higheralgebra}),
    it follows that $\tau_{\ge 1} L_p E \to \tau_{\ge 1} L_p F$ 
    is an equivalence.
    Consider the canonical morphism of fiber sequences 
    \begin{center}
        \begin{tikzcd}
            \tau_{\ge 1} L_p E \ar[r] \ar[d, "\cong"] &L_p E\ar[r] \ar[d, "L_p f"] &\tau_{\le 0} L_p E \ar[d]\\
            \tau_{\ge 1} L_p F \ar[r] &L_p F\ar[r] &\tau_{\le 0} L_p F\rlap{.}
        \end{tikzcd}
    \end{center}
    As $L_p$ is an exact functor, $p$-completing again yields the morphism of fiber sequences 
    \begin{center}
        \begin{tikzcd}
            L_p \tau_{\ge 1} L_p E \ar[r] \ar[d, "\cong"] &L_p E\ar[r] \ar[d, "L_p f"] &L_p \tau_{\le 0} L_p E \ar[d]\\
            L_p \tau_{\ge 1} L_p F \ar[r] &L_p F\ar[r] &L_p \tau_{\le 0} L_p F\rlap{.}
        \end{tikzcd}
    \end{center}
    In order to see that $f \colon E \to F$ is a $p$-equivalence, we have to show that $L_p f$ 
    is an equivalence.
    By the above morphism of fiber sequences, it suffices to show that $L_p \tau_{\le 0} L_p E \to L_p \tau_{\le 0} L_p F$ 
    is an equivalence. We will in fact show that both objects are equivalent to $0$.
    We will show that $L_p \tau_{\le 0} L_p E \cong 0$, the proof for $F$ is the same.
    Since $E$ is $1$-connective, we see that $\pi_n(\tau_{\le 0} L_p E)$ is uniquely $p$-divisible 
    for all $n$ (for $n > 0$ the homotopy object is $0$, and for $n \le 0$ use \cite[Lemma 2.9]{mattis2024unstable} 
    applied to the $p$-equivalence $E \to L_p E$).
    Hence, the claim follows immediately from \cite[Lemma 2.10]{mattis2024unstable},
    as the standard t-structure on $\spectra{X}$ is left-separated (cf.\ e.g.\ the first paragraph of \cite[Section 3.2]{mattis2024unstable},
    using that $\topos X$ is hypercomplete).
    This proves the lemma.
\end{proof}
Next, we strengthen \cite[Proposition 3.19]{mattis2024unstable} to arbitrary 
maps of nilpotent sheaves $X \to Y$ in $\topos X_*$, i.e.\ we remove the assumption that the $p$-completions 
of $X$ and $Y$ are still nilpotent.
For this, we need the following cofinality result:
\begin{lem} \label{lem:fib-lem:cofinal}
    Let $f \colon X \to Y$ be a morphism in $\topos X_*$ of pointed connected sheaves.
    Then the $\N$-indexed inverse systems 
    $(\tau_{\le n} \Fib{X \to Y})_n$ and $(\Fib{\tau_{\le n} X \to \tau_{\le n} Y})_n$
    are cofinal,
    i.e.\ for every $n$ there exists morphisms
    \begin{equation*}
        \alpha_n \colon \Fib{\tau_{\le n} X \to \tau_{\le n} Y} \to \tau_{\le n-1} \Fib{X \to Y}
    \end{equation*}
    and 
    \begin{equation*}
        \beta_n \colon \tau_{\le n} \Fib{X \to Y} \to \Fib{\tau_{\le n} X \to \tau_{\le n} Y}
    \end{equation*}
    such that the compositions $\beta_{n-1} \alpha_n$ and $\alpha_n \beta_n$ 
    are just the transition morphisms on the respective towers.
    In particular, the towers have equivalent limits.
\end{lem}
\begin{proof}
    A long exact sequence argument shows that the canonical map
    \begin{equation*}
        \tau_{\le n-1} \Fib{X \to Y} \to \tau_{\le n-1} \Fib{\tau_{\le n} X \to \tau_{\le n} Y}
    \end{equation*}
    is an equivalence for all $n$.
    Thus, it suffices to show that the two towers 
    $(\tau_{\le n-1} \Fib{\tau_{\le n} X \to \tau_{\le n} Y})_n$ and $(\Fib{\tau_{\le n} X \to \tau_{\le n} Y})_n$
    are cofinal.

    Note first that for every $n$ the unit gives a morphism 
    \begin{equation*}
        \alpha_n \colon \Fib{\tau_{\le n} X \to \tau_{\le n} Y} \to \tau_{\le n-1} \Fib{\tau_{\le n} X \to \tau_{\le n} Y}.
    \end{equation*}
    On the other hand, $\Fib{\tau_{\le n} X \to \tau_{\le n} Y}$ is $n$-truncated (as a limit of $n$-truncated objects),
    and hence the map $\Fib{\tau_{\le n + 1} X \to \tau_{\le n + 1} Y} \to \Fib{\tau_{\le n} X \to \tau_{\le n} Y}$
    induced by the respective units $\tau_{\le n+1}X \to \tau_{\le n}X$ and $\tau_{\le n+1}Y \to \tau_{\le n}Y$ factors over a map 
    \begin{equation*}
        \beta_n \colon \tau_{\le n} \Fib{\tau_{\le n + 1} X \to \tau_{\le n +1 } Y} \to \Fib{\tau_{\le n} X \to \tau_{\le n} Y}.
    \end{equation*}
    It is now easy (but tedious) to see that for every $n$ the compositions $\beta_{n-1} \alpha_n$ and $\alpha_n \beta_n$ 
    are just the transition morphisms on the respective towers.
\end{proof}

\begin{lem} \label{lem:fib-lem:truncated}
    Let $f \colon X \to Y$ be a morphism in $\topos X_*$ of pointed 
    nilpotent sheaves.
    Suppose that $X$ and $Y$ are $n$-truncated for some $n \ge 0$.
    Then 
    \begin{equation*}
        L_p \tau_{\ge 1} \Fib{X \to Y} \cong \tau_{\ge 1} \Fib{L_p X \to L_p Y}.
    \end{equation*}
\end{lem}
\begin{proof}
    Since $X$ and $Y$ are $n$-truncated, it follows from \cite[Proposition 3.20]{mattis2024unstable}
    that also $L_p X$ and $L_p Y$ are nilpotent.
    Thus, the lemma is an immediate consequence of \cite[Proposition 3.19]{mattis2024unstable}.
\end{proof}

\begin{prop} [Bousfield--Kan Fiber Lemma] \label{lem:fib-lem:main-thm}
    Suppose that $\topos X$ has a locally finite dimensional cover.
    Let $f \colon X \to Y$ be a morphism in $\topos X_*$ of pointed nilpotent
    sheaves.
    Then 
    \begin{equation*}
        L_p \tau_{\ge 1} \Fib{X \to Y} \cong \tau_{\ge 1} \Fib{L_p X \to L_p Y}.
    \end{equation*}
\end{prop}
\begin{proof}
    Since $\topos X$ has a locally finite dimensional cover,
    we have equivalences 
    \begin{equation*}
        L_p X \cong \limil{n} L_p \tau_{\le n} X
    \end{equation*}
    and
    \begin{equation*}
        L_p Y \cong \limil{n} L_p \tau_{\le n} Y,
    \end{equation*}
    see \cite[Proposition 6.12]{mattis2024fracture}.
    Hence, we compute 
    \begin{align*}
        \tau_{\ge 1} \Fib{L_p X \to L_p Y} 
        &\cong \tau_{\ge 1} \Fib{\limil{n} L_p \tau_{\le n} X \to \limil{n} L_p \tau_{\le n} Y} \\
        &\cong \tau_{\ge 1} \limil{n} \Fib{L_p \tau_{\le n} X \to L_p \tau_{\le n} Y} \\
        &\cong \tau_{\ge 1} \limil{n} \tau_{\ge 1} \Fib{L_p \tau_{\le n} X \to L_p \tau_{\le n} Y} \\
        &\cong \tau_{\ge 1} \limil{n} L_p \tau_{\ge 1} \Fib{\tau_{\le n} X \to \tau_{\le n} Y} \\
        &\cong \tau_{\ge 1} \limil{n} L_p \tau_{\ge 1} \tau_{\le n} \Fib{X \to Y} \\
        &\cong L_p \tau_{\ge 1} \Fib{X \to Y}.
    \end{align*}
    Here, we used that limits commute with limits in the second equivalence,
    e.g.\ \cite[Lemma 4.2]{mattis2024fracture} in the third equivalence,
    \cref{lem:fib-lem:truncated} in the fourth equivalence,
    \cref{lem:fib-lem:cofinal} in the fifth equivalence,
    and again \cite[Proposition 6.12]{mattis2024fracture} and the fact that the $p$-completion 
    of a connected sheaf is connected, cf.\ \cite[Lemma 3.12]{mattis2024unstable}, in the last equivalence.
    This proves the lemma.
\end{proof}

\begin{rmk}
    Note that in the proof of \cref{lem:fib-lem:main-thm}, 
    we cannot argue as in \cref{lem:fib-lem:truncated},
    because it is unclear (and probably wrong in general)
    that $L_p X$ is nilpotent for every nilpotent sheaf $X \in \topos X_*$.
\end{rmk}

In the last part of this section,
our goal is a strengthened version of \cite[Proposition 6.12]{mattis2024fracture}:
We want to prove that $p$-completion commutes with limits along locally highly connected towers,
and not just the tower of truncations $(\tau_{\le n} X)_n$.
For this, we need the following well-known lemma for which we could not find a reference in the language of $\infty$-categories.
\begin{lem} \label{lem:p-comp:three-fibers}
    Let $\Cat C$ be a pointed $\infty$-category with finite limits,
    and $f \colon X \to Y$ and $g \colon Y \to Z$ be morphisms in $\Cat C$.
    Then there is a canonical fiber sequence 
    \begin{equation*}
        \Fib{f} \to \Fib{gf} \to \Fib{g}.
    \end{equation*}
\end{lem}
\begin{proof}
    Consider the following commutative diagram:
    \begin{center}
        \begin{tikzcd}
            \Fib{f} \ar[r] \ar[d] &* \ar[d]\\
            \Fib{gf} \ar[d] \ar[r] &\Fib{g} \ar[d] \ar[r] &* \ar[d] \\
            X \ar[r, "f"] &Y \ar[r, "g"] &Z\rlap{.}
        \end{tikzcd}
    \end{center}
    We want to show that the upper left square is cartesian.
    By definition, the horizontal rectangle, the vertical rectangle 
    and the lower right square are all cartesian.
    Hence, we conclude by applying the pasting law for pullback squares twice (cf.\ 
    the dual of \cite[Lemma 4.4.2.1]{highertopoi}).
\end{proof}

The next (very technical) lemma states that $p$-completion reduces the connectivity of a morphism 
between nilpotent sheaves roughly by the local homotopy dimension of the topos.

\begin{lem} \label{lem:p-comp:connectivity}
    Suppose that $\topos X$ is locally of homotopy dimension $\le N$.
    Let $f \colon X \to Y \in \topos X_*$ be a morphism between nilpotent sheaves
    such that $\Fib{f}$ is $(N + k + 2)$-connective for some $k \ge 1$.
    Then the fiber $\Fib{L_p f}$ is $k$-connective. 
\end{lem}
\begin{proof}
    By \cref{lem:nilpotent:morphism-of-object}, also the morphism $f$ is nilpotent (cf.\ \cref{def:nilpotent:morphism-def} for a definition).
    Hence, using \cref{lem:nilpotent:refined-Post-Moore-tower}, we may choose a Moore--Postnikov refinement of $f$,
    i.e.\ a sequence of connected sheaves $(X_n)_n$ under $X$ and over $Y$ 
    with $X \cong \limil{n} X_n$, $X_0 \cong Y$,
    and fiber sequences $X_n \to X_{n-1} \to K(A_n, k_n)$,
    where the $A_n \in \AbObj{\Disc{\topos X}}$ are sheaves of abelian groups,
    and the $k_n \ge 2$ are integers, such that $k_n \to \infty$ as $n \to \infty$.
    Since $\Fib{f}$ is $(N + k + 2)$-connective, 
    we may even assume that $k_n \ge N + k + 2$ for all $n$, cf.\ \cref{lem:nilpotent:refined-Post-Moore-tower-lower-bound}.
    Note that since $Y$ is nilpotent,
    using \cite[Lemma A.12]{mattis2024unstable} inductively, we conclude 
    that also all the $X_n$ are nilpotent.
    Moreover, it is clear that $(X_n)_n$ is a highly connected tower.
    By \cref{lem:fib-lem:main-thm}
    we see that $L_p X_n \cong \tau_{\ge 1} \Fib{L_p X_{n-1} \to L_p K(A_n, k_n)}$.
    Using the proof of \cite[Lemma 6.11]{mattis2024fracture}
    (where the connectivity of the $p$-completion of an Eilenberg-MacLane sheaf is computed), 
    we see that $L_p K(A_n, k_n)$ is at least $(k+1)$-connective, and the connectivity tends to infinity 
    as $n \to \infty$. As $L_p X_{n-1}$ is connected,
    we see by a long exact sequence argument that already $\Fib{L_p X_{n-1} \to L_p K(A_n, k_n)}$ 
    is connected,
    and hence $L_p X_n \cong \Fib{L_p X_{n-1} \to L_p K(A_n, k_n)}$.
    Hence, $\Fib{L_p X_n \to L_p X_{n-1}} \cong \Omega L_p K(A_n, k_n)$ is at least $k$-connective,
    again with connectivity tending to infinity as $n \to \infty$.
    By \cref{lem:p-comp:three-fibers} we have a fiber sequence 
    \begin{equation*}
        \Fib{L_p X_{n+1} \to L_p X_n} \to \Fib{L_p X_{n+1} \to L_p Y} \to \Fib{L_p X_n \to L_p Y}.
    \end{equation*}
    Since $X_0 \cong Y$, we can thus inductively prove that also 
    $\Fib{L_p X_{n+1} \to L_p Y}$ is at least $k$-connective.
    Moreover, as the connectivity of $\Fib{L_p X_{n+1} \to L_p X_n}$ 
    tends to infinity as $n \to \infty$, we conclude that in fact 
    $(\Fib{L_p X_{n} \to L_p Y})_n$ is a highly connected tower.
    Additionally, we can compute its limit:
    \begin{equation*}
        \limil{n} \Fib{L_p X_n \to L_p Y} \cong \Fib{\limil{n} L_p X_n \to L_p Y} \cong \Fib{L_p X \to L_p Y},
    \end{equation*}
    where the second equivalence is \cref{lem:p-comp:p-comp-tower-limit}, using that $(X_n)_n$ is a highly connected tower of nilpotent sheaves.
    Hence, we conclude from \cref{lem:highly-connected:tower-limit-htpy-groups}
    that also $\Fib{L_p X \to L_p Y}$ is $k$-connective,
    which is exactly what we wanted to prove.
\end{proof}

We also need the following result which is a strengthening of \cite[Lemma 6.11]{mattis2024fracture}:
\begin{lem} \label{lem:p-comp:p-comp-highly-connected}
    Suppose that $\topos X$ admits a locally finite dimensional cover.
    Let $(X_n)_n$ be a locally highly connected tower (subordinate to this cover),
    such that all the sheaves $X_n \in \topos X_*$ are nilpotent.
    Then also the tower $(L_p X_n)_n$ is locally highly connected. 
\end{lem}
\begin{proof}
    Since $p_i^*$ commutes with the $p$-completion by \cite[Lemma 6.10]{mattis2024fracture},
    and preserves nilpotent objects (as it is the left adjoint of a geometric morphism),
    we can assume that $\topos X$ is locally of homotopy dimension $\le N$ for some $N \in \N$
    and that $(X_n)_n$ is a highly connected tower, and our goal is to show that also 
    $(L_p X_n)_n$ is highly connected.
    So let $k \in \N$.
    Since $(X_n)_n$ is highly connected, there exists $L \ge 0$ such that for all 
    $m \ge L$ the fiber $\Fib{X_m \to X_L}$ is $(k + N + 3)$-connective.
    We have to find $M \ge 0$ 
    such that for all $m \ge M$ the fiber $\Fib{L_p X_m \to L_p X_M}$ is $(k + 1)$-connective.
    The claim follows from \cref{lem:p-comp:connectivity} using $M \coloneqq L$.
\end{proof}

\begin{prop} \label{lem:p-comp:p-comp-tower-limit}
    Suppose that $\topos X$ admits a locally finite dimensional cover.
    Let $(X_n)_n$ be a locally highly connected tower (subordinate to this cover),
    such that all the sheaves $X_n \in \topos X_*$ are nilpotent.
    Then the canonical map
    $L_p \limil{n} X_n \to \limil{n} L_p X_n$ is an equivalence.
\end{prop}
\begin{proof}
    The proof is identical to that of \cite[Proposition 6.12]{mattis2024fracture},
    but we now need \cref{lem:p-comp:p-comp-highly-connected} 
    to see that $(L_p X_n)_n$ is still locally highly connected.
\end{proof}

%% file: effective.tex
\subsection{\texorpdfstring{$\Pp^1$}{P1}-Postnikov towers}
In the sequel we will need that over a perfect field any $2$-effective nilpotent motivic space 
admits a version of the Postnikov-tower, where the layers are infinite $\Pp^1$-loop spaces.
This is a deep result of Asok--Bachmann--Hopkins \cite{asok2023freudenthal}, which we now recall:
\begin{prop} [Asok--Bachmann--Hopkins] \label{lem:etale-nilpotent:effective-Postnikov-refinement}
    Let $k$ be a perfect field and $X \in \Spc{k}_*$ a pointed nilpotent motivic space.
    Suppose moreover that $X$ is $q$-effective for some $q \ge 2$.
    Then $X$ admits a refined Postnikov tower where the layers are $q$-effective infinite loop spaces,
    i.e.\ there exists an $\N$-indexed inverse system $X_\bullet \colon \op{\N} \to \Spc{k}_{*,X/}$ of pointed motivic spaces $X_i$ under $X$,
    and fiber sequences $X_{i+1} \to X_i \to K_i$, such that 
    \begin{enumerate}[label=\textbf{(\alph*)}]
        \item $X_0 \cong *$,
        \item for every $i$, $X_i$ is nilpotent and $q$-effective,
        \item $X \cong \limil{i} X_i$,
        \item the connectivities of the $K_i$ tend to $\infty$ as $i \to \infty$, and
        \item for every $i$, there exists a motivic spectrum $E_i \in \SH{k}$ which is $2$-connective and $q$-effective,
        and an equivalence $K_i \cong \pLoopP E_i$. In particular $K_i$ is $2$-connective and $q$-effective.
    \end{enumerate}
\end{prop}
\begin{proof}
    It follows from \cite[Construction 4.1.7]{asok2023freudenthal} 
    that there exist motivic spaces $X_i$ under $X$ and fiber sequences 
    $X_{i+1} \to X_i \to K_i$,
    satisfying \textbf{(a)}, \textbf{(b)}, \textbf{(c)} and \textbf{(d)},
    such that $K_i \cong \tau_{\ge(q+1, q)} K(A_i, n_i)$ with $A_i$ a strictly $\AffSpc{1}$-invariant sheaf 
    of abelian groups, $n_i \ge 2$ and $n_i \to \infty$ as $i \to \infty$.
    Here we use the notation $\tau_{\ge(q+1, q)}$ from \cite[Definition 4.1.2]{asok2023freudenthal}.
    Now the remaining claim \textbf{(e)} is \cite[Remark 4.1.13]{asok2023freudenthal}.
\end{proof}

\begin{rmk}
    In the sequel we will repeatedly use some kind of \emph{induction} on this tower.
    The main strategy will be the following:
    If a statement for motivic spaces is stable under limits 
    and holds for sheaves of the form $\pLoopP G$ for $G \in \SH{k}$,
    then it holds for any nilpotent $2$-effective motivic space.
\end{rmk}

This version of the Postnikov-tower behaves well with respect to rationalization,
as captured by the following lemma.
\begin{lem} \label{lem:etale-nilpotent:effective-Postnikov-refinement-rational}
    Let $k$ be a perfect field and $X \in \Spc{k}_*$ a pointed nilpotent motivic space.
    Suppose moreover that $X$ is $q$-effective, where $q \ge 2$.
    Let $R \subseteq \Q$ be a subring.
    If $X$ is $R$-local, then we may assume that also the $X_i$, $K_i$ and $E_i$ 
    from \cref{lem:etale-nilpotent:effective-Postnikov-refinement} are $R$-local.
\end{lem}
\begin{proof}
    Choose sheaves $X_i$, $K_i$ and $E_i$ for all $i$ as in \cref{lem:etale-nilpotent:effective-Postnikov-refinement}.
    We show that also the collection of $L_R X_i$, $L_R K_i$ and $L_R E_i$ 
    satisfies \textbf{(a)} to \textbf{(e)} of \loccitdot

    First, it follows from \cite[Theorem 4.3.11]{asok2022localization} that $L_R X_{i+1} \to L_R X_i \to L_R K_i$ 
    is still a fiber sequence. We prove the remaining points:

    \textbf{(a)}: It is clear that $L_R X_0 \cong L_R * \cong *$.

    \textbf{(b)}: That $L_R X_i$ is nilpotent was shown in \cite[Proposition 4.3.8]{asok2022localization}
        (they only show that $L_R X_i$ is weakly $\AffSpc{1}$-nilpotent, but under the assumption that 
        $X_i$ is nilpotent their proof in fact shows that $L_R X_i$ is nilpotent).
        That it is $q$-effective was shown in \cite[Proposition 4.3.4]{asok2023freudenthal}.

    \textbf{(c)}: We have $X \cong L_R X \cong L_R \limil{n} X_n \cong \limil{n} L_R X_n$ 
        by combining \cite[Proposition 4.3.8 and Theorem 4.3.9]{asok2022localization} 
        (which shows that the $R$-localization of a motivic space can be computed as 
        the $R$-localization of the underlying Nisnevich sheaf) 
        with \cite[Proposition 6.9]{mattis2024fracture}
        (which applies as $\ShvTop{\nis}{\Sm{k}}$ admits a locally finite dimensional cover,
        cf.\ \cite[Proposition A.3]{mattis2024fracture}).

    \textbf{(d)}: This follows from \cite[Proposition 4.3.8]{asok2022localization}.
    
    \textbf{(e)}: It was shown in \cite[Lemma 3.2.7 (1) and Corollary 3.2.9]{asok2023freudenthal} that $L_R E_i$ is still 
        $2$-connective and $q$-effective.
        Note that $L_R K_i \cong L_R \pLoopP E_i \cong \pLoop L_R E_i$.
        Indeed, we have $\pLoopP = \pLoop \omega^\infty$, 
        and $\pLoop$ commutes with $L_R$ on $1$-connective objects by \cite[Lemma 3.18]{mattis2024fracture}
        (or rather its analog for motivic spaces, which can be proven in exactly the same way),
        whereas $\omega^\infty$ commutes with $L_R$ by \cite[Corollary 2.7]{mattis2024fracture} as it preserves filtered colimits by \cite[Lemma 6.1]{bachmannyakerson}.
\end{proof}

%% file: etale-topos.tex
\subsection{\'Etale \texorpdfstring{$\infty$}{Infinity}-topoi}
In this section, we collect some results about the small-étale 
and smooth-étale $\infty$-topoi.
\begin{defn}
    Let $S$ be a qcqs scheme.
    \begin{itemize}
        \item We write $\smet{S}$ for the category of qcqs étale $S$-schemes,
            equipped with the étale topology, and $\ShvTopH{\et}{\smet{S}}$ for the $\infty$-topos of 
            étale hypersheaves on $\smet{S}$. 
        \item We write $\Sm{S}$ for the category of qcqs smooth $S$-schemes,
            equipped with the étale topology, and $\ShvTopH{\et}{\Sm{S}}$ for the $\infty$-topos of 
            étale hypersheaves on $\Sm{S}$. 
    \end{itemize}
\end{defn}

\begin{prop} \label{lem:etale:points}
    Let $S$ be a scheme.
    A conservative family of points of $\ShvTopH{\et}{\Sm{S}}$ is given by evaluating 
    at $\Spec(\mathcal O_{X, \overline{x}}^{sh})$,
    the spectrum of the strict henselization of the local ring of $X$ at $\overline{x}$,
    where $X$ is a smooth $S$-scheme and $\overline{x} \in X$ is a geometric point.

    Similarly, a conservative family of points of $\ShvTopH{\et}{\smet{S}}$ is given by evaluating at $\Spec(\mathcal O_{S, \overline{s}}^{sh})$,
    the spectrum of the strict henselization of the local ring of $S$ at $\overline{s}$,
    where $\overline{s} \in S$ is a geometric point.
\end{prop}
\begin{proof}
    That these functors are points is \cite[Example 4.32]{Clausen_2021}.
    We now prove that they are jointly conservative.
    Since we work with hypersheaves, we can check equivalences on homotopy sheaves 
    and are thus reduced to show that on the underlying $1$-topos 
    those points form a conservative family.
    This was proven in \cite[\href{https://stacks.math.columbia.edu/tag/03PU}{Theorem 03PU}]{stacks-project} for the small étale 
    site, the smooth case follows similarly (e.g.\ by pulling back to various small étale sites $\ShvTopH{\et}{\smet{X}}$
    for $X \in \Sm{S}$).
\end{proof}

\begin{defn}
    Let $S$ be a qcqs scheme, and $x \in S$.
    Write $\cd{}{x} \in \N \cup \{ \infty \}$
    for the étale cohomological dimension of the field $k(x)$,
    cf.\ \cite[\href{https://stacks.math.columbia.edu/tag/0F0Q}{Tag 0F0Q}]{stacks-project}.
    Similarly, write $\cd{}{S} \in \N \cup \{ \infty \}$ for the étale cohomological dimension of $S$.
\end{defn}

\begin{defn}
    Let $S$ be a qcqs scheme.
    We say that $S$ is \emph{étale bounded}
    if there is a global bound on the étale 
    cohomological dimension of the residue fields of $S$,
    i.e. $\sup_{x \in S} \cd{}{x} < \infty$.
\end{defn}

\begin{lem} \label{lem:etale:etale-bdd-finite-coh-dim}
    Let $S$ be a qcqs scheme of finite Krull-dimension that is étale bounded.
    Then also $\cd{}{S} < \infty$. 
    Moreover, if $\sup_{s \in S} \cd{}{s} \le N$, then 
    $\cd{}{S} \le N + \dim(S)$.
\end{lem}
\begin{proof}
    See e.g. the proof of \cite[Corollary 3.29]{Clausen_2021}.
\end{proof}

\begin{lem} \label{lem:etale:overcat-equiv}
    Let $S$ be a scheme, and $U \to S$ an étale $S$ scheme.
    The canonical functor 
    \begin{equation*}
        \ShvTopH{\et}{\smet{S}} \to \ShvTopH{\et}{\smet{U}}
    \end{equation*}
    induces an equivalence 
    \begin{equation*}
        \ShvTopH{\et}{\smet{S}}_{/U} \cong \ShvTopH{\et}{\smet{U}}
    \end{equation*}
\end{lem}
\begin{proof}
    Since morphisms between étale $S$-schemes are themselves étale
    \cite[\href{https://stacks.math.columbia.edu/tag/02GW}{Tag 02GW}]{stacks-project},
    we get an equivalence of sites $(\smet{S})_{/U} \cong \smet{U}$,
    where the slice category carries the canonical site structure.
    Then the result follows from \cite[Exposé III, Proposition 5.4]{SGA4},
    as all involved $\infty$-topoi are (hypercompletions of) $1$-localic topoi.
\end{proof}

\begin{lem} \label{lem:etale:finite-htyp-dim}
    Let $S$ be a qcqs
    scheme of finite Krull-dimension, such that $S$ is étale bounded.
    There exists an $N \ge 0$ such that $\ShvTopH{\et}{\smet{S}}$ is locally of homotopy dimension $\le N$.
    In particular, it is Postnikov-complete.
\end{lem}
\begin{proof}
    Let $N \coloneqq \dim(S) + \sup_{x \in S} \cd{}{x}$.
    This is finite since $S$ is étale bounded.
    By \cite[Corollary 3.29]{Clausen_2021} we see that 
    for every étale $S$-scheme $U \to S$ the cohomological dimension of $\ShvTopH{\et}{\smet{U}}$ is $\le N$.
    Note that $\ShvTopH{\et}{\smet{S}}_{/U} \cong \ShvTopH{\et}{\smet{U}}$ by \cref{lem:etale:overcat-equiv}.
    In particular, we see that $\ShvTopH{\et}{\smet{S}}$ is locally of cohomological dimension $\le N$.
    Since it is hypercomplete by definition,
    it follows that it is locally of homotopy dimension $\le N$, cf.\ \cite[Proposition 1.3.3.10]{sag}.
    It follows from \cite[Proposition 7.2.1.10]{highertopoi} that it is also Postnikov-complete.
\end{proof}

\begin{lem} \label{lem:etale:gabber-bounded-enhanced}
    Let $S$ be a qcqs scheme of finite Krull-dimension,
    which is moreover étale bounded.
    Let $p \colon X \to S$ be a smooth finite type $S$-scheme.
    Then $X$ is étale bounded.
\end{lem}
\begin{proof}
    Let $K \coloneqq \sup_{s \in S} \cd{}{s}$, and let 
    $M \coloneqq \sup_{x \in X} \operatorname{trdeg}(k(x)/k(p(x)))$
    be the supremum 
    over the transcendence degrees.
    $K$ is finite by assumption, and by combining \cite[\href{https://stacks.math.columbia.edu/tag/0A3V}{Tag 0A3V}]{stacks-project} with 
    \cite[\href{https://stacks.math.columbia.edu/tag/0A21}{Tag 0A21}]{stacks-project}
    we see that also $M$ is finite.
    We claim that $\sup_{x \in X} \cd{}{x} \le K + M$.
    This follows from \cite[\href{https://stacks.math.columbia.edu/tag/0F0T}{Tag 0F0T}]{stacks-project}
    as for $x \in X$ we have $\cd{}{x} \le \cd{}{p(x)} + \operatorname{trdeg}(k(x)/k(p(x))) \le K + M$.
\end{proof}

\begin{prop} \label{lem:etale:cover}
    Let $S$ be a qcqs scheme of finite Krull-dimension, which is moreover étale bounded.
    Then $\ShvTopH{\et}{\Sm{S}}$ admits a locally finite dimensional cover 
    in the sense of \cite[Definition 5.1]{mattis2024fracture}. 
    In particular, it is Postnikov-complete.
\end{prop}
\begin{proof}
    Arguing as in the proof of \cite[Proposition A.3]{mattis2024fracture}, one reduces 
    to the claim that for $U \in \Sm{S}$, 
    the $\infty$-topos $\ShvTopH{\et}{\smet{U}}$ is locally of homotopy dimension $\le N$ for some $N$ 
    and has enough points.
    The first statement was proven in \cref{lem:etale:finite-htyp-dim}, as $U$ is étale bounded, cf.\ \cref{lem:etale:gabber-bounded-enhanced}.
    The other statement is \cref{lem:etale:points}.
    That $\ShvTopH{\et}{\Sm{S}}$ is Postnikov-complete follows from \cite[Lemma 5.3]{mattis2024fracture}.
\end{proof}

\begin{prop} \label{lem:etale:morphism-of-covers}
    Let $S$ be a qcqs scheme of finite Krull-dimension, which is moreover étale bounded.
    The étale hypersheafification functor 
    $L_{\et} \colon \ShvTop{\nis}{\Sm{S}} \to \ShvTopH{\et}{\Sm{S}}$
    upgrades to a morphism of $\infty$-topoi with locally finite dimensional covers,
    where we equip the Nisnevich topos with the locally finite dimensional cover 
    from \cite[Proposition A.3]{mattis2024fracture},
    and the étale topos with the cover from \cref{lem:etale:cover}.
\end{prop}
\begin{proof}
    For this, it suffices to note that 
    étale hypersheafification commutes with restriction to a small étale site, 
    i.e.\ that for every smooth $S$-scheme $U$ the following diagram commutes:
    \begin{center}
        \begin{tikzcd}
            \ShvTop{\nis}{\Sm{S}}\ar[r]\ar[d, "L_{\et}"]&\ShvTop{\nis}{\smet{U}}\ar[d, "L_{\et}"]\\
            \ShvTopH{\et}{\Sm{S}}\ar[r]&\ShvTopH{\et}{\smet{U}}\rlap{.}
        \end{tikzcd}
    \end{center}
    This follows as already the corresponding diagram of sites commutes.
\end{proof}

\begin{defn} \label{def:etale:small-to-big}
    We write $\iota \colon \smet{S} \to \Sm{S}$ for the inclusion of sites, and
    \begin{equation*}
        \iota_! \colon \ShvTopH{\et}{\smet{S}} \rightleftarrows \ShvTopH{\et}{\Sm{S}} \noloc \iota^*
    \end{equation*}
    for the induced geometric morphism on the associated hypercomplete $\infty$-topoi,
    i.e.\ the hypercompletion of the geometric morphism from \cite[Proposition A.11]{pstragowski2022syntheticspectracellularmotivic}.
    Here, $\iota^*$ is given by restriction along $\iota$,
    and $\iota_!$ is left Kan extension along $\iota$, followed by 
    hypersheafification.
    Note that $\iota^*$ also has a right adjoint $\iota_*$:
    It exists before the hypercompletion by \cite[Proposition A.13]{pstragowski2022syntheticspectracellularmotivic}
    because $\iota$ has the covering lifting property \cite[Definition A.12]{pstragowski2022syntheticspectracellularmotivic} .
\end{defn}

\begin{lem} \label{lem:rigidity:iota-fully-faithful}
    The functors $\iota_!$ and $\iota_*$ are fully faithful.
\end{lem}
\begin{proof}
    The functor $\iota_!$ is fully faithful by \cite[Lemma 6.1]{Bachmann2021etalerigidity},
    note that in the reference, $e^*$ is the name for $\iota_!$.
    Since $\iota_*$ is right adjoint to the right adjoint of $\iota_!$,
    it is formal that it is also fully faithful (cf.\ (the dual of) 
    \cite[Corollary 2.7]{naumann2024symmetricmonoidalfracturesquare} for a fun proof of this fact).
\end{proof}

\begin{defn} \label{defn:etale:pushpull}
    Let $f \colon T \to S$ be a morphism of schemes.
    Pullback along $f$ defines a morphism of sites $\Sm{S} \to \Sm{T}$.
    We write
    \begin{equation*}
        f^* \colon \ShvTopH{\et}{\Sm{S}} \rightleftarrows \ShvTopH{\et}{\Sm{T}} \noloc f_*
    \end{equation*}
    for the induced geometric morphism on the associated hypercomplete $\infty$-topoi,
    i.e.\ the hypercompletion of the geometric morphism from \cite[Proposition A.11]{pstragowski2022syntheticspectracellularmotivic}.
    Here, $f_*$ is given by restriction along the morphism of sites,
    and $f^*$ is left Kan extension along $f$, followed by 
    hypersheafification.

    Similarly, we have a geometric morphism
    \begin{equation*}
        f^* \colon \ShvTopH{\et}{\smet{S}} \rightleftarrows \ShvTopH{\et}{\smet{T}} \noloc f_*
    \end{equation*}
    using the small étale sites.
\end{defn}

\begin{defn} \label{defn:etale:morphism-to-local}
    Let $S$ be a scheme, and $\overline{s} \in S$ a geometric point.
    Write $\rho_{\overline{s}} \colon \shens{S}{s} \coloneqq \Spec{\mathcal O_{S, \overline{s}}^{\operatorname{sh}}} \to S$ for the canonical morphism 
    from the spectrum of the strict henselization of the local ring of $S$ at $\overline{s}$ to $S$.
    In particular, using \cref{defn:etale:pushpull} we have adjunctions 
    \begin{equation*}
        \rho_{\overline{s}}^* \colon \ShvTopH{\et}{\Sm{S}} \rightleftarrows \ShvTopH{\et}{\Sm{\shens{S}{s}}} \noloc \rho_{\overline{s},*}
    \end{equation*}
    and
    \begin{equation*}
        \rho_{\overline{s}}^* \colon \ShvTopH{\et}{\smet{S}} \rightleftarrows \ShvTopH{\et}{\smet{(\shens{S}{s})}} \noloc \rho_{\overline{s},*}.
    \end{equation*}
\end{defn}

\begin{lem} \label{lem:etale:hens-descent}
    Let $S$ be a scheme and $\overline{s} \in S$ a geometric point.
    Suppose that $T \in \Sm{\shens{S}{s}}$.
    Then there exists a cofinal filtered system $(\overline{s} \to U_i \to S)_i$
    of quasicompact étale neighborhoods of $\overline{s}$ in $S$ (i.e.\ $\shens{S}{s} \cong \limil{i} U_i$),
    and for every $i$ a smooth $U_i$-scheme $T_i$,
    together with transition morphisms $T_i \to T_j$ making the obvious diagram commute,
    such that $T \cong \limil{i} T_i$.
    If $T$ is étale over $\shens{S}{s}$, then we can arrange that all the $T_i \to U_i$ are étale.

    Similarly, we can descend any étale cover $T' \to T$ in $\Sm{\shens{S}{s}}$,
    or any disjoint union decomposition $T = \amalg_j T_j$.
\end{lem}
\begin{proof}
    Combine \cite[\href{https://stacks.math.columbia.edu/tag/01ZM}{Lemma 01ZM}]{stacks-project} with 
    \cite[\href{https://stacks.math.columbia.edu/tag/0C0C}{Lemma 0C0C}]{stacks-project}. For the étale statements,
    we instead use \cite[\href{https://stacks.math.columbia.edu/tag/07RP}{Lemma 07RP}]{stacks-project},
    where we additionally use descent for surjective morphisms \cite[\href{https://stacks.math.columbia.edu/tag/07RR}{Lemma 07RR}]{stacks-project} for the covers.
    For disjoint union decompositions, we can combine \cite[\href{https://stacks.math.columbia.edu/tag/01ZP}{Lemma 01ZP}]{stacks-project}
    with \cite[\href{https://stacks.math.columbia.edu/tag/0EUU}{Lemma 0EUU}]{stacks-project} and again the statement about surjective morphisms.
\end{proof}

\begin{lem} \label{lem:etale:conenctivity-of-maps}
    Let $\topos X$ be an $\infty$-topos, and $T \in \topos X$ an object of homotopy dimension $\le N$ for some $N \ge 0$.
    Let $X \in \topos X$ and $m \ge 0$. Write $f \colon X \to \tau_{\le N + m} X$ for the canonical map.
    Then $\Map{\topos X}{T}{f}$ is an $(m+1)$-connective map of anima.
\end{lem}
\begin{proof}
    The global sections functor $\Gamma \colon \topos X_{/T} \to \An$ sends 
    $k$-connective morphisms to $(k - N)$-connective morphisms, cf.\ \cite[Lemma 7.2.1.7]{highertopoi}.
    The map $f$ is $(N+m+1)$-connective,
    and so is $f \times T \colon X \times T \to (\tau_{\le N + m} X) \times T \in \topos X_{/T}$.
    Note that $\Gamma(f \times T) \cong \Map{\topos X_{/T}}{T}{f \times T} \cong \Map{\topos X}{T}{f}$,
    hence we see that $\Map{\topos X}{T}{f}$ is an $(m+1)$-connective map of anima.
\end{proof}

\begin{lem} \label{lem:etale:global-to-local-formula}
    Let $S$ be a qcqs scheme of finite Krull dimension which is moreover étale bounded, and $\overline{s} \in S$ a geometric point.
    Let $\{ \overline{s} \to U_i \to S \}_{i \in I}$ be a cofinal system 
    of étale neighborhoods of $\overline{s}$, and $0 \in I$ be an initial object.
    Suppose that $T_0 \to U_0$ is a scheme over $U_0$ such that $T_0 \to U_0 \to S$
    is smooth. Let $T_i \coloneqq T_0 \times_{U_0} U_i$ and $T \coloneqq T_0 \times_{U_0} \shens{S}{s} \cong \limil{i} T_i$.
    Then for every $X \in \ShvTopH{\et}{\Sm{S}}$
    we have $\rho_{\overline{s}}^* X(T) \cong \colimil{i} X(T_i)$.
\end{lem}
\begin{proof}
    We first prove the lemma under the additional assumption that $X$ is $N$-truncated for some $N \ge 0$.
    Write $f^* \colon \Cat P(\Sm{S}) \rightleftarrows \Cat P(\Sm{\shens{S}{s}}) \noloc f_*$ 
    for the adjunction on presheaves where 
    $f_*$ is given by precomposition along the pullback functor $f \colon \Sm{S} \to \Sm{\shens{S}{s}}$,
    and $f^*$ is left Kan extension.
    In particular, $\rho_{\overline{s}}^* X \cong L_{\et} f^* X$.
    Hence, it suffices to show that (1) $f^* X(T) \cong \colimil{i} X(T_i)$ 
    and (2) $f^* X$ satisfies étale hyperdescent.

    \textbf{(1)}: By definition of left Kan extension we have an equivalence 
    \begin{equation*}
        f^* X(T) \cong \colim{U, T \to f(U)} X(U),
    \end{equation*}
    where the colimit runs over the comma category $(T \downarrow \Sm{S})$.
    By the universal property of the pullback, 
    this category is equivalent to ${(\Sm{S})}_{T/}$,
    where we view $T$ as an $S$-scheme.
    It is left to show that the $T_i$ are cofinal in this category.
    For this, let $T \to U$ any $S$-morphism with $U \in \Sm{S}$,
    and we have to show that there is an $i$ such that $T \to U$ 
    factors over $T \to T_i$.
    This is an immediate application of \cite[\href{https://stacks.math.columbia.edu/tag/01ZC}{Proposition 01ZC}]{stacks-project}.

    \textbf{(2)}:
    Since $f^*$ preserves truncated objects, and since any truncated sheaf is automatically 
    a hypersheaf, it suffices to show that $f^* X$ is an étale sheaf.
    Note that $f^* X$ is a $\Sigma$-sheaf (i.e.\ it sends finite coproducts to products):
    Indeed, let $(V_j)_j$ be a finite family of smooth schemes over $\shens{S}{s}$.
    By \cref{lem:etale:hens-descent} we may assume that 
    we get compatible families $(V_{j,i})_j$ over a cofinal system 
    of étale neighborhoods $\overline{s} \to U_i \to S$,
    such that $V_{j,i} \times_{U_i} \shens{S}{s} \cong V_i$.
    Now, by (1) we know that $f^* X(\amalg_j V_j) \cong \colim{i} X(\amalg_j V_{j,i})$.
    Since $X$ is a $\Sigma$-sheaf, and since finite limits commute with filtered colimits in $\An$,
    we conclude that also $f^* X$ is a $\Sigma$-sheaf.
    Hence, it now suffices to prove the sheaf condition for an étale cover $V \to U$
    consisting of a single morphism. Again, using \cref{lem:etale:hens-descent},
    we can find compatible étale covers $V_i \to U_i$.
    We argue as above, using (1) and the fact that on homotopy groups the totalization of the \v{C}ech nerve 
    behaves like a finite limit, and hence again commutes with filtered colimits.

    We end the proof by showing the general statement.
    So let $X \in \ShvTopH{\et}{\Sm{S}}$ be arbitrary.
    First note that $T_0$ is étale bounded, with some bound $K \ge 0$ 
    by \cref{lem:etale:gabber-bounded-enhanced}.
    Moreover, by the proof of this lemma, since any $T_i$ is an étale $T_0$-scheme,
    we see that it is again étale bounded with the same bound $K$,
    and similarly for the pro-étale $T_0$-scheme $T$.
    Hence, all the $T_i \in \ShvTopH{\et}{\Sm{S}}$ 
    and also $T \in \ShvTopH{\et}{\Sm{\shens{S}{s}}}$ have cohomological dimension $\le N$,
    where $N \coloneqq K + \dim(T_0)$,
    cf.\ \cref{lem:etale:etale-bdd-finite-coh-dim} (as $\dim(T) = \dim(T_i) = \dim(T_0)$).
    By the proof of \cite[Proposition 1.3.3.10]{highertopoi},
    using Postnikov-completeness, $T$ and the $T_i$ have in fact homotopy dimension $\le N$.
    We now have for every $m \ge 0$ a commutative diagram 
    \begin{center}
        \begin{tikzcd}
            (\rho_{\overline{s}}^* X)(T) = \Map{\shens{S}{s}}{T}{\rho_{\overline{s}}^* X} \ar[dd] \ar[r] &\Map{\shens{S}{s}}{T}{\tau_{\le N + m} \rho_{\overline{s}}^* X} \ar[d, "\cong"] \\
            &\Map{\shens{S}{s}}{T}{\rho_{\overline{s}}^* \tau_{\le N + m} X} \ar[d, "\cong"] \\
            \colimil{i} X(T_i) = \colimil{i} \Map{S}{T_i}{X} \ar[r] &\colimil{i} \Map{S}{T_i}{\tau_{\le N + m} X}
        \end{tikzcd}
    \end{center}
    Here, the top right vertical map is an equivalence since $\rho_{\overline{s}}^*$ is a geometric morphism,
    and the bottom right vertical map is an equivalence by the above special case where $X$ was assumed to be truncated.
    Our goal is to show that the left vertical map of anima is an equivalence,
    it suffices to show that it is $\infty$-connective.
    For this, we will show that the horizontal maps are both $m$-connective, 
    since the left vertical map is independent of $m$, this then proves the claim.
    First note that $m$-connective maps are stable under filtered colimits.
    Since $T$ and the $T_i$ have homotopy dimension $\le N$,
    the claim now follows from \cref{lem:etale:conenctivity-of-maps}.
\end{proof}

\begin{lem} \label{lem:etale:morphism-to-local-commutation}
    Let $S$ be a qcqs scheme of finite Krull dimension, that is moreover étale bounded, and $\overline{s} \in S$ a geometric point.
    The following squares are commutative:
    \begin{center}
        \begin{tikzcd}
            \ShvTopH{\et}{\smet{S}} \ar[r, "\rho_{\overline{s}}^*"] \ar[d, "\iota_!"] &\ShvTopH{\et}{\smet{(\shens{S}{s})}} \ar[d, "\iota_!"]  \\
            \ShvTopH{\et}{\Sm{S}} \ar[r, "\rho_{\overline{s}}^*"] &\ShvTopH{\et}{\Sm{\shens{S}{s}}}
        \end{tikzcd}
    \end{center}
    and
    \begin{center}
        \begin{tikzcd}
            \ShvTopH{\et}{\smet{S}} \ar[r, "\rho_{\overline{s}}^*"]  &\ShvTopH{\et}{\smet{(\shens{S}{s})}}  \\
            \ShvTopH{\et}{\Sm{S}} \ar[r, "\rho_{\overline{s}}^*"] \ar[u, "\iota^*"'] &\ShvTopH{\et}{\Sm{\shens{S}{s}}} \ar[u, "\iota^*"']\rlap{.}
        \end{tikzcd}
    \end{center}
\end{lem}
\begin{proof}
    For the first square, is suffices to show that the associated square of right adjoints commutes.
    As all the right adjoints are given by restriction along morphisms of sites,
    this follows since they already commute on the level of sites.

    For the second square,
    note that there exists a natural transformation
    \begin{equation*}
        \rho_{\overline{s}}^* \iota^* \xrightarrow[\simeq]{\mathrm{unit}} \iota^* \iota_! \rho_{\overline{s}}^* \iota^* \cong \iota^* \rho_{\overline{s}}^* \iota_! \iota^* \xrightarrow{\mathrm{counit}} \iota^* \rho_{\overline{s}}^*
    \end{equation*}
    given by the Beck--Chevalley 
    transformation (using the commutativity of the first diagram).
    The morphism induced by the unit is an equivalence since $\iota_!$ is fully faithful, cf.\ \cref{lem:rigidity:iota-fully-faithful}.
    Hence, it suffices to show that also the morphism induced by the counit is an equivalence.
    So let $X \in \ShvTopH{\et}{\Sm{S}}$.
    Since $\iota^*$ is given by restriction,
    we have to show that for every $T \in \smet{\left(\shens{S}{s}\right)}$ the canonical map 
    \begin{equation*}
        (\rho_{\overline{s}}^* \iota_! \iota^* X)(T) \xrightarrow{\mathrm{counit}} (\rho_{\overline{s}}^* X)(T)
    \end{equation*}
    is an equivalence.
    By \cref{lem:etale:hens-descent} we may choose étale neighborhoods $\{\overline{s} \to U_i \to S\}_{i}$
    and for every $i$ an étale $U_i$-scheme $T_i$ 
    such that $T \cong \limil{i} T_i$.
    It now follows from \cref{lem:etale:global-to-local-formula}
    that this morphism is equivalent to the morphism 
    \begin{equation*}
        \colimil{i} (\iota_! \iota^* X)(T_i) \xrightarrow{\mathrm{counit}} \colimil{i} X(T_i).
    \end{equation*}
    Since $T_i \in \smet{S}$ and $\iota_!$ is given by left Kan extension 
    along the inclusion $\smet{S} \to \Sm{S}$, it follows that this morphism is an equivalence.
\end{proof}

\begin{lem} \label{lem:etale:morphism-to-local-conservative}
    Let $S$ be a qcqs scheme of finite Krull dimension, that is moreover étale bounded.
    The collection of functors $\{\rho_{\overline{s}}^*\}_{\overline{s} \in S}$ is jointly conservative 
    on both $\ShvTopH{\et}{\Sm{S}}$ and $\ShvTopH{\et}{\smet{S}}$ (where the $\overline{s}$ are the geometric points of $S$).
\end{lem}
\begin{proof}
    We first show that the functors $\rho_{\overline{s}}^* \colon \ShvTopH{\et}{\Sm{S}} \to \ShvTopH{\et}{\Sm{\shens{S}{s}}}$
    are jointly conservative.
    So let $f \colon X \to Y$ be a morphism in $\ShvTopH{\et}{\Sm{S}}$ such that 
    $\rho_{\overline{s}}^* f$ is an equivalence for all geometric points $\overline{s}$.
    As for the étale topology strictly henselian local rings form the points of the topos,
    and we are working with hypersheaves, it suffices to show 
    that for every $T \in \Sm{S}$ and geometric point $\overline{t} \in T$
    the morphism $f \colon X(\shens{T}{t}) \to Y(\shens{T}{t})$ is an equivalence.
    Fix $T$ and $\overline{t}$. Hence, we have to show that 
    \begin{equation} \label{eq:etale:morphism-to-local-conservative:morphism-in-question}
        f \colon \colim{\overline{t} \to U \xrightarrow{\et} T} X(U) \to \colim{\overline{t} \to U \xrightarrow{\et} T} Y(U)
    \end{equation}
    is an equivalence.
    Write $\overline{s} \in S$ for the geometric point under $\overline{t} \in T$,
    so that there is a commutative diagram 
    \begin{center}
        \begin{tikzcd}
            \shens{T}{t} \ar[r, "\rho_{\overline{t}}"] \ar[d, "\alpha"] &T \ar[d, "\beta"] \\
            \shens{S}{s} \ar[r, "\rho_{\overline{s}}"] &S \rlap{.}
        \end{tikzcd}
    \end{center}
    In particular, $\rho_{\overline{t}}^* \beta^* f = \alpha^* \rho_{\overline{s}}^* f$ is an equivalence.
    By \cref{lem:etale:global-to-local-formula}, 
    the equivalence $\rho_{\overline{t}}^* \beta^* f \colon \rho_{\overline{t}}^* \beta^* X(\shens{T}{t}) \cong \rho_{\overline{t}}^* \beta^* Y(\shens{T}{t})$
    may be identified with the morphism 
    \begin{equation*}
        \beta^* f \colon \colim{\overline{t} \to U \xrightarrow{\et} T} \beta^* X(U) \xrightarrow{\simeq} \colim{\overline{t} \to U \xrightarrow{\et} T} \beta^* Y(U).
    \end{equation*}
    Since any étale $T$-scheme $U$ is in particular a smooth $S$-scheme, 
    we see that $\beta^* Z(U) \cong Z(U)$ for any $Z$.
    Hence, the above equivalence can be identified with the morphism
    \eqref{eq:etale:morphism-to-local-conservative:morphism-in-question},
    which proves the claim.

    We finish the proof by showing that also the functors $\rho_{\overline{s}}^* \colon \ShvTopH{\et}{\smet{S}} \to \ShvTopH{\et}{\smet{(\shens{S}{s})}}$
    are jointly conservative.
    For this, let $f \colon X \to Y \in \ShvTopH{\et}{\smet{S}}$ be a morphism such that 
    $\rho_{\overline{s}}^* f$ is an equivalence for all geometric points $\overline{s} \in S$.
    Then also $\iota_! \rho_{\overline{s}}^* f \cong \rho_{\overline{s}}^* \iota_! f$ is an equivalence for all $\overline{s} \in S$, using \cref{lem:etale:morphism-to-local-commutation}.
    Hence, it follows from the first part that $\iota_! f$ is an equivalence.
    Thus, we deduce that $f$ is already an equivalence since $\iota_!$ is fully faithful, cf.\ \cref{lem:rigidity:iota-fully-faithful}.
\end{proof}

\begin{lem} \label{lem:etale:small-topos-over-strictly-local}
    Let $k$ be a separably closed field.
    Then evaluation at $k$ induces 
    an equivalence $\ShvTopH{\et}{\smet{k}} \xrightarrow{\simeq} \An$. 
\end{lem}
\begin{proof}
    Since $k$ is separably closed,
    the functor $\operatorname{Fin} \to \smet{k}$ from finite sets to $\smet{k}$
    given by sending a set $A$ to $\amalg_A \Spec(k)$ is an equivalence of categories.
    Under this equivalence, the étale topology corresponds to the topology 
    generated by finite coproduct decompositions.
    Hence, using \cite[Lemma 2.4]{bachmann2020norms},
    we see that $\ShvTopH{\et}{\smet{k}} \cong \mathcal{P}_{\Sigma}(\operatorname{Fin})$
    (the reference shows this holds before hypercompletion, and the right-hand side is hypercomplete by \cite[Lemma 2.6]{bachmann2020norms}),
    i.e.\ the $\infty$-category of product-preserving presheaves on finite sets,
    which is equivalent to the $\infty$-category of anima, cf.\ \cite[§5.1.4]{cesnavicius2023purityflatcohomology}.
    Unwinding the definitions, we see that this equivalence 
    is precisely given by evaluation at $k$.
\end{proof}

\begin{lem} \label{lem:rigidtiy:stalk-functor}
    Suppose that $k$ is a separably closed field.
    Then the functor 
    \begin{equation*}
        (-)(k) \colon \ShvTopH{\et}{\Sm{k}} \to \An
    \end{equation*}
    given by evaluation at $k$ has a left and a right adjoint.
    In particular, it preserves all limits and colimits and commutes with $p$-completion,
    i.e.\ $L_p (F(k)) \cong (L_p F)(k)$ for all $F \in \ShvTopH{\et}{\Sm{k}}$.
\end{lem}
\begin{proof}
    There is a commutative diagram
    \begin{center}
        \begin{tikzcd}
            \ShvTopH{\et}{\Sm{k}} \ar[dr, "(-)(k)"']\ar[r, "\iota^*"]&\ShvTopH{\et}{\smet{k}} \ar[d, "\cong"',"(-)(k)"]\\
            &\An\rlap{.}
        \end{tikzcd}
    \end{center}
    Since $k$ is separably closed, we have seen in \cref{lem:etale:small-topos-over-strictly-local} that the right vertical arrow is an equivalence.
    Moreover, $\iota^*$ has a left and a right adjoint, cf.\ \cref{def:etale:small-to-big}, hence the same is true for the diagonal functor.
    That evaluation at $k$ then also commutes with $p$-completion 
    follows formally from the fact that it preserves both $p$-equivalences and $p$-complete objects,
    cf.\ \cite[Lemma 3.11]{mattis2024unstable}.
\end{proof}

%% file: canonical.tex
\subsection{The canonical resolution of a connective sheaf}
Let $\topos X$ be an $\infty$-topos.
In this section, we prove that for every $n$-connective sheaf $X \in \topos X_{*, \ge n}$
there is a canonical way of writing $X$ as a sifted colimit of sheaves of the form 
$\Sigma^n T$ with $T \in \topos X_*$.

\begin{prop} \label{lem:canonical:main-thm}
    Let $X \in \topos X_{*, \ge n}$.
    Then $X$ is the geometric realization of the simplicial diagram 
    given by $[k] \mapsto (\Sigma^n\Omega^n)^{k+1}(X)$.
    In particular, $X$ is a sifted colimit of sheaves of the form 
    $\Sigma^n T$ with $T \in \topos X_*$.
\end{prop}
\begin{proof}
    If we can show that the functor $\Omega^n \colon \topos X_{*, \ge n} \to \topos X_*$ 
    is monadic (i.e.\ is conservative and preserves geometric realizations of $\Omega^n$-split 
    simplicial diagrams), then this follows from (the proof of) \cite[Proposition 4.7.3.14]{higheralgebra}.
    Consider the following commutative diagram of right adjoints
    \begin{center}
        \begin{tikzcd}
            \operatorname{Mon}_{\mathbb E_n} (\topos X_*) \ar[dr, "\fgt{}"']&\operatorname{Mon}^{\operatorname{gp}}_{\mathbb E_n} (\topos X_{*}) \ar[l, hook] & \topos X_{*, \ge n} \ar[l, "\Omega^n", "\cong"'] \ar[dl, "\Omega^n"] \\
            &\topos X_{*}\rlap{,}
        \end{tikzcd}
    \end{center}
    where the right horizontal morphism is an equivalence by \cite[Theorem 5.2.6.15]{higheralgebra}.
    Since $\fgt{} \colon \operatorname{Mon}_{\mathbb E_n} (\topos X_*) \to \topos X_*$ 
    is monadic, it suffices to show that the forgetful functor 
    $\operatorname{Mon}^{\operatorname{gp}}_{\mathbb E_n} (\topos X_{*}) \hookrightarrow \operatorname{Mon}_{\mathbb E_n} (\topos X_{*})$
    is conservative and preserves geometric realizations that are $\fgt{}$-split.
    The first claim is clear.
    For the second claim, we have to show that a $\fgt{}$-split geometric realization (in $\operatorname{Mon}_{\mathbb E_n} (\topos X_{*})$) of grouplike $\mathbb E_n$-monoids is still grouplike.
    Grouplike objects are characterized by certain maps $G \times G \to G \times G$ being equivalences, 
    cf.\ \cite[Definition 5.2.6.2]{higheralgebra}. 
    Since geometric realizations that are $\fgt{}$-split are computed underlying, and since colimits are universal in $\topos X$,
    the result follows.
\end{proof}

%% file: etale-motivic-htpy.tex
\section{\texorpdfstring{$\AffSpc{1}$}{A1}-invariance and étale sheaves}

In this section, we will define $\AffSpc{1}$-invariant étale sheaves,
and discuss ($p$-completely) étale $\AffSpc{1}$-nilpotent sheaves,
a variant of nilpotence in an $\infty$-topos, where the layers of the refined Postnikov tower 
are given by ($p$-completely) $\AffSpc{1}$-invariant infinite loop sheaves.

\begin{defn}
    Let $S$ be a scheme.
    Write $\SpcEt{S} \subset \ShvTopH{\et}{\Sm{S}}$ for the $\infty$-category 
    of $\AffSpc{1}$-invariant étale hypersheaves on $\Sm{S}$.
    As in the Nisnevich case, there is a Bousfield localization
    \begin{equation*}
        L_{\et, \AffSpc{1}} \colon \ShvTopH{\et}{\Sm{S}} \rightleftarrows \SpcEt{S} \noloc \iota_{\et, \AffSpc{1}}
    \end{equation*}
    at the projections $X \times \AffSpc{1} \to X$ for $X \in \Sm{S}$.
    We call objects of this $\infty$-category \emph{étale motivic spaces} over $S$.
    We write $\SHoneet{S} \coloneqq \Sp(\SpcEt{S})$
    for the stabilization.
\end{defn}

Recall the following lemma:
\begin{lem}[{e.g.\ \cite[Lemma A.1]{mattis2024unstable}}] \label{lem:etale-mot:SHone-adjunction}
    Let $S$ be a scheme. There is an adjunction
    \begin{equation*}
        L_{\et, \AffSpc{1}} \colon \Sp(\ShvTopH{\et}{\Sm{S}}) \rightleftarrows \SHoneet{S} \noloc \iota_{\et, \AffSpc{1}},
    \end{equation*}
    such that the following two diagrams are commutative:
    \begin{center}
        \begin{tikzcd}
            \Sp(\ShvTopH{\et}{\Sm{S}}) \ar[r, "L_{\et, \AffSpc{1}}"] &\SHoneet{S} &\Sp(\ShvTopH{\et}{\Sm{S}}) \ar[d, "\pLoop"] &\SHoneet{S} \ar[l, "\iota_{\et, \AffSpc{1}}"] \ar[d, "\pLoop"] \\
            \ShvTopH{\et}{\Sm{S}}_* \ar[u, "\Sus"] \ar[r, "L_{\et, \AffSpc{1}}"] &\SpcEt{S}_* \ar[u, "\Sus"] &\ShvTopH{\et}{\Sm{S}}_* &\SpcEt{S}_* \ar[l, "\iota_{\et, \AffSpc{1}}"] \rlap{.}
        \end{tikzcd}
    \end{center}
\end{lem}

\begin{lem}
    Let $S$ be a scheme.
    The functor $\iota_{\et, \AffSpc{1}} \colon \SHoneet{S} \to \Sp(\ShvTopH{\et}{\Sm{S}})$
    is fully faithful, with essential image those étale hypersheaves of spectra that are $\AffSpc{1}$-invariant.
    In particular, $\SHoneet{S}$ is equivalent to the $\infty$-category with the same name from \cite{Bachmann2021etalerigidity,bachmann2021remarksetalemotivicstable}.
\end{lem}
\begin{proof}
    The functor is fully faithful by \cite[Lemma A.2]{mattis2024unstable}.
    Everything in the essential image is $\AffSpc{1}$-invariant.
    Indeed, this can be checked after applying the functors $\pLoop \Sigma^n$ for varying $n$, 
    but now note that $\pLoop \iota_{\et, \AffSpc{1}} \cong \iota_{\et, \AffSpc{1}} \pLoop$ by \cref{lem:etale-mot:SHone-adjunction}.
    On the other hand, suppose that $E \in \Sp(\ShvTopH{\et}{\Sm{S}})$ is $\AffSpc{1}$-invariant.
    In particular, the $\pLoop \Sigma^n E$ are all $\AffSpc{1}$-invariant and thus define an object of $\SHoneet{S}$. 
\end{proof}

\begin{lem} \label{lem:etale:coherent-compact}
    Let $\topos X$ be a Postnikov-complete $\infty$-topos, and $T \in \topos X$ a coherent object of 
    cohomological dimension $\le N$ for some $N$.
    Then $T \in \topos X$ and $\pSus T \in \spectra{X}$ are both compact.
\end{lem}
\begin{proof}
    The second claim follows from the first one since the right adjoint $\Loop$ 
    of $\pSus$ preserves filtered colimits (recall that in any $\infty$-topos,
    filtered colimits commute with finite limits).

    Now, $\tau_{\le n} T$ is compact in $\ShvTopH{\et}{\Sm{S}}_{\le n}$ for every $n$ by \cite[Corollary A.2.3.2]{sag}.
    Let $X_i$ be a filtered system in $\topos X$.
    For every $n \ge 0$ we have the following commutative diagram:
    \begin{center}
        \begin{tikzcd}
            \Map{\topos X}{T}{\colimil{i} X_i} \ar[d] \ar[r] &\colimil{i} \Map{\topos X}{T}{X_i} \ar[d]\\
            \Map{\topos X}{T}{\tau_{\le N + m} \colimil{i} X_i} \ar[d, "\cong"] \ar[r] &\colimil{i} \Map{\topos X}{T}{\tau_{\le N + m} X_i} \ar[d, "\cong"]\\
            \Map{\topos X_{\le N + m}}{\tau_{\le N + m} T}{\tau_{\le N + m} \colimil{i} X_i} \ar[r, "\cong"] &\colimil{i} \Map{\topos X_{\le N + m}}{\tau_{\le N + m} T}{\tau_{\le N + m} X_i}
        \end{tikzcd}
    \end{center}
    Here, the bottom vertical maps are equivalences by adjunction,
    and the bottom horizontal map is an equivalence 
    since $\tau_{\le N +m} T$ is compact in $\topos X_{\le N + m}$,
    and $\tau_{\le N + m} \colon \topos X \to \topos X_{\le N + m}$ preserves colimits.
    Hence, also the middle horizontal map is an equivalence.
    By the proof of \cite[Proposition 1.3.3.10]{sag},
    using Postnikov-completeness of $\topos X$, the object $T$ has homotopy dimension $\le N$.
    Hence, both top vertical maps are $(m + 1)$-connective by \cref{lem:etale:conenctivity-of-maps}.
    Hence, also the top horizontal map is $(m+1)$-connective.
    Since $m$ was arbitrary, we conclude that the top horizontal map 
    is $\infty$-connective, hence an equivalence of anima.
\end{proof}

\begin{lem} \label{lem:etale-mot:A1-stable-under-filtered}
    Let $S$ be a qcqs scheme of finite Krull-dimension, that is moreover étale bounded.
    The functor $\iota_{\et, \AffSpc{1}} \colon \SpcEt{S} \to \ShvTopH{\et}{\Sm{S}}$ 
    commutes with filtered colimits,
    and the functor $\iota_{\et, \AffSpc{1}} \colon \SHoneet{S} \to \Sp(\ShvTopH{\et}{\Sm{S}})$ 
    commutes with all colimits.
\end{lem}
\begin{proof}
    Since $\iota_{\et, \AffSpc{1}} \colon \SHoneet{S} \to \Sp(\ShvTopH{\et}{\Sm{S}})$ is exact,
    it commutes with finite colimits.
    Therefore, we only have to show that it commutes with filtered colimits.

    As $\AffSpc{1}$-invariant objects are the local objects 
    with respect to the morphisms $T \times \AffSpc{1} \to T$ (resp.\ $\pSus (T \times \AffSpc{1}) \to \pSus T$),
    it suffices to show that $T$ (resp.\ $\pSus T$) is compact for every $T \in \Sm{S}$.

    Any $T \in \Sm{S}$ has finite cohomological dimension by \cref{lem:etale:gabber-bounded-enhanced,lem:etale:etale-bdd-finite-coh-dim}.
    As the étale topology is finitary, we see from \cite[Proposition A.3.1.3]{sag}
    that $T \in \ShvTopH{\et}{\Sm{S}}$ is coherent.
    Hence, both $T$ and $\pSus T$ are compact by \cref{lem:etale:coherent-compact},
    which immediately implies the lemma.
\end{proof}

\begin{prop} \label{lem:etale-motives:preserves-connectedness}
    Let $S$ be a qcqs scheme of finite Krull-dimension, that is moreover étale bounded.
    The $\AffSpc{1}$-localization functor
    \begin{equation*}
        L_{\et,\AffSpc{1}} \colon \ShvTopH{\et}{\Sm{S}} \to \ShvTopH{\et}{\Sm{S}}
    \end{equation*}
    preserves connected objects.
\end{prop}
\begin{proof}
    As the étale topology is finitary, étale hypersheaves are stable under filtered colimits 
    in the $\infty$-category of presheaves.
    Hence, the same proof as for the Nisnevich topology applies (but with Nisnevich sheafification 
    replaced by étale hypersheafification),
    see e.g.\ \cite[Lemma 1.2]{bachmann2024stronglya1}.
\end{proof}

\begin{prop} \label{lem:etale-motives:connected-cover-invariant}
    Let $S$ be a qcqs scheme of finite Krull-dimension, that is moreover étale bounded.
    Let $X \in \ShvTopH{\et}{\Sm{S}}_*$.
    If $X$ is $\AffSpc{1}$-invariant,
    then also $\tau_{\ge 1} X$ is $\AffSpc{1}$-invariant.
\end{prop}
\begin{proof}
    Consider the counit $\tau_{\ge 1} X \to X$.
    Since $X$ is $\AffSpc{1}$-invariant by assumption,
    this map factors over the canonical map $\tau_{\ge 1} X \to L_{\et, \AffSpc{1}} \tau_{\ge 1} X$.
    As $L_{\et, \AffSpc{1}} \tau_{\ge 1} X$ is connected by \cref{lem:etale-motives:preserves-connectedness},
    it follows that $L_{\et, \AffSpc{1}} \tau_{\ge 1} X \to X$ factors over $\tau_{\ge 1} X \to X$.
    Hence, we have the following commutative diagram:
    \begin{center}
        \begin{tikzcd}
            \tau_{\ge 1} X \ar[r] \ar[dr] &L_{\et, \AffSpc{1}} \tau_{\ge 1} X \ar[d] \ar[r] &\tau_{\ge 1} X \ar[dl]\\
            &X\rlap{,}
        \end{tikzcd}
    \end{center}
    where the top composition is the identity on $\tau_{\ge 1} X$.
    As $\AffSpc{1}$-invariant objects are stable under retracts (this holds for the local 
    objects of any Bousfield localization),
    the proposition follows.
\end{proof}

\begin{lem} \label{lem:etale-motives:morphism-to-local-preserves-A1}
    Let $S$ be a qcqs scheme of finite Krull dimension, that is moreover étale bounded, and $\overline{s} \in S$ a geometric point.
    The functor
    \begin{equation*}
        \rho_{\overline{s}}^* \colon \ShvTopH{\et}{\Sm{S}} \to \ShvTopH{\et}{\Sm{\shens{S}{s}}}
    \end{equation*}
    from \cref{defn:etale:morphism-to-local} preserves $\AffSpc{1}$-invariant sheaves.
\end{lem}
\begin{proof}
    Let $X \in \ShvTopH{\et}{\Sm{S}}$ be $\AffSpc{1}$-invariant.
    By definition, the strict henselization 
    \begin{equation*}
        \shens{S}{s} \cong \lim{\overline{s} \to U \xrightarrow{\et} S} U
    \end{equation*}
    is the limit over all étale neighborhoods $U$ of $\overline{s}$ in $S$.
    Let $T \in \Sm{\shens{S}{s}}$, so that by \cref{lem:etale:hens-descent} 
    we may choose a cofinal system of étale neighborhoods $\overline{s} \to U_i \to S$
    and for every $i$ a smooth $U_i$-scheme $T_i$ 
    such that $T \cong \limil{i} T_i$.
    By \cref{lem:etale:global-to-local-formula} we have that 
    $\rho_{\overline{s}}^*(X)(T) \cong \colimil{i} X(T_i)$.
    The result follows from $\AffSpc{1}$-invariance of $X$ 
    and the fact that $\limil{i} (T_i \times \AffSpc{1}) \cong T \times \AffSpc{1}$.
\end{proof}

\begin{lem} \label{lem:etale-motives:morphism-to-local-detects-A1}
    Let $S$ be a scheme of finite Krull dimension that is moreover étale bounded.
    The collection of functors
    \begin{equation*}
        \rho_{\overline{s}}^* \colon \ShvTopH{\et}{\Sm{S}} \to \ShvTopH{\et}{\Sm{\shens{S}{s}}}
    \end{equation*}
    from \cref{defn:etale:morphism-to-local} where $\overline{s}$ ranges over the collection 
    of geometric points of $S$ jointly detect $\AffSpc{1}$-invariant objects,
    i.e.\ if $\rho_{\overline{s}}^* X$ is $\AffSpc{1}$-invariant for all $\overline{s}$,
    then $X$ is $\AffSpc{1}$-invariant.
\end{lem}
\begin{proof}
    Let $X \in \ShvTopH{\et}{\Sm{S}}$,
    such that $\rho_{\overline{s}}^* X$ is $\AffSpc{1}$-invariant for all $\overline{s}$.
    Consider the canonical map $f \colon X \to L_{\et,\AffSpc{1}} X$.
    We have seen in \cref{lem:etale-motives:morphism-to-local-preserves-A1} 
    that $\rho_{\overline{s}}^* L_{\et,\AffSpc{1}} X$ is $\AffSpc{1}$-invariant.
    Moreover, since $\rho_{\overline{s}}^* (\AffSpc{1}_S \times_S T) = \AffSpc{1}_{\shens{S}{s}} \times_{\shens{S}{s}} \rho_{\overline{s}}^* T$
    for every smooth $S$-scheme $T$, 
    it follows immediately that $\rho_{\overline{s}}^*$ preserves $\AffSpc{1}$-equivalences.
    In particular, we see that $\rho_{\overline{s}}^* f$ 
    is an $\AffSpc{1}$-equivalence between $\AffSpc{1}$-invariant sheaves, hence an equivalence.
    This allows us to conclude from \cref{lem:etale:morphism-to-local-conservative} 
    that already $f$ is an equivalence, i.e.\ $X$ is $\AffSpc{1}$-invariant.
\end{proof}

\begin{defn} \label{def:etale-nilpotent:etale-nilpotent}
    Let $S$ be a scheme.
    Let $\Cat L \colon \Stab{\ShvTopH{\et}{\Sm{S}}} \to \Stab{\ShvTopH{\et}{\Sm{S}}}$ be a left Bousfield localization.
    Let $X \in \ShvTopH{\et}{\Sm{S}}_*$.
    We say that $X$ is \emph{$\Cat L$-locally étale $\AffSpc{1}$-nilpotent} if
    there exists a highly connected tower $(X_n)_n$ under $X$ such that the following holds:
    \begin{enumerate}
        \item $X_0 \cong *$.
        \item $X \cong \limil{n} X_n$.
        \item Each of the morphisms $X_{n+1} \to X_n$ is part of a fiber sequence $X_{n+1} \to X_n \to K_n$
            where $K_n \cong \pLoop E_n$ is a connected infinite loop sheaf
            for some $E_n \in \Stab{\ShvTopH{\et}{\Sm{S}}}_{\ge 1}$.
        \item For every $n$ the object $\Cat L E_n$ is $\AffSpc{1}$-invariant.
    \end{enumerate}
    In the sequel, we will consider the following three special cases:
    \begin{itemize}
        \item If $\Cat L = \id{}$ is the identity functor, we will say that $X$ is \emph{étale $\AffSpc{1}$-nilpotent}.
        \item If $\Cat L = L_p$ is $p$-completion for some prime $p$, we will say that $X$ is \emph{$p$-completely étale $\AffSpc{1}$-nilpotent}.
        \item If $\Cat L = L_\Q$ is rationalization, we will say that $X$ is \emph{rationally étale $\AffSpc{1}$-nilpotent}.
    \end{itemize}
\end{defn}

\begin{lem} \label{lem:etale-nilpotent:etale-nilpotent-implies-p-comp-rationally}
    Let $S$ be a scheme.
    Let $X \in \ShvTopH{\et}{\Sm{S}}_*$ be étale $\AffSpc{1}$-nilpotent.
    Then $X$ is $p$-completely étale $\AffSpc{1}$-nilpotent for every prime $p$,
    and moreover rationally étale $\AffSpc{1}$-nilpotent.
\end{lem}
\begin{proof}
    Let $E \in \Sp(\ShvTopH{\et}{\Sm{S}})$ be $\AffSpc{1}$-invariant.
    Then $L_p E \cong \limil{n} E \sslash p^n$ by \cite[Lemma 2.5]{mattis2024unstable},
    and hence is $\AffSpc{1}$-invariant as a limit of $\AffSpc{1}$-invariant sheaves.
    Similarly, $L_\Q E$ can be written as a filtered colimit of copies of $E$ 
    by \cite[Lemma 2.6]{mattis2024fracture}.
    Hence, it is $\AffSpc{1}$-invariant since 
    $\AffSpc{1}$-invariant sheaves of spectra are stable under colimits by \cref{lem:etale-mot:A1-stable-under-filtered}.
    These two observations immediately imply the lemma.
\end{proof}

\begin{lem} \label{lem:etale-nilpotent:etale-nilpotent-Kn-A1}
    Let $S$ be a scheme.
    Let $X \in \ShvTopH{\et}{\Sm{S}}_*$ be étale $\AffSpc{1}$-nilpotent
    (resp.\ $p$-completely étale $\AffSpc{1}$-nilpotent, resp.\ rationally étale $\AffSpc{1}$-nilpotent),
    with sheaves $(X_n, K_n, E_n)$ as in \cref{def:etale-nilpotent:etale-nilpotent}.
    Then $K_n$ (resp.\ $L_p K_n$, resp.\ $L_\Q K_n$) is $\AffSpc{1}$-invariant.
\end{lem}
\begin{proof}
    If $X$ is étale $\AffSpc{1}$-nilpotent,
    then $E_n$ is $\AffSpc{1}$-invariant,
    and hence so is $K_n \cong \pLoop E_n$ by \cref{lem:etale-mot:SHone-adjunction}.

    If $X$ is $p$-completely $\AffSpc{1}$-nilpotent,
    then $L_p E_n$ is $\AffSpc{1}$-invariant.
    We have $L_p K_n \cong L_p \pLoop E_n \cong \tau_{\ge 1} \pLoop L_p E_n$ 
    by \cite[Lemma 3.17]{mattis2024unstable} (using that $E_n$ is $1$-connective).
    Hence, $L_p K_n$ is $\AffSpc{1}$-invariant by \cref{lem:etale-mot:SHone-adjunction,lem:etale-motives:connected-cover-invariant}.

    Similarly, if $X$ is rationally $\AffSpc{1}$-nilpotent,
    then $L_\Q E_n$ is $\AffSpc{1}$-invariant.
    We have $L_\Q K_n \cong L_\Q \pLoop E_n \cong \pLoop L_\Q E_n$ 
    by \cite[Lemma 3.18]{mattis2024fracture} (using that $E_n$ is $1$-connective).
    Hence, $L_\Q K_n$ is $\AffSpc{1}$-invariant by \cref{lem:etale-mot:SHone-adjunction}.
\end{proof}

\begin{lem} \label{lem:etale-nilpotent:nilpotent-A1-inv}
    Let $S$ be a qcqs scheme of finite Krull-dimension, that is moreover étale bounded.
    Let $X \in \ShvTopH{\et}{\Sm{S}}_*$ be étale $\AffSpc{1}$-nilpotent 
    (resp.\ $p$-completely étale $\AffSpc{1}$-nilpotent, resp.\ rationally étale $\AffSpc{1}$-nilpotent).
    Then $X$ (resp.\ $L_p X$, resp.\ $L_\Q X$) is $\AffSpc{1}$-invariant.
\end{lem}
\begin{proof}
    Let $\Cat L$ be one of $\id{}$, $L_p$ or $L_\Q$, 
    and suppose that $X$ is $\Cat L$-locally étale $\AffSpc{1}$-nilpotent.
    Choose sheaves $(X_n)_n$, $(K_n)_n$ and sheaves of spectra $(E_n)_n$ as in \cref{def:etale-nilpotent:etale-nilpotent}.
    All the sheaves $X_n$ are nilpotent:
    As $X_0 = *$ is nilpotent, this follows inductively,
    using the fiber sequence $X_{n+1} \to X_n \to K_n$ and
    \cite[Lemma A.12]{mattis2024unstable}, and the fact that 
    $K_n$ as a connected infinite loop sheaf is also nilpotent,
    cf.\ \cite[Lemma A.11]{mattis2024unstable}.

    Note that $\Cat L X \cong \limil{n} \Cat L X_n$:
    for $\Cat L = \id{}$ this is by definition, for $\Cat L = L_p$ 
    we use \cref{lem:p-comp:p-comp-tower-limit}, whereas for $\Cat L = L_\Q$
    we use \cite[Proposition 6.9]{mattis2024fracture} (for the latter two cases 
    we need that $\ShvTopH{\et}{\Sm{S}}$ admits a locally finite dimensional cover,
    which was shown in \cref{lem:etale:cover}).
    As $\AffSpc{1}$-invariant sheaves are stable under limits,
    it is sufficient to inductively show that $\Cat L X_n$ is $\AffSpc{1}$-invariant.
    For $n = 0$ this is clear as $\Cat L X_0 = \Cat L * \cong *$.
    For $n \ge 0$ consider the fiber sequence 
    $X_{n+1} \to X_n \to K_n$.
    Applying $\Cat L$ we get
    \begin{equation*}
        \Cat L X_{n+1} \cong \tau_{\ge 1} \Fib{\Cat L X_n \to \Cat L K_n},
    \end{equation*}
    this is clear if $\Cat L = \id{}$, and follows from 
    \cref{lem:fib-lem:main-thm} in the case of $\Cat L = L_p$,
    and from \cite[Lemma 3.13]{mattis2024fracture} in the case of $\Cat L = L_\Q$.
    As $\Cat L X_n$ is $\AffSpc{1}$-invariant by the induction hypothesis,
    and $\Cat L K_n$ is $\AffSpc{1}$-invariant by \cref{lem:etale-nilpotent:etale-nilpotent-Kn-A1},
    we get that also $\Cat L X_{n+1}$ is $\AffSpc{1}$-invariant (using again \cref{lem:etale-motives:connected-cover-invariant}).
\end{proof}

%% file: rational.tex
\section{Rational unstable étale motives}
In this section we prove that over a perfect field of finite cohomological dimension 
any nilpotent and $2$-effective motivic space satisfies étale hyperdescent.
\begin{prop} \label{lem:rational:transfers-etale-descent}
    Let $k$ be a perfect field, with $\cd{}{k} < \infty$.
    Let $E \in \SH{k}$ be rational.
    Then the Nisnevich sheaf of spectra $\omega^\infty E \in \SHone{k}$ 
    satisfies étale hyperdescent.
\end{prop}
\begin{proof}
    First note that there is an equivalence $\SH{S}_\Q \cong \mathcal D_{\AffSpc{1}}(S, \Q)$ 
    between the $\infty$-category of rational motivic spectra and Morel's $\AffSpc{1}$-derived category with rational coefficients,
    cf.\ \cite[§5.3.35]{cisinski2019triangulated}; we will use this in the following without mention.
    Note that there is a decomposition $E \cong E_+ \oplus E_-$ by \cite[§16.2.1]{cisinski2019triangulated}.
    It follows from \cite[Corollary 16.2.14]{cisinski2019triangulated} that $E_- \cong 0$ (as $\cd{}{k} < \infty$),
    and thus from \cite[Theorem 16.2.18]{cisinski2019triangulated} that $\omega^\infty E \cong \omega^\infty E_+$ satisfies étale hyperdescent
    (note that in the reference they use the word \emph{descent} for the word \emph{hyperdescent}).
\end{proof}

\begin{thm} \label{lem:rational:main-thm}
    Let $k$ be a perfect field, with $\cd{}{k} < \infty$.
    Let $X \in \Spc{k}_*$ be nilpotent, $2$-effective and rational.
    Then $X$ satisfies étale hyperdescent.
\end{thm}
\begin{proof}
    We use the version of the Postnikov-tower from \cref{lem:etale-nilpotent:effective-Postnikov-refinement},
    i.e.\ we have a tower of pointed motivic spaces $X_n \in \Spc{k}_*$ under $X$ with $X_0 = *$,
    and $2$-effective motivic spectra $E_n \in \SH{k}$,
    together with fiber sequences $X_{n+1} \to X_n \to \pLoopP E_n$,
    and an equivalence $X \cong \limil{n} X_n$.
    By \cref{lem:etale-nilpotent:effective-Postnikov-refinement-rational} we may further assume that all the $X_n$ and $E_n$ are rational.
    Thus, as étale hypersheaves are stable under limits, by induction it suffices to show 
    that $\pLoopP E_n$ satisfies étale hyperdescent for all $n$.
    This follows, as already $\omega^\infty E_n$ satisfies étale hyperdescent,
    cf.\ \cref{lem:rational:transfers-etale-descent}.
\end{proof}

%% file: at-char.tex
\section{Unstable étale motives at the characteristic}
In this section, we prove some vanishing results for $\AffSpc{1}$-invariant 
sheaves at the characteristic.

\begin{prop} \label{lem:at-char:p-completely-null}
    Let $p$ be a prime and $S$ be an $\finfld{p}$-scheme of finite Krull dimension which is moreover étale bounded.
    If $X \in \ShvTopH{\et}{\Sm{S}}$ 
    is $p$-completely étale $\AffSpc{1}$-nilpotent,
    then $L_p X \cong *$.
\end{prop}
\begin{proof}
    By definition of étale $\AffSpc{1}$-nilpotent 
    there is a sequence of nilpotent sheaves $X_n$ under $X$,
    with $X \cong \limil{n} X_n$,
    and fiber sequences $X_{n+1} \to X_n \to K_n$.
    Moreover, $K_n \cong \pLoop E_n$, where $E_n$ is a $1$-connective étale sheaf of spectra,
    such that $L_p E_n$ is $\AffSpc{1}$-invariant.
    It follows from \cite[Theorem A.1]{bachmann2021remarksetalemotivicstable} 
    that $L_p E_n \cong 0$, and hence $L_p K_n \cong \tau_{\ge 1} \pLoop L_p E_n \cong *$,
    where we used \cite[Lemma 3.17]{mattis2024unstable}.
    Using \cref{lem:fib-lem:main-thm} we see inductively
    \begin{equation*}
        L_p X_{n+1} \cong \tau_{\ge 1} \Fib{L_p X_n \to L_p K_n} \cong \tau_{\ge 1} \Fib{* \to *} \cong *.
    \end{equation*}
    Hence, using \cref{lem:p-comp:p-comp-tower-limit}, we also get $L_p X \cong \limil{n} \tau_{\le n} L_p X_n \cong *$,
    which proves the proposition.
\end{proof}

\begin{prop} \label{lem:at-char:HFp-module-case}
    Let $p$ be a prime and $k$ be a perfect field of characteristic $p$.
    Let $M \in \Mod{H\finfld{p}}{\ShvTopH{\et}{\Sm{S}, \Sp}}_{\ge 2}$ be a $2$-connective sheaf of $H\finfld{p}$-modules.
    Suppose that $\pLoop M$ is $\AffSpc{1}$-invariant.
    Then $\pLoop M \cong *$ and $M \cong 0$.
\end{prop}
\begin{proof}
    Since $M$ is connective, it is clear that $\pLoop M \cong *$ implies $M \cong 0$.
    Hence, it suffices to prove the first claim.
    Consider the Artin--Schreier sequence 
    \begin{equation*}
        H\finfld{p} \to H\Ga \xrightarrow{1 - F} H\Ga
    \end{equation*}
    as a fiber sequence in $\ShvTopH{\et}{\Sm{S}, \Sp}$.
    Here, we view $\Ga \coloneqq \AffSpc{1}$ as an étale sheaf of $\finfld{p}$-vectorspaces.
    Tensoring with $M$ over $H\finfld{p}$ gives a fiber sequence 
    \begin{equation*}
        M \to M \otimes_{H\finfld{p}} H\Ga \to M \otimes_{H\finfld{p}} H\Ga,
    \end{equation*}
    and by rotating twice and applying $\pLoop$ also a fiber sequence 
    \begin{equation*}
        \pLoop(\Omega M \otimes_{H\finfld{p}} H\Ga) \to \pLoop(\Omega M \otimes_{H\finfld{p}} H\Ga) \to \pLoop M.
    \end{equation*}
    We claim that applying $L_{\AffSpc{1},\nis}$ preserves this fiber sequence,
    i.e.\ that we get a fiber sequence 
    \begin{equation*}
        L_{\AffSpc{1},\nis} \pLoop(\Omega M \otimes_{H\finfld{p}} H\Ga) \to L_{\AffSpc{1},\nis} \pLoop(\Omega M \otimes_{H\finfld{p}} H\Ga) \to L_{\AffSpc{1},\nis} \pLoop M.
    \end{equation*}
    This follows from \cite[Theorem 2.3.3]{Asok2017simplicialsuspension} if we can show that $\pi_0^{\nis}(\pLoop M) \cong *$, and that $\pi_1^{\nis}(\pLoop M)$ 
    is strongly $\AffSpc{1}$-invariant. The second claim follows from a theorem of Morel \cite[Corollary 1.8]{bachmann2024stronglya1} 
    since $\pLoop M$ is assumed to be $\AffSpc{1}$-invariant and is even an étale sheaf.
    For the first claim, note that 
    $\pi_0^{\nis}(\pLoop M) \cong \pi_0^{\nis}(M) \cong L_{\nis} (U \mapsto \pi_0(M(U)))$.
    By Zariski descent it therefore suffices to show that $\pi_0(M(U)) \cong 0$ 
    for all \emph{affine} $U \in \Sm{k}$.
    There is a spectral sequence $E^{p,q}_2 = H^{p}_{\et}(U, \pi_{q}(M)) \implies \pi_{q-p}(M(U))$. Since $M$ is $2$-connective by assumption,
    it therefore suffices to show that $H^{p}_{\et}(U, \pi_q(M)) = 0$ for all $p \ge 2$.
    This is an immediate application of \cite[Théorème X.5.1]{SGA4} since $\pi_q(M)$ 
    is a sheaf of $\finfld{p}$-vectorspaces by assumption and $U$ was assumed to be affine.

    Now, since $\pLoop M$ is assumed to be $\AffSpc{1}$-invariant,
    and is moreover even an étale sheaf,
    we see that $\pLoop M \cong L_{\AffSpc{1},\nis} \pLoop M$.
    We claim that it suffices to show that $L_{\AffSpc{1},\nis} \pLoop(\Omega M \otimes_{H\finfld{p}} H\Ga) \cong *$.
    Indeed, we have seen above that $\pLoop M$ is connected as a Nisnevich sheaf,
    and hence the long exact sequence in Nisnevich homotopy sheaves 
    shows that $\pLoop M \cong *$.
    Since $\Ga$ admits the structure of a sheaf of $\finfld{p}$-algebras,
    we see that $H\Ga$ is an $\mathbb{E}_\infty$-algebra in $\Mod{H\finfld{p}}{\ShvTopH{\et}{\Sm{S},\Sp}}$.
    In particular, $\Omega M \otimes_{H\finfld{p}} H\Ga$ is canonically an $H\Ga$-module.
    Since $\pLoop$ and $L_{\AffSpc{1},\nis}$ are both lax symmetric monoidal,
    we see that $L_{\AffSpc{1},\nis} \pLoop(\Omega M \otimes_{H\finfld{p}} H\Ga)$ acquires the structure 
    of a module over the commutative monoid 
    $L_{\AffSpc{1},\nis} \pLoop(H\Ga) \cong L_{\AffSpc{1},\nis} \Ga \cong *$ 
    (using that the underlying sheaf of $\Ga$ is just $\AffSpc{1}$).
    Hence, $L_{\AffSpc{1},\nis} \pLoop(\Omega M \otimes_{H\finfld{p}} H\Ga) \cong *$.
\end{proof}

\begin{cor} \label{lem:at-char:HZ-module-case}
    Let $p$ be a prime and $k$ be a perfect field of characteristic $p$.
    Let $M \in \Mod{H\Z}{\ShvTopH{\et}{\Sm{S}, \Sp}}_{\ge 3}$ be a $3$-connective sheaf of $H\Z$-modules.
    Suppose that $L_p \pLoop M$ is $\AffSpc{1}$-invariant.
    Then $L_p \pLoop M \cong *$ and $L_p M \cong 0$.
\end{cor}
\begin{proof}
    Note that $M \sslash p$ is an $H\finfld{p}$-module (since we assume $M$ to be an $H\Z$-module).
    There is a fiber sequence $\Omega M \sslash p \to M \to M$, which is preserved by $\pLoop$.
    Applying \cref{lem:fib-lem:main-thm} we thus get an equivalence 
    \begin{equation*}
        L_p \pLoop \Omega M \sslash p \cong \tau_{\ge 1} \Fib{L_p \pLoop M \to L_p \pLoop M}.
    \end{equation*}
    Since $M \sslash p$ is an $H\finfld{p}$-module,
    we see that $\pLoop \Omega M \sslash p$ is already $p$-complete. Hence,
    \begin{equation*}
        \pLoop \Omega M \sslash p \cong \tau_{\ge 1} \Fib{L_p \pLoop M \to L_p \pLoop M}.
    \end{equation*}
    Since $L_p \pLoop M$ is assumed to be $\AffSpc{1}$-invariant,
    and $\AffSpc{1}$-invariant sheaves are stable under limits and connected covers (see \cref{lem:etale-motives:connected-cover-invariant} for the latter),
    we conclude that $\pLoop \Omega M \sslash p$ is $\AffSpc{1}$-invariant.
    Therefore, it follows from \cref{lem:at-char:HFp-module-case} that 
    $M \sslash p \cong 0$, and hence $L_p M \cong 0$.
    In particular, also $L_p \pLoop M \cong \tau_{\ge 1} \pLoop L_p M \cong *$,
    where we used \cite[Corollary 3.18]{mattis2024unstable}.
\end{proof}

\begin{cor} \label{lem:at-char:EM-vanishing}
    Let $p$ be a prime and $k$ be a perfect field of characteristic $p$.
    Suppose that $A \in \ShvTop{\et}{\Sm{k}, \Ab}$ is 
    a sheaf of abelian groups and $i \ge 3$ such that $L_p K(A, i)$ 
    is $\AffSpc{1}$-invariant.
    Then $L_p K(A, n) \cong *$ for all $n \ge 1$ and $L_p HA \cong 0$.
\end{cor}
\begin{proof}
    Note that $\Sigma^k HA$ is in the (shift of the) heart, and hence admits an $H\Z$-module structure.
    Moreover, $K(A, i) \cong \pLoop \Sigma^k HA$.
    Hence, we see from \cref{lem:at-char:HZ-module-case} that $L_p HA \cong 0$.
    Now, for any $n \ge 1$ we have 
    \begin{equation*}
        L_p K(A, n) \cong \tau_{\ge 1} \pLoop \Sigma^n L_p HA \cong *,
    \end{equation*}
    using \cite[Corollary 3.18]{mattis2024unstable}.
\end{proof}

\begin{rmk}
    Note that the \cref{lem:at-char:HFp-module-case} does not immediately follow from \cite[Theorem A.1]{bachmann2021remarksetalemotivicstable},
    as it is not clear that the $\AffSpc{1}$-invariance of $\pLoop M$ implies 
    the $\AffSpc{1}$-invariance of $M$, not even in the case that $M$ is $2$-connective 
    (\emph{a posteriori} it does as $M \cong 0$ in this case).
    A similar remark applies to \cref{lem:at-char:HZ-module-case,lem:at-char:EM-vanishing}.
\end{rmk}

%% file: rigidity.tex
\section{Proof of the rigidity theorem}
In this section we prove our main result, the unstable rigidity theorem.
Let $S$ be a qcqs scheme with finite Krull-dimension,
such that $S$ is étale bounded.
Moreover, let $p$ be a prime invertible on $S$.
\begin{lem} \label{lem:rigidity:iota-shriek-nilpotent}
    The functor $\iota_! \colon \ShvTopH{\et}{\smet{S}}_* \to \ShvTopH{\et}{\Sm{S}}_*$ from \cref{def:etale:small-to-big} 
    preserves nilpotent objects.
\end{lem}
\begin{proof}
    The statement is true for any geometric morphism,
    since they commute with taking homotopy objects \cite[Remark 6.5.1.4]{highertopoi}.
\end{proof}

\begin{lem} \label{lem:rigidity:p-complete-on-p-complete-obejcts}
    The adjunctions $\iota_! \dashv \iota^* \dashv \iota_*$ induce adjunctions 
    \begin{equation*}
        \iota_!^p \colon \complete{\ShvTopH{\et}{\smet{S}}_{*,}} \rightleftarrows \complete{\ShvTopH{\et}{\Sm{S}}_{*,}} \noloc \iota^*_p
    \end{equation*}
    and 
    \begin{equation*}
        \iota^*_p \colon \complete{\ShvTopH{\et}{\Sm{S}}_{*,}} \rightleftarrows \complete{\ShvTopH{\et}{\smet{S}}_{*,}} \noloc \iota_*^p.
    \end{equation*}
    Here, $\iota^*_p$ and $\iota_*^p$ are just given by restriction of $\iota^*$ and $\iota_*$, respectively.
\end{lem}
\begin{proof}
    Note that $\iota^*$ restricts to a functor 
    $\iota^*_p \colon \complete{\ShvTopH{\et}{\Sm{S}}_{*,}} \to \complete{\ShvTopH{\et}{\smet{S}}_{*,}}$,
    since it preserves $p$-complete objects as a right adjoint (this follows formally 
    from \cite[Lemma 3.11]{mattis2024unstable}).
    Similarly, $\iota_*$ restricts to a functor 
    $\iota_*^p \colon \complete{\ShvTopH{\et}{\smet{S}}_{*,}} \to \complete{\ShvTopH{\et}{\Sm{S}}_{*,}}$.
    Write $\iota_!^p$ for the composition $L_p \circ \iota_!|_{\complete{\ShvTopH{\et}{\smet{S}}_{*,}}}$.
    It is formal that this gives adjunctions $\iota_!^p \dashv \iota^*_p \dashv \iota_*^p$.
\end{proof}

\begin{lem} \label{lem:rigidity:iota-shriek-p-fully-faithful}
    The functors $\iota_!^p$ and $\iota_*^p$ are fully faithful.
\end{lem}
\begin{proof}
    It suffices to show that the double right adjoint $\iota_*^p$ is fully faithful (cf.\ (the dual of) 
    \cite[Corollary 2.7]{naumann2024symmetricmonoidalfracturesquare}).
    This is clear since $\iota_*^p$ is just the restriction of the fully 
    faithful functor $\iota_*$, cf.\ \cref{lem:rigidity:iota-fully-faithful}.
\end{proof}

\begin{prop} \label{lem:rigidity:factors-through}
    Let $X \in \ShvTopH{\et}{\smet{S}}_*$ be nilpotent.
    Then $\iota_! X$ is $p$-completely étale $\AffSpc{1}$-nilpotent
    (and in particular $L_p \iota_! X$ is $\AffSpc{1}$-invariant).
\end{prop}
\begin{proof}
    Since $X$ is nilpotent, we can choose a principal refinement $(X_n, K_n)_n$ of the Postnikov-tower of $X$
    with $K_n \cong \pLoop E_n$ for some $2$-connective sheaf of spectra $E_n$ \cite[Lemma A.15]{mattis2024unstable}.
    Then $(\iota_! X_n, \iota_! K_n)_n$ is a refinement of the Postnikov tower of $\iota_! X$.
    Indeed, it is clear that this gives a refinement of the Postnikov tower of $\limil{n} \iota_! X_n$.
    Moreover, $\iota_!$ commutes with this limit by \cref{lem:highly-connected:main-thm},
    as $\ShvTopH{\et}{\smet{S}}$ is locally of homotopy dimension $\le N$ for some $N$,
    cf.\ \cref{lem:etale:finite-htyp-dim},
    and $\ShvTopH{\et}{\Sm{S}}$ has a locally finite dimensional cover, cf.\ \cref{lem:etale:cover}.
    Note that $\iota_! K_n \cong \pLoop \iota_! E_n$.
    It therefore suffices to show that $L_p \iota_! E_n$
    is $\AffSpc{1}$-invariant,
    which follows from the proof of \cite[Corollary 6.2]{Bachmann2021etalerigidity},
    since $L_p \iota_! E_n \cong L_p \iota_! L_p E_n$,
    as $\iota^*$ preserve $p$-equivalences.

    That $L_p \iota_! X$ is $\AffSpc{1}$-invariant now follows from \cref{lem:etale-nilpotent:nilpotent-A1-inv}.
\end{proof}

\begin{defn}
    Let $X \in \ShvTopH{\et}{\Sm{S}}_*$.
    We say that $X$ is \emph{$p$-completely small} 
    if the counit $\iota_! \iota^* X \to X$ is a $p$-equivalence,
    i.e.\ if and only if $L_p \iota_! \iota^* X \to L_p X$ is an equivalence.
    In particular, in this case $L_p X$ is in the essential image of $\iota_!^p$.
\end{defn}

\begin{thm}[Essential Image] \label{lem:rigidity:ess-img}
    Let $W \in \ShvTopH{\et}{\Sm{S}}_*$ be $p$-completely étale $\AffSpc{1}$-nilpotent.
    Then $W$ is $p$-completely small.
    In particular, $L_p W$ is in the essential image of $\iota_!^p$, restricted to 
    the subcategory of $p$-complete nilpotent objects.    
\end{thm}
\begin{proof}
    The ``in particular'' part follows as both $\iota_!$ and $\iota^*$ 
    preserve $p$-equivalences (since they are left adjoints),
    and hence $L_p \iota_! \iota^* \cong \iota_!^p \iota^*_p L_p$.
    By assumption there is an $\N$-indexed system $(W_n)_n$ under $W$
    with $\limil{n} W_n \cong W$, $W_0 \cong *$ and fiber sequences $W_{n+1} \to W_n \to K_n$
    for some sheaves $K_n$ 
    such that $K_n \cong \pLoop E_n$,
    where $E_n \in \Stab{\ShvTopH{\et}{\Sm{S}}}_{\ge 1}$ is a $1$-connective sheaf of spectra,
    such that $L_p E_n$ is $\AffSpc{1}$-invariant.
    We proceed in several steps:
    
    \underbar{\textbf{Step 1:}} We prove that $L_p K_n \cong L_p \iota_! \iota^* K_n$.
    Since $E_n$ is $1$-connective, it follows from \cite[Lemma 3.17]{mattis2024unstable} that $L_p K_n \cong \tau_{\ge 1} \pLoop L_p E_n$.
    Moreover, it follows from \cite[Theorem 6.6]{Bachmann2021etalerigidity} 
    that $L_p E_n \cong L_p \iota_! \iota^* E_n$ (this is the stable rigidity theorem,
    which needs our assumptions on $S$. In the reference it is shown that $\iota_!$ induces an equivalence 
    on $p$-complete objects, and hence the above claim follows since both $\iota_!$ and its right adjoint $\iota^*$ preserve 
    $p$-equivalences).
    Therefore, we get 
    \begin{align*}
        L_p K_n 
        &\cong \tau_{\ge 1} \pLoop L_p E_n \\
        &\cong \tau_{\ge 1} \pLoop L_p \iota_! \iota^* E_n \\
        &\cong L_p \pLoop \iota_! \iota^* E_n \\
        &\cong L_p \iota_! \iota^* \pLoop E_n \\
        &\cong L_p \iota_! \iota^* K_n,
    \end{align*}
    where we used again \cite[Lemma 3.17]{mattis2024unstable} in the third equivalence (as also $\iota_! \iota^* E_n$ 
    is $1$-connective as both functors are t-exact, cf.\ \cite[Remark 1.3.2.8]{sag}),
    and e.g.\ \cite[Lemmas A.1 and A.3]{mattis2024unstable} in the fourth equivalence.
    
    \underbar{\textbf{Step 2:}} We prove that $L_p W_n \cong L_p \iota_! \iota^* W_n$.
    We prove this by induction on $n$, using that for $W_0 = *$ the statement is true 
    since all involved functors preserve the terminal object.
    Suppose now that $L_p W_n \cong L_p \iota_! \iota^* W_n$.
    We have 
    \begin{align*}
        L_p \iota_! \iota^* W_{n+1} 
        &\cong L_p \iota_! \iota^* \Fib{W_n \to K_n} \\
        &\cong L_p \Fib{\iota_! \iota^* W_n \to \iota_! \iota^* K_n} \\
        &\cong \tau_{\ge 1} \Fib{L_p \iota_! \iota^* W_n \to L_p \iota_! \iota^* K_n},
    \end{align*}
    using in the second equivalence that both $\iota_!$ and $\iota^*$ preserve finite limits,
    as both are (the left adjoints of) geometric morphisms.
    The last equivalence holds by \cref{lem:fib-lem:main-thm}, since $\Fib{\iota_! \iota^* W_n \to \iota_! \iota^* K_n} \cong \iota_! \iota^* W_{n+1}$
    is connected.
    Analogously, one shows $L_p W_{n+1} \cong \tau_{\ge 1} \Fib{L_p W_n \to L_p K_n}$.
    Thus, we conclude using the induction hypothesis and the first step.

    \underbar{\textbf{Step 3:}} We prove that $L_p W \cong L_p \iota_! \iota^* W$.
    We calculate 
    \begin{align*}
        L_p \iota_! \iota^* W 
        &\cong L_p \iota_! \iota^* \limil{n} W_n \\
        &\cong L_p \iota_! \limil{n} \iota^* W_n \\
        &\cong L_p \limil{n} \iota_! \iota^* W_n \\
        &\cong \limil{n} L_p \iota_! \iota^* W_n \\
        &\cong \limil{n} L_p W_n \\
        &\cong L_p W \rlap{.}
    \end{align*}
    The first equivalence holds by definition of $W$.
    The functor $\iota^*$ commutes with the limit since it is a right adjoint,
    whereas $\iota_!$ preserves the limit by \cref{lem:highly-connected:main-thm},
    using that $(W_n)_n$ and hence also $(\iota^* W_n)_n$
    are highly connected towers.
    The fourth and last equivalence are applications of \cref{lem:p-comp:p-comp-tower-limit},
    using that $\ShvTopH{\et}{\Sm{S}}$ admits a locally finite-dimensional cover by 
    \cref{lem:etale:cover}.
    The fifth equivalence is given by the second step.
\end{proof}

Combining the above results we get the unstable rigidity theorem.
\begin{cor}[Étale rigidity for motivic spaces] \label{lem:rigidity:main-thm}
    The functor 
    \begin{equation*}
        \iota_!^p \colon \compnil{\ShvTopH{\et}{\smet{S}}_{*,}} \to \complete{\ShvTopH{\et}{\Sm{S}}_{*,}}
    \end{equation*}
    is fully faithful,
    with essential image those $p$-complete $\AffSpc{1}$-invariant sheaves
    that are the $p$-completion of a $p$-completely étale $\AffSpc{1}$-nilpotent
    sheaf.
\end{cor}
\begin{proof}
    We have seen in \cref{lem:rigidity:iota-shriek-p-fully-faithful} that the functor $\iota_!^p$ is fully faithful.
    Note that by \cref{lem:rigidity:factors-through} it factors through the proclaimed essential image.
    Hence, the corollary follows from \cref{lem:rigidity:ess-img}.
\end{proof}

%% file: retract.tex
\section{A retract of projective space}
In this section let $k$ be an algebraically closed field and $p \neq \operatorname{char}(k)$
a prime number.
We prove that in étale $\AffSpc{1}$-homotopy theory 
over $k$ there is a retract diagram $L_p S^2 \to L_p \Pp^1 \to L_p S^2$.
Choose a compatible system of primitive roots of unity $(\zeta_{p^n})_n$,
i.e.\ a sequence with $\zeta_{p^n} \in k^\times$ a primitive $p^n$-th root of unity
such that moreover $(\zeta_{p^{n+1}})^p = \zeta_{p^n}$ for all $n$.
In particular, $(\zeta_{p^n})_n$
defines an element of $\Z_p(1)(k) = \limil{n} k^\times[p^n]$.

Recall that in \cite[Section 3]{Bachmann2021etalerigidity} 
Bachmann constructed (the desuspension of) a map $\sigma \colon L_p L_{\et, \AffSpc{1}} \Pp^1 \to L_p S^2$,
and its stabilization $\sigma_{st} \colon L_p L_{\et, \AffSpc{1}} \Sus \Pp^1 \to L_p \S^2 \in \SHoneet{k}$
(since $k$ is algebraically closed, we can choose a trivialization of the twisting 
space, cf.\ \cite[Theorem 3.6 (3)]{Bachmann2021etalerigidity}).

Our first goal is the construction of the other map of the retract.
For this, we will work in the Nisnevich topology.
We write $L_{\nis, \AffSpc{1}} \colon \Cat P(\Sm{S}) \to \Spc{k}$ 
for the motivic localization functor, and 
\begin{equation*}
    H_{\nis} \colon \ShvTop{\nis}{\Sm{k}, \Ab} \cong \Stab{\ShvTop{\nis}{\Sm{k}}}^{\heartsuit} \to \Stab{\ShvTop{\nis}{\Sm{k}}}
\end{equation*}
for the embedding of the heart.
Similarly, for every $n \ge 1$ we write 
\begin{equation*}
    K_{\nis}(-, n) \colon \ShvTop{\nis}{\Sm{k}, \Ab} \to \ShvTop{\nis}{\Sm{k}}
\end{equation*}
for the Nisnevich-local Eilenberg-MacLane sheaves.
We begin by constructing an element $\tau_0 \in \pi_2(L_p K_{\nis}(\Gm, 1))(k)$:
\begin{constr}
    Write $A[p^n] \coloneqq \ker(A \xrightarrow{p^n} A)$ for the $p^n$-torsion of 
    a sheaf of abelian groups $A$.
    Using \cite[Corollary 3.18 and Lemma 2.24]{mattis2024unstable}, we have 
    \begin{equation*}
        \pi_2 (L_p K_{\nis}(\Gm, 1)) \cong \pi_2 (L_p \Sigma^1 H_{\nis}\Gm) \cong \limil{n} \Gm[p^n] = \limil{n} \mu_{p^n} = \Z_p(1),
    \end{equation*}
    and therefore 
    \begin{equation*}
        \pi_2(L_p K_{\nis}(\Gm, 1))(k) \cong \Z_p(1)(k) = \limil{n} \mu_{p^n}(k).
    \end{equation*}
    Under this equivalence, we let $\tau_0 \in \pi_2(L_p K_{\nis}(\Gm, 1))(k)$ be the element given by 
    the system of roots of unity $(\zeta_{p^n})_n$.
\end{constr}

\begin{lem}
    There is a map $\psi \colon \pi_2 L_p L_{\nis,\AffSpc{1}}\Pp^1 \to \pi_2 L_p K_{\nis}(\Gm, 1)$
    such that $\tau_0$ is in the image of this map.
\end{lem}
\begin{proof}
    We will show that there exists a Nisnevich sheaf $F \in \ShvTop{\nis}{\Sm{k}}_*$ 
    that fits into an exact sequence of Nisnevich sheaves of groups 
    \begin{equation*}
        \pi_2 L_p L_{\nis,\AffSpc{1}}\Pp^1 \xrightarrow{\psi} \pi_2 L_p K_{\nis}(\Gm, 1) \xrightarrow{\phi} \pi_1 L_p F
    \end{equation*}
    such that $\phi(\tau_0) = 0$.
    Recall from the discussion before \cite[Theorem 6.29]{morel2012a1} that there is a short exact sequence of Nisnevich sheaves of groups
    \begin{equation*}
        0 \to \mathbf{K}^{\mathrm{MW}}_2 \to \pi_1(L_{\nis,\AffSpc{1}} \Pp^1) \to \Gm \to 0.
    \end{equation*}
    Write $F \coloneqq \Fib{L_{\nis,\AffSpc{1}} \Pp^1 \to K_{\nis}(\Gm, 1)}$ for the fiber of 
    the canonical map $L_{\nis,\AffSpc{1}} \Pp^1 \to K_{\nis}(\pi_1(L_{\nis,\AffSpc{1}} \Pp^1), 1) \to K_{\nis}(\Gm, 1)$ induced by the right map 
    in the above short exact sequence.
    Note that $F$ is connected as $\pi_1(L_{\nis,\AffSpc{1}} \Pp^1) \to \Gm$ is surjective and $L_{\nis,\AffSpc{1}} \Pp^1 \cong L_{\nis,\AffSpc{1}} \Sigma \Gm$ is connected.
    As $\pi_2 K_{\nis}(\Gm, 1) = 0$, we see that $\pi_1(F) \cong \mathbf{K}^{\mathrm{MW}}_2$.

    Since $L_{\nis,\AffSpc{1}} \Pp^1$ and $K_{\nis}(\Gm, 1)$ are nilpotent by \cite[Remark 3.4.9]{asok2022localization} and \cite[Lemma A.11]{mattis2024unstable},
    it follows from \cref{lem:fib-lem:main-thm} that after $p$-completing, we still have
    \begin{equation*}
        L_p F \cong \tau_{\ge 1} \Fib{L_p L_{\nis,\AffSpc{1}} \Pp^1 \to L_p K_{\nis}(\Gm, 1)}.
    \end{equation*}
    Thus, the long exact sequence in homotopy yields the wanted exact sequence.

    It is left to show that $\phi(\tau_0) = 0$.
    As $\phi(\tau_0) \in (\pi_1 L_p F)(k)$, it clearly suffices to show 
    that this group is $0$.
    We have 
    \begin{equation*}
        (\pi_1 L_p F)(k) \cong \pi_1 L_p (F(k)) \cong \mathbb{L}_0 \pi_1 (F(k)) \cong \mathbb{L}_0 ((\pi_1 F)(k))
    \end{equation*}
    by \cref{lem:etale:small-topos-over-strictly-local} and \cite[Theorem 11.1.2]{MCAT},
    here $\mathbb{L}_0$ denotes the zeroth derived $p$-completion functor 
    on the category of groups as defined e.g.\ in \cite[Chapter 10.4]{MCAT} (where it is called $\mathbb{E}_p$).
    Now note that $(\pi_1 F)(k) = \mathbf{K}_2^{\mathrm{MW}}(k)$ is uniquely $p$-divisible 
    as $k$ is algebraically closed, for this combine e.g.\ \cite[Example 2.3.11]{deglise2023notes} with \cite[Chapter I, Corollary 1.3]{bass2006milnor}.
    This immediately implies that $\mathbb{L}_0 (\pi_1 F)(k) \cong 0$,
    cf.\ \cite[Proposition 10.4.7 (iii)]{MCAT}, which is what we wanted to show.
\end{proof}

\begin{constr}
    Hence, we can choose an element $\tau \in (\pi_2 L_p L_{\nis,\AffSpc{1}} \Pp^1)(k)$ such that
    $\psi(\tau) = \tau_0$.
    Therefore, $\tau$ defines a map $S^2 \to L_p L_{\nis,\AffSpc{1}} \Pp^1$,
    and thus also a map 
    \begin{equation*}
        \tau \colon L_p S^2 \to L_p L_{\nis,\AffSpc{1}} \Pp^1,
    \end{equation*}
    which we will also denote by $\tau$.
\end{constr}
Bachmann constructed in \cite[Lemma 6.4]{Bachmann2021etalerigidity} (the desuspension of) a map $\tau_{st} \colon L_p \Sigma^2 \S \to L_p L_{\nis,\AffSpc{1}} \Sus \Pp^1$
in $\SHone{k}$. His construction depends on a variety of choices, but he then shows that independently of these choices,
$\tau_{st}$ is an equivalence.
We will show that our map $\tau$ stabilizes to one of these maps $\tau_{st}$.
For this, let us make the following definition.
\begin{defn}
    Let $f \colon L_p \Sigma^2 \S \to L_p L_{\nis,\AffSpc{1}} \Sus \Pp^1 \in \SHone{k}$ be a map,
    or equivalently by adjunction (and up to homotopy), $f \in \pi_2(L_p L_{\nis,\AffSpc{1}} \Sus \Pp^1)(k)$ be a homotopy class.
    We say that $f$ is \emph{of stable $\tau$-type} 
    if the image of $f$ under the chain of maps 
    \begin{equation} \label{eq:retracts:map-to-MW}
        \pi_2(L_p L_{\nis,\AffSpc{1}} \Sus \Pp^1)(k) \to \pi_2((L_{\nis,\AffSpc{1}} \Sus \Pp^1) \sslash p)(k) \to \pi_1(L_{\nis,\AffSpc{1}} \Sus \Pp^1)(k) \cong \mathbf{K}^{\mathrm{MW}}_1(k)
    \end{equation}
    is the element $[\zeta_p]$ corresponding to our chosen primitive $p$-th root of unity.
    Here, the first map is the projection out of the limit,
    the second map is the boundary map of the cofiber sequence 
    $L_{\nis,\AffSpc{1}} \Sus \Pp^1 \xrightarrow{p} L_{\nis,\AffSpc{1}} \Sus \Pp^1 \to (L_{\nis,\AffSpc{1}} \Sus \Pp^1) \sslash p$, and the last isomorphism is \cite[Corollary 25]{morel2012a1}.
\end{defn}
We chose this definition because of the following result.
\begin{prop} \label{lem:retract:tau-stable}
    We have that $L_p \Sus \tau \colon L_p \S^2 \to L_p L_{\nis,\AffSpc{1}} \Sus \Pp^1 \in \SHone{k}$ 
    is of stable $\tau$-type.
\end{prop}
\begin{proof}
    By adjunction, $L_p \Sus \tau$ corresponds to a 
    $\tilde{\tau} \in \pi_2(L_p L_{\nis,\AffSpc{1}} \Sus \Pp^1)(k)$.
    We have to show that the image of $\tilde{\tau}$ 
    in $\mathbf{K}^{\mathrm{MW}}_1(k)$ under the map \eqref{eq:retracts:map-to-MW}
    is $[\zeta_p]$. For this, consider the following commutative diagram:
    \begin{center}
        \begin{tikzcd}
            \pi_2(L_p L_{\nis,\AffSpc{1}}\Pp^1)(k) \ar[r, blue] \ar[d, red] &\pi_2 (L_pK_{\nis}(\Gm, 1))(k) \ar[rd, bend left=20, blue, "\cong"] \ar[d, "\cong"] \\
            \pi_2(L_p L_{\nis,\AffSpc{1}}\Sus \Pp^1)(k) \ar[r] \ar[d, red] &\pi_2 (L_p\Sigma H_{\nis}\Gm)(k) \ar[r, "\cong"] \ar[d] &\Z(1)(k) \ar[d, blue] \\
            \pi_2(L_{\nis,\AffSpc{1}}\Sus \Pp^1 \sslash p)(k) \ar[r] \ar[d, red] &\pi_2(\Sigma H_{\nis}\Gm \sslash p)(k) \ar[r, "\cong"] \ar[d] &\mu_p(k) \ar[ddl, blue, bend left] \\
            \pi_1(L_{\nis,\AffSpc{1}}\Sus \Pp^1)(k) \ar[r] \ar[d, red, "\cong"] &\pi_1(\Sigma H_{\nis}\Gm)(k) \ar[d, "\cong"] & \\
            \mathbf{K}^{\mathrm{MW}}_1(k) \ar[r, red, "\cong"] &\Gm(k)\rlap{.}
        \end{tikzcd}
    \end{center}
    Here, the black horizontal maps on the left are induced by the map 
    \begin{equation*}
        L_{\nis,\AffSpc{1}} \Sus \Pp^1 \to \Sus K_{\nis}(\Gm, 1) \to \tau_{\le 1} \Sus K_{\nis}(\Gm, 1) \cong \Sigma H_{\nis}\Gm,
    \end{equation*}
    where the last equivalence follows from the Freudenthal suspension theorem for $\infty$-topoi,
    cf.\ \cite[Corollary 4.16]{DevalapurkarHaine}.
    Note that the bottom \textcolor{red}{red} arrow is an equivalence 
    as $k$ is algebraically closed (see e.g.\ the short exact sequence in \cite[Corollary 2.3.10]{deglise2023notes}
    and the fact that the fundamental ideal $I(k) = \ker(\mathrm{GW}(k) \to \Z)$ is zero if $k$ is algebraically closed,
    as then $\mathrm{GW}(k) \cong \Z$).
    By definition of $\tau$ the \textcolor{blue}{blue} composition maps $\tau$ to $\zeta_p$,
    which is exactly what we wanted to show for the \textcolor{red}{red} composition.
\end{proof}

Now, in view of the construction of Bachmann's $\tau_{st}$ \cite[Lemma 6.4]{Bachmann2021etalerigidity}
and the proof of \cite[Theorem 6.5]{Bachmann2021etalerigidity},
we have the following result.
\begin{thm}[Bachmann] \label{lem:retract:stable-tau-type}
    Let $f \colon L_p \Sigma^2 \S \to L_p L_{\nis,\AffSpc{1}} \Sus \Pp^1 \in \SHone{k}$ be a map of stable $\tau$-type.
    Then $L_p L_{\et,\AffSpc{1}} (\sigma_{st} f)$ is an automorphism.
\end{thm}

\begin{prop} \label{lem:retract:pi_2_S_2}
    The ($p$-completed) suspension spectrum functor 
    \begin{equation*}
        L_p \Sus \colon \complete{\ShvTopH{\et}{\Sm{S}}} \to \complete{\Stab{\ShvTopH{\et}{\Sm{S}}}}
    \end{equation*}
    induces an isomorphism of rings 
    \begin{equation*}
        \pi_0 \Map{}{L_p S^2}{L_p S^2} \to \pi_0 \Map{}{L_p \Sigma^2 \S}{L_p \Sigma^2 \S}.
    \end{equation*}
\end{prop}
\begin{proof}
    We first show that both sides are equivalent to $\Z_p$ (as groups).
    Note that by adjunction, we have 
    $\pi_0 \Map{}{L_p S^2}{L_p S^2} \cong \pi_0 \Map{}{S^2}{L_p S^2} \cong \pi_2(L_p S^2)(k)$.
    Now, from \cref{lem:rigidtiy:stalk-functor} we see that evaluation at $k$ is a limit-preserving geometric morphism,
    and in particular commutes with homotopy objects and $p$-completion.
    Hence, the above is equivalent to $\pi_2L_p (S^2(k)) = \pi_2 (L_p S_{\mathrm{top}}^2) \cong \Z_p$.

    A similar calculation shows that the right-hand side is also $\Z_p$.

    Since both abelian groups are endowed with ring structures (via composition),
    and since the group $\Z_p$ has a unique ring structure (after a choice of unit $1 \in \Z_p$ which is given by the identity morphism),
    it follows that both groups are equivalent to $\Z_p$ as rings.
    Since the map is induced by a functor, it is clear that it is a map of rings.
    But now the only ring map $\Z_p \to \Z_p$ is the identity.
\end{proof}

From now on we will also write $\tau \colon L_p S^2 \to L_p L_{\et,\AffSpc{1}} \Pp^1$
for $L_{\et, \AffSpc{1}} \tau$.
\begin{thm} \label{lem:retract:main-thm}
    The maps $\tau \colon L_p S^2 \to L_p L_{\et,\AffSpc{1}} \Pp^1$
    and $\sigma \colon L_p L_{\et,\AffSpc{1}} \Pp^1 \to L_p S^2$
    form a retract (up to precomposing $\tau$ with an automorphism of $L_p S^2$), i.e.\ there is a homotopy $\sigma \tau \cong \id{L_p S^2}$.
\end{thm}
\begin{proof}
    By \cref{lem:retract:pi_2_S_2}, the suspension spectrum functor induces an isomorphism of rings 
    \begin{equation*}
        \pi_0 \Map{}{L_p S^2}{L_p S^2} \to \pi_0 \Map{}{L_p \Sigma^2 \S}{L_p \Sigma^2 \S}.
    \end{equation*} 
    As by definition $L_p \Sus \sigma \cong \sigma_{st}$, and $L_p \Sus \tau$ 
    is of stable $\tau$-type by \cref{lem:retract:tau-stable},
    we see that $L_p \Sus (\sigma \tau)$ is an automorphism, cf.\ \cref{lem:retract:stable-tau-type}. 
    Hence, in view of the ring isomorphism above, we see that also $\sigma \tau$ is an automorphism.
    Therefore, up to replacing $\tau$ by $\tau (\sigma \tau)^{-1}$,
    we see that $\tau$ and $\sigma$ form a retract.
\end{proof}

Actually, the retract constructed above is an equivalence after two suspensions.
For this, we need the following simple lemma:
\begin{lem} \label{lem:retract:swap-identity}
    The swap endomorphism 
    \begin{equation*}
        \operatorname{swap} \colon L_{\nis,\AffSpc{1}} \Pp^1 \wedge L_{\nis,\AffSpc{1}} \Pp^1 \to L_{\nis,\AffSpc{1}} \Pp^1 \wedge L_{\nis,\AffSpc{1}} \Pp^1
    \end{equation*}
    is homotopic to the identity.
\end{lem}
\begin{proof}
    For readability, we will exclude the functor $L_{\nis,\AffSpc{1}}$ from the notation.
    The symmetric monoidal suspension spectrum functor $\SusP \colon \Spc{k}_* \to \SH{k}$
    induces a morphism on homotopy classes of maps 
    \begin{equation*}
        \Phi \colon [\Pp^1 \wedge \Pp^1, \Pp^1 \wedge \Pp^1] \to [\SusP \Pp^1 \otimes \SusP \Pp^1, \SusP \Pp^1 \otimes \SusP \Pp^1]
    \end{equation*}
    which maps the swap morphism to the swap morphism.
    It therefore suffices to show that $\Phi$ is injective,
    and the stable swap morphism is homotopic to the identity.
    Recall the computations of Morel of $[(\Pp^1)^{\wedge n}, (\Pp^1)^{\wedge n}] \cong \mathbf{K}_0^{\mathrm{MW}}(k) = \operatorname{GW}(k)$ for all $n \ge 2$,
    cf.\ \cite[Corollary 24]{morel2012a1}.
    Here, $\operatorname{GW}(k)$ is the Grothendieck-Witt ring of $k$.
    This then stabilizes to show that $\Phi$ is in fact an isomorphism.
    In order to prove the second claim, we proceed as follows:
    under the equivalence
    \begin{equation*}
        [\SusP \Pp^1 \otimes \SusP \Pp^1, \SusP \Pp^1 \otimes \SusP \Pp^1] \cong \operatorname{GW}(k),
    \end{equation*}
    the swap morphism corresponds to the element $\langle-1\rangle \in \operatorname{GW}(k)$,
    for this combine the proof of \cite[Lemma 6.1.1 (2)]{morel2003trieste} 
    with \cite[Remark 6.3.5]{morel2003trieste}.
    As $k$ is algebraically closed, in particular $-1 \in k^{\times}$ is a square,
    and we see that $\langle-1\rangle = 1 \in \operatorname{GW}(k)$. 
\end{proof}

\begin{thm} \label{lem:retract:double-suspension}
    The map 
    \begin{equation*}
        L_p L_{\et,\AffSpc{1}} \Sigma^2 \tau \colon L_p L_{\et,\AffSpc{1}} S^4 \cong L_p L_{\et,\AffSpc{1}} \Sigma^2 S^2 \to L_p L_{\et,\AffSpc{1}} \Sigma^2 \Pp^1
    \end{equation*}
    is an equivalence.
\end{thm}
\begin{proof}
    Consider the following diagram where we omit $p$-completions and étale-$\AffSpc{1}$-localizations:
    \begin{center}
        \begin{tikzcd}
            S^2 \wedge \Pp^1 \ar[r, red, "\tau \wedge \Pp^1"] \ar[d, blue, "S^2 \wedge \sigma"] & \Pp^1 \wedge \Pp^1 \ar[d, "\Pp^1 \wedge \sigma"] \ar[r, red, "\operatorname{swap}"] &\Pp^1 \wedge \Pp^1 \ar[d, red, "\sigma \wedge \Pp^1"]\\
            S^2 \wedge S^2 \ar[r, "\tau \wedge S^2"] \ar[d, blue, "\operatorname{swap}"] & \Pp^1 \wedge S^2 \ar[d, "\operatorname{swap}"] \ar[r, "\operatorname{swap}"] & S^2 \wedge \Pp^1 \\
            S^2 \wedge S^2 \ar[r, blue, "S^2 \wedge \tau"] & S^2 \wedge \Pp^1 \ar[ur, red, equal] \rlap{.}
        \end{tikzcd}
    \end{center}
    Here, the swap morphisms $\operatorname{swap} \colon S^2 \wedge S^2 \to S^2 \wedge S^2$ and $\operatorname{swap} \colon \Pp^1 \wedge \Pp^1 \to \Pp^1 \wedge \Pp^1$
    are homotopic to the identity,
    for the first claim this holds since
    \begin{equation*}
        \det \left(\begin{matrix}
            0 &0 &1 &0 \\
            0 &0 &0 &1 \\
            1 &0 &0 &0 \\
            0 &1 &0 &0 \\
        \end{matrix}\right) = 1
    \end{equation*}
    and for the second claim see \cref{lem:retract:swap-identity}
    (using that $L_{\et, \AffSpc{1}}$ and $L_p$ are symmetric monoidal).
    Note that the diagram commutes: 
    the top left square by functoriality of $\wedge$, the bottom left and top right square since $\wedge$ is symmetric monoidal,
    and the bottom right triangle is trivially commutative.
    Since $\sigma \tau = \id{}$ by \cref{lem:retract:main-thm},
    it follows that the \textcolor{red}{red} composition on the top and right is the identity.
    Hence, also the \textcolor{blue}{blue} composition on the left and bottom is the identity,
    proving that $S^2 \wedge \sigma$ is a two-sided inverse of $S^2 \wedge \tau$.
\end{proof} 

\begin{rmk}
    It is unknown to the author if already $\tau$ is an equivalence.
    On the other hand, the retract only exists on $S^2$, i.e.\ 
    there is no retract $L_p S^1 \to L_p \Gm \to L_p S^1$.
    Indeed, as $S^1$ is connected, so is $L_p S^1$, and in fact one can show that 
    $L_p S^1 \cong K(\Z_p, 1)$.
    On the other hand, since $\Gm$ is $0$-truncated,
    we get that $L_p \Gm = \Gm$, cf.\ \cite[Lemma 3.13]{mattis2024unstable}.
    In particular, $L_p \Gm$ is $0$-truncated and hence $L_p S^1$ cannot be a retract of it.
\end{rmk}

%% file: examples.tex
\section{Examples of étale \texorpdfstring{$\AffSpc{1}$}{A1}-nilpotent sheaves}
In this section, we will give examples of $p$-completely étale $\AffSpc{1}$-nilpotent sheaves.
We use those to prove our main theorem, \cref{lem:intro:main-thm-cor}, in \cref{lem:etale-nilpotent:4-connective-main-thm}.

\begin{lem} \label{lem:etale-nilpotent:check-mod-p}
    Let $S$ be a scheme and $E \in \Sp(\ShvTopH{\et}{\Sm{S}})$.
    If $L_\Q E$ is $\AffSpc{1}$-invariant,
    and $E \sslash p$ is $\AffSpc{1}$-invariant for all primes $p$,
    then also $E$ is $\AffSpc{1}$-invariant.
\end{lem}
\begin{proof}
    We first show that for every prime $p$, the $p$-completion $L_p E$ is $\AffSpc{1}$-invariant.
    Since $L_p E \cong \limil{n} E \sslash p^n$ and $\AffSpc{1}$-invariant sheaves are stable under limits,
    it suffices to show that $E \sslash p^n$ is $\AffSpc{1}$-invariant for all $n$.
    We assumed this to be the case for $n = 1$.
    For $n > 1$, recall that there is a fiber sequence 
    \begin{equation*}
        E \sslash p^n \to E \sslash p^{n+1} \to E \sslash p,
    \end{equation*}
    and we conclude by induction.
    Now consider the cartesian fracture square 
    \begin{center}
        \begin{tikzcd}
            E \cartsymb \ar[d] \ar[r] &\prod_p L_p E \ar[d] \\
            L_\Q E \ar[r] &L_\Q \prod_p L_p E \rlap{,}
        \end{tikzcd}
    \end{center}
    see e.g.\ \cite[Corollary 7.3]{mattis2024fracture}.
    By the above, all the $L_p E$ are $\AffSpc{1}$-invariant.
    Since $\AffSpc{1}$-invariant objects are stable under limits,
    this holds also for the product.
    Moreover, the rationalization of the $\AffSpc{1}$-invariant sheaf of spectra $F \coloneqq \prod_p L_p E$
    is again $\AffSpc{1}$-invariant: this holds since $L_\Q F$ is given by a filtered colimit of a diagram involving only $F$
    (cf.\ \cite[Lemma 2.6]{mattis2024fracture})
    and $\AffSpc{1}$-invariant sheaves of spectra are closed under colimits, cf.\ \cref{lem:etale-mot:A1-stable-under-filtered}.
    On the other hand, $L_\Q E$ is $\AffSpc{1}$-invariant by assumption.
    Hence, so is $E$ as a pullback of $\AffSpc{1}$-invariant sheaves of spectra.
\end{proof}

\begin{lem} \label{lem:etale-nilpotent:iota-shriek-evaluation}
    Let $k$ be an algebraically closed field, $X \in \Sm{k}$ and $x \in X$ a point.
    Let $F \in \ShvTopH{\et}{\smet{k}}$ or $F \in \Sp(\ShvTopH{\et}{\smet{k}})$.
    Then $(\iota_! F)(k) \cong (\iota_! F)(\mathcal{O}_{X, x}^{\sh})$.
\end{lem}
\begin{proof}
    If $k$ is algebraically closed, then $\ShvTopH{\et}{\smet{k}} \cong \An$ via evaluation at $k$
    by \cref{lem:etale:small-topos-over-strictly-local}.
    In particular, we see that $\iota_!$ is equivalent to the constant sheaf functor (as 
    $\An$ is the initial $\infty$-topos).
    In particular, $s^* \iota_! \cong \id{\An}$ for every point $s^* \colon \ShvTopH{\et}{\smet{k}} \to \An$.
    Since the points of the étale topos are given by strictly henselian local rings, cf.\ \cref{lem:etale:points}, the result follows.
\end{proof}

\begin{lem} \label{lem:etale-nilpotent:transfers}
    Let $k$ be a perfect field with $\cd{}{k} < \infty$ and write $e$ for the exponential characteristic of $k$.
    Let $E \in \SH{k}$ be a motivic spectrum such that multiplication by $e$ is invertible on $E$.
    Then $L_{\et} \omega^\infty E \in \Sp(\ShvTopH{\et}{\Sm{k}})$ is $\AffSpc{1}$-invariant.
\end{lem}
\begin{proof}
    Let $\overline{k}$ be an algebraic closure of $k$.
    Write $\rho^* \colon \ShvTopH{\et}{\Sm{k}} \to \ShvTopH{\et}{\Sm{\overline{k}}}$
    for the geometric morphism from \cref{defn:etale:morphism-to-local}.
    By \cref{lem:etale-motives:morphism-to-local-detects-A1} we see that $\rho^*$ detects $\AffSpc{1}$-invariance,
    i.e.\ it suffices to show that $\rho^* L_{\et} \omega^\infty E$ is $\AffSpc{1}$-invariant.
    Consider the following diagram:
    \begin{center}
        \begin{tikzcd}
            \Sp(\ShvTopH{\et}{\Sm{k}}) \ar[d, "\rho^*"] &\Sp(\ShvTop{\nis}{\Sm{k}}) \ar[l, "L_{\et}"] \ar[d, "\rho^*"] &\SH{k} \ar[d, "\rho^*"] \ar[l, "\omega^\infty"]\\
            \Sp(\ShvTopH{\et}{\Sm{\overline{k}}}) &\Sp(\ShvTop{\nis}{\Sm{\overline{k}}}) \ar[l, "L_{\et}"] &\SH{\overline{k}}\ar[l, "\omega^\infty"]\rlap{,}
        \end{tikzcd}
    \end{center}
    where the functor in the middle is defined in the same way as the functor on the left,
    and the functor on the right is its $\Gm$-stabilization.
    The diagram commutes: for the left square this holds as all involved functors 
    are induced by morphisms of sites, and they already commute on the level of sites.
    For the right square, it suffices to note that $\rho^* \colon \Sp(\ShvTop{\nis}{\Sm{k}}) \to \Sp(\ShvTop{\nis}{\Sm{\overline{k}}})$
    commutes with $\Gm$-loops, i.e.\ $\Omega_{\Gm} \rho^* X \cong \rho^* \Omega_{\Gm} X$.
    For this, see the discussion before \cite[Lemma A.7]{hoyois2013hopkinsmorel}.
    Therefore, $\rho^* L_{\et} \omega^\infty E \cong L_{\et} \omega^\infty \rho^* E$.
    Moreover, multiplication by $e$ on $\rho^* E$ is clearly an equivalence.
    Hence, we may assume that $E \in \SH{\overline{k}}$,
    i.e.\ that $k$ is already algebraically closed.

    Write $F \coloneqq L_{\et} \omega^\infty E \in \Sp(\ShvTopH{\et}{\Sm{k}})$.
    By \cref{lem:etale-nilpotent:check-mod-p},
    it suffices to show that $L_\Q F$ is $\AffSpc{1}$-invariant
    and that the $F \sslash p$ are $\AffSpc{1}$-invariant for all primes $p$.

    We first show that $L_\Q F$ is $\AffSpc{1}$-invariant.
    Since $L_{\et}$ and $\omega^\infty$ commute with filtered colimits 
    (see \cite[Lemma 6.1]{bachmannyakerson} for the latter),
    it follows from \cite[Corollary 2.7]{mattis2024fracture}
    that $L_\Q F \cong L_{\et} \omega^\infty L_\Q E$.
    We have seen in \cref{lem:rational:transfers-etale-descent}
    that $\omega^\infty L_\Q E$ satisfies étale hyperdescent,
    hence we see that $L_{\et} \omega^\infty L_\Q E \cong \omega^\infty L_\Q E$,
    which is $\AffSpc{1}$-invariant.

    We finish the proof by showing that $F \sslash p$ is $\AffSpc{1}$-invariant.
    If $p = e$, then multiplication by $p$ is invertible on $F$ by assumption,
    and hence $F \sslash p = 0$ is $\AffSpc{1}$-invariant.
    Hence, we may assume that $p \in k^\times$.
    Note that $F \sslash p$ is $p$-complete,
    and hence by \cite[Theorem 3.1]{bachmann2021remarksetalemotivicstable} it suffices to show that $F \sslash p$ is small,
    i.e. that $\iota_! \iota^* (F \sslash p) \to F \sslash p$ is an equivalence.
    Since we work in the étale topology, it hence suffices to show 
    that $(\iota_! \iota^* F \sslash p)(\mathcal{O}_{X, x}^{sh}) \to (F \sslash p)(\mathcal{O}_{X, x}^{sh})$ 
    is an equivalence for every $X \in \Sm{k}$ and $x \in X$, cf.\ \cref{lem:etale:points}.
    In fact, by \cite[Chapter II, Remark 2.17 (b)]{milne1980etale} we may check this on closed points 
    (we can use the $1$-categorical result since we work with hypersheaves and hence 
    can check equivalences on homotopy sheaves).
    So pick an $X \in \Sm{k}$ and $x \in X$ a closed point.
    This implies that $k(x)$ is a finite (hence algebraic) 
    field extension of $k$. Since $k$ is algebraically closed, it follows that $k = k(x)$.
    In particular, we see that $\mathcal{O}_{X, x}^{sh} \cong \mathcal{O}_{X, x}^{h}$.
    We first show that $(F \sslash p)(\mathcal{O}_{X, x}^{sh}) \cong (F \sslash p)(k)$.
    Indeed, since both $\mathcal{O}_{X, x}^{sh}$ and $k$ are strictly henselian,
    and $F \sslash p \cong L_{\et} \omega^\infty E \sslash p$, it suffices to show 
    that $(\omega^\infty E \sslash p)(\mathcal{O}_{X, x}^{sh}) \cong (\omega^\infty E \sslash p)(k)$.
    Since $\mathcal{O}_{X, x}^{sh} \cong \mathcal{O}_{X, x}^{h}$ 
    this follows from \cite[Theorem 1.2]{ananyevskiy2018rigidity} 
    as $p \in k^\times$ and multiplication by $p^2$ is null on $E \sslash p$.
    Now consider the following equivalences 
    \begin{equation*}
        (F \sslash p)(k)
        \cong (\iota_! \iota^* (F \sslash p))(k)
        \cong (\iota_! \iota^* (F \sslash p))(\mathcal{O}_{X, x}^{sh}).
    \end{equation*}
    Here we used that $k$ is living in the small étale site
    and \cref{lem:etale-nilpotent:iota-shriek-evaluation}.
    Combining these equivalences yields $(F \sslash p)(\mathcal{O}_{X, x}^{sh}) \cong (\iota_! \iota^* (F \sslash p))(\mathcal{O}_{X, x}^{sh})$
    which is what we wanted to show.
\end{proof}

\begin{prop} \label{lem:etale-nilpotent:2-effective-nilpotent}
    Let $k$ be a perfect field with $\cd{}{k} < \infty$ and write $e$ for the exponential characteristic of $k$.
    Let $X \in \Spc{k}_*$ be a nilpotent and $2$-effective motivic space
    that is $\Z[\frac{1}{e}]$-local.
    Then $L_{\et} X \in \ShvTopH{\et}{\Sm{k}}_*$ is étale $\AffSpc{1}$-nilpotent.
    In particular, $L_{\et} X$ is still $\AffSpc{1}$-invariant.
\end{prop}
\begin{proof}
    Choose motivic spaces $X_n$ and $\Pp^1$-infinite loop spaces $K_n = \pLoopP E_n$ 
    with $E_n \in \SH{k}$ a $2$-connective $2$-effective motivic spectrum
    as in \cref{lem:etale-nilpotent:effective-Postnikov-refinement}.
    We may moreover assume that multiplication by $e$ is invertible on $E_n$ by \cref{lem:etale-nilpotent:effective-Postnikov-refinement-rational}.

    As $L_{\et}$ is a geometric morphism, it preserves finite limits, and hence
    we get fiber sequences $L_{\et} X_{n+1} \to L_{\et} X_n \to L_{\et} K_n$,
    where $L_{\et} K_n \cong L_{\et} \pLoopP E_n \cong \pLoop L_{\et} \omega^\infty E_n$
    is a connected infinite loop sheaf.
    We moreover have $L_{\et} X \cong \limil{n} L_{\et} X_n$:
    for this, note that as the connectivity of the $K_i$ tends to $\infty$ as $i \to \infty$,
    the tower $(X_n)_n$ is highly connected.
    Hence, the claim follows from \cref{lem:highly-connected:morphism-of-covers-preserves-towers},
    using that the étale hypersheafification functor $L_{\et}$ upgrades to a morphism 
    of $\infty$-topoi with locally finite dimensional cover, cf.\ \cref{lem:etale:morphism-of-covers}.
    Hence, we are reduced to show that $L_{\et} \omega^\infty E_n$ is $\AffSpc{1}$-invariant.
    This was shown in \cref{lem:etale-nilpotent:transfers}.

    It follows now from \cref{lem:etale-nilpotent:nilpotent-A1-inv} that $L_{\et} X$ is in particular 
    $\AffSpc{1}$-invariant.
\end{proof}

\begin{lem}[Stability under retracts] \label{lem:etale-nilpotent:retracts}
    Let $X, Y \in \ShvTopH{\et}{\Sm{k}}_*$.
    Suppose that there exists a retract $L_p X \xrightarrow{i} L_p Y \xrightarrow{r}$.
    If $Y$ is $p$-completely small, then so is $X$.
\end{lem}
\begin{proof}
    By assumption $L_p Y \cong L_p \iota_! \iota^* Y$,
    and we have to show that $L_p X \cong L_p \iota_! \iota^* X$.
    For this, consider the commutative diagram 
    \begin{center}
        \begin{tikzcd}
            L_p \iota_! \iota^* L_p X \ar[d, "i"] &L_p \iota_! \iota^* X \ar[l, "\cong"'] \ar[r] &L_p X \ar[d, "i"] \\
            L_p \iota_! \iota^* L_p Y &L_p \iota_! \iota^* Y \ar[l, "\cong"'] \ar[r, "\cong"] &L_p Y\rlap{.}
        \end{tikzcd}
    \end{center}
    Since $\iota_!$ and $\iota^*$ preserve $p$-equivalences (as they are left adjoints),
    the left horizontal maps are equivalences.
    By assumption the vertical maps admit retractions,
    and the lower right horizontal map is an equivalence.
    As equivalences are stable under retracts, we conclude that the upper right horizontal map 
    is an equivalence. This is precisely what we wanted to prove.
\end{proof}

We now use the last few lemmas and the retract from \cref{lem:retract:main-thm}
to show that in fact over a perfect field $k$ any $\AffSpc{1}$-invariant $4$-connective 
étale sheaf is $p$-completely small.
We start with the case of an algebraically closed field
and then deduce the general case.
\begin{cor} \label{lem:etale-nilpotent:4-connective-nilpotent-alg-closed}
    Let $k$ be an algebraically closed field, $p \neq \operatorname{char}(k)$ a prime and $X \in \ShvTopH{\et}{\Sm{k}}_*$.
    If $X$ is $4$-connective and $\AffSpc{1}$-invariant then $X$ is $p$-completely small,
    i.e.\ $L_p X \cong L_p \iota_! \iota^* X$.
\end{cor}
\begin{proof}
    By \cref{lem:canonical:main-thm} there is an equivalence $X \cong \colimil{i} W_i$,
    with $W_i$ of the form $W_i \cong \Sigma^4 V_i$ for some $V_i \in \ShvTopH{\et}{\Sm{k}}_*$.

    Suppose for the moment that we know that $L_{\et, \AffSpc{1}} W_i$ is $p$-completely small.
    We show that the same is true for $X$.
    For this, note that we have maps 
    \begin{equation*}
        X = \colim{i} W_i \to \colim{i} L_{\et, \AffSpc{1}} W_i \to L_{\et, \AffSpc{1}} \colim{i} L_{\et, \AffSpc{1}} W_i.
    \end{equation*}
    Since $\AffSpc{1}$-equivalences are stable under colimits, and since $X$ is $\AffSpc{1}$-invariant,
    we see that 
    \begin{equation*}
        L_{\et, \AffSpc{1}} \colim{i} L_{\et, \AffSpc{1}} W_i \cong L_{\et, \AffSpc{1}} \colim{i} W_i \cong L_{\et, \AffSpc{1}} X \cong X.
    \end{equation*}
    In particular, by construction we therefore have a retract diagram 
    \begin{equation*}
        X \to \colim{i} L_{\et, \AffSpc{1}} W_i \to X.
    \end{equation*}
    By assumption the map $\iota_! \iota^* L_{\et, \AffSpc{1}} W_i \to L_{\et, \AffSpc{1}} W_i$
    is a $p$-equivalence for all $i$,
    and since both $\iota_!$ and $\iota^*$ commute with colimits, 
    and $p$-equivalences are stable under colimits,
    we see that the canonical map $\iota_! \iota^* \colimil{i} L_{\et, \AffSpc{1}} W_i \to \colimil{i} L_{\et, \AffSpc{1}} W_i$
    is a $p$-equivalence.
    Hence, also the retract $\iota_! \iota^* X \to X$ is a $p$-equivalence, i.e.\ $X$ is $p$-completely small.

    It remains to show that $L_{\et, \AffSpc{1}} (S^4 \wedge Z)$ is $p$-completely small
    for every $Z \in \ShvTopH{\et}{\Sm{k}}_*$.
    By applying \cref{lem:retract:main-thm} twice (here we use our assumption that $k$ is algebraically closed),
    there is a retract 
    \begin{equation*}
        L_p L_{\et, \AffSpc{1}} (S^4 \wedge Z) \to L_p L_{\et, \AffSpc{1}} (\Pp^1 \wedge \Pp^1 \wedge Z) \cong L_p L_{\et, \AffSpc{1}} (S^2 \wedge \Gm^{\wedge 2} \wedge Z).
    \end{equation*}
    Write $Y \coloneqq L_{\Z[\frac{1}{e}]} L_{\nis, \AffSpc{1}} (S^2 \wedge \Gm^{\wedge 2} \wedge \iota_{\nis} Z) \in \Spc{k}_*$ for the Nisnevich local version
    (where we also invert the exponential characteristic $e$ of $k$).
    This is clearly $2$-connective and $2$-effective.
    Hence, we see that $L_{\et} Y$ is étale $\AffSpc{1}$-nilpotent, cf.\ \cref{lem:etale-nilpotent:2-effective-nilpotent},
    and in particular $\AffSpc{1}$-invariant by \cref{lem:etale-nilpotent:nilpotent-A1-inv}, whence $L_{\Z[\frac{1}{e}]} L_{\et, \AffSpc{1}} (S^2 \wedge \Gm^{\wedge 2} \wedge Z) \cong L_{\et} Y$.
    Thus, $L_p L_{\et, \AffSpc{1}} (S^4 \wedge Z)$ is a retract of $L_p L_{\et} Y$
    (for this note that $L_p L_{\Z[\frac{1}{e}]} \cong L_p$ as functors since $p \neq e$ by assumption,
    and $p$-completion in particular inverts $\Z[\frac{1}{e}]$-local equivalences).
    Since by \cref{lem:etale-nilpotent:etale-nilpotent-implies-p-comp-rationally} $L_{\et} Y$ is also $p$-completely étale $\AffSpc{1}$-nilpotent,
    it is in particular $p$-completely small by \cref{lem:rigidity:ess-img}.
    Hence, we see that the same is true for $L_{\et, \AffSpc{1}} (S^4 \wedge Z)$ using \cref{lem:etale-nilpotent:retracts}.
\end{proof}
We now generalize the above corollary to arbitrary perfect fields of finite étale cohomological dimension.
\begin{prop} \label{lem:etale-nilpotent:4-connective-nilpotent}
    Let $k$ be a perfect field with $\cd{}{k} < \infty$, $p \neq \operatorname{char}(k)$ a prime and $X \in \ShvTopH{\et}{\Sm{k}}_*$.
    If $X$ is $4$-connective and $\AffSpc{1}$-invariant, then $X$ is $p$-completely small, i.e.\ $L_p \iota_! \iota^* X \cong L_p X$.
\end{prop}
\begin{proof}
    Let $\overline{k}$ be an algebraic closure of $k$.
    Consider the geometric morphisms 
    $\rho^* \colon \ShvTopH{\et}{\Sm{k}} \rightleftarrows \ShvTopH{\et}{\Sm{\overline{k}}} \noloc \rho_*$
    and 
    $\rho^* \colon \ShvTopH{\et}{\smet{k}} \rightleftarrows \ShvTopH{\et}{\smet{\overline{k}}} \noloc \rho_*$
    from \cref{defn:etale:morphism-to-local}, both left adjoints are conservative by \cref{lem:etale:morphism-to-local-conservative}.
    Moreover, $\iota^* \rho^* \cong \rho^* \iota^*$,
    as well as $\iota_! \rho^* \cong \rho^* \iota_!$ by \cref{lem:etale:morphism-to-local-commutation}.
    As $\rho^*$ is conservative, it suffices to show by \cite[Lemma 3.11]{mattis2024unstable} that 
    $\rho^* \iota_! \iota^* X \to \rho^* X$ is a $p$-equivalence.
    By the above, this is equivalent to the morphism
    \begin{equation*}
        \iota_! \iota^* \rho^* X \to \rho^* X.
    \end{equation*}
    As $\rho^* X$ is $4$-connective (since $\rho^*$ is a geometric morphism) 
    and $\AffSpc{1}$-invariant by \cref{lem:etale-motives:morphism-to-local-preserves-A1},
    it follows from \cref{lem:etale-nilpotent:4-connective-nilpotent-alg-closed} 
    that $\iota_! \iota^* \rho^* X \to \rho^* X$ is a $p$-equivalence.
\end{proof}

\begin{cor} \label{lem:etale-nilpotent:4-connective-main-thm}
    Let $k$ be a perfect field with $\cd{}{k} < \infty$ and $p \neq \operatorname{char}(k)$.
    Then 
    \begin{equation*}
        \iota_!^p \colon \complete{(\ShvTopH{\et}{\smet{k}}_*)} \to \complete{(\ShvTopH{\et}{\Sm{k}}_*)}
    \end{equation*}
    induces an equivalence between the full subcategory of $\complete{(\ShvTopH{\et}{\smet{k}}_*)}$ 
    consisting of those sheaves that are the $p$-completion of a $4$-connective sheaf,
    and the full subcategory of $\complete{(\ShvTopH{\et}{\Sm{k}}_*)}$ 
    consisting of those sheaves that are the $p$-completion of a $4$-connective $\AffSpc{1}$-invariant sheaf.
\end{cor}
\begin{proof}
    From \cref{lem:rigidity:iota-shriek-p-fully-faithful} we see that $\iota_!^p$ is fully faithful.
    First, let $Y \in \ShvTopH{\et}{\smet{k}}_{*, \ge 4}$ be $4$-connective.
    Note that $\iota_!^p L_p Y \cong L_p \iota_! Y$ (since $\iota_!$ preserves $p$-equivalences as a left adjoint),
    and by \cref{lem:rigidity:factors-through} this sheaf is $\AffSpc{1}$-invariant.
    Moreover, $\iota_! Y$ is $4$-connective since $\iota_!$ is a geometric morphism.

    To prove the corollary, it thus suffices to show that if $X \in \ShvTopH{\et}{\Sm{k}}_{*, \ge 4}$ is $4$-connective 
    and $\AffSpc{1}$-invariant, then $L_p X$ is in the essential image of $\iota_!^p$ restricted 
    to the full subcategory of $\complete{(\ShvTopH{\et}{\smet{k}}_*)}$ 
    consisting of those sheaves that are the $p$-completion of a $4$-connective sheaf.
    By \cref{lem:etale-nilpotent:4-connective-nilpotent} we see that $L_p X \cong L_p \iota_! \iota^* X \cong \iota_!^p L_p \iota^* X$.
    Since $\iota^*$ is a geometric morphism, we see that $\iota^* X$ is $4$-connective.
\end{proof}

%% file: application-strict.tex
\section{Application: Étale strict \texorpdfstring{$\AffSpc{1}$}{A1}-invariance}
In this section, we use our rigidity result to prove a weak version of Morel's theorem 
that strongly $\AffSpc{1}$-invariant Nisnevich sheaves of abelian groups are strictly $\AffSpc{1}$-invariant.
For this, we need the following list of definitions.
\begin{defn}
    Let $S$ be a scheme.
    Let $A \in \ShvTop{\et}{\Sm{S}, \Ab}$ be an étale sheaf of abelian groups.
    We say that $A$ is
    \begin{itemize}
        \item \emph{$n$-strictly étale $\AffSpc{1}$-invariant} for some $n \in \N$
            if $K(A, n) \in \ShvTopH{\et}{\Sm{S}}$ is $\AffSpc{1}$-invariant,
        \item \emph{strictly étale $\AffSpc{1}$-invariant}
            if $A$ is $n$-strictly étale $\AffSpc{1}$-invariant for all $n \in \N$,
        \item \emph{$p$-completely $n$-strictly étale $\AffSpc{1}$-invariant} for a prime $p$ and $n \in \N$
            if $L_p K(A, n) \in \ShvTopH{\et}{\Sm{S}}$ is $\AffSpc{1}$-invariant,
        \item \emph{$p$-completely strictly étale $\AffSpc{1}$-invariant} for a prime $p$
            if $A$ is $p$-completely $n$-strictly étale $\AffSpc{1}$-invariant for all $n \in \N$,
        \item \emph{rationally $n$-strictly étale $\AffSpc{1}$-invariant} for some $n \in \N$
            if $L_\Q K(A, n) \in \ShvTopH{\et}{\Sm{S}}$ is $\AffSpc{1}$-invariant,
        \item \emph{rationally strictly étale $\AffSpc{1}$-invariant}
            if $A$ is rationally $n$-strictly étale $\AffSpc{1}$-invariant for all $n \in \N$,
    \end{itemize}
\end{defn}

\begin{lem} \label{lem:strict:loops}
    Let $S$ be a scheme.
    Let $A \in \ShvTop{\et}{\Sm{S}, \Ab}$ be an étale sheaf of abelian groups,
    and $1 \le n \le m$.
    If $A$ is $m$-strictly étale $\AffSpc{1}$-invariant (resp.\ $p$-completely, resp.\ rationally),
    then $A$ is $n$-strictly étale $\AffSpc{1}$-invariant (resp.\ $p$-completely, resp.\ rationally).
\end{lem}
\begin{proof}
    We have $K(A, n) \cong \Omega^{m-n} K(A, m)$.
    As limits of $\AffSpc{1}$-invariant sheaves are $\AffSpc{1}$-invariant,
    the statement follows.
    For the $p$-complete or rational versions of the statement,
    note that $L_p K(A, n) \cong \tau_{\ge 1} \Omega^{m-n} L_p K(A, m)$,
    and $L_\Q K(A, n) \cong \tau_{\ge 1} \Omega^{m-n} L_\Q K(A, m)$
    by \cref{lem:fib-lem:main-thm} and \cite[Lemma 3.13]{mattis2024fracture}.
    To show that these sheaves are $\AffSpc{1}$-invariant,
    it suffices to note that the connected cover functor preserves $\AffSpc{1}$-invariance,
    cf.\ \cref{lem:etale-motives:connected-cover-invariant}.
\end{proof}

\begin{lem} \label{lem:strict:reduction}
    Let $S$ be a qcqs scheme.
    Let $A \in \ShvTop{\et}{\Sm{S}, \Ab}$ be an étale sheaf of abelian groups,
    and $n \ge 1$.
    If $A$ is $n$-strictly étale $\AffSpc{1}$-invariant,
    then $A$ is rationally $n$-strictly étale $\AffSpc{1}$-invariant.

    If $n \ge 2$, and $A$ is $n$-strictly étale $\AffSpc{1}$-invariant,
    then $A$ is $p$-completely $(n-1)$-strictly étale $\AffSpc{1}$-invariant
    for all primes $p$.
\end{lem}
\begin{proof}
    By assumption, $K(A, n)$ is $\AffSpc{1}$-invariant, we have to show that the same is true 
    for $L_p K(A, n-1)$ and $L_\Q K(A, n)$.

    We first prove that $L_\Q K(A, n)$ is $\AffSpc{1}$-invariant.
    Note that we have equivalences 
    \begin{align*}
        L_\Q K(A, n) 
        &\cong \pLoop L_\Q \Sigma^n HA \\
        &\cong \pLoop \colimil{\N} \Sigma^n HA \\
        &\cong \colimil{\N} \pLoop \Sigma^n HA \\
        &\cong \colimil{\N} K(A, n),
    \end{align*}
    where the colimit is over the $\N$-indexed diagram from \cite[Lemma 2.6]{mattis2024fracture}.
    In the third equivalence, we used that $\pLoop$ commutes with filtered colimits
    (combine \cite[Corollary 5.3.6.10]{highertopoi} with the fact that filtered colimits commute with finite limits in any $\infty$-topos,
    cf.\ \cite[Example 7.3.4.7]{highertopoi}).
    Since $\AffSpc{1}$-invariant sheaves are stable under filtered colimits 
    by \cref{lem:etale-mot:A1-stable-under-filtered}, and $K(A, n)$ is $\AffSpc{1}$-invariant by assumption, the result follows.

    We now show that if $n \ge 2$ also $L_p K(A, {n-1})$ is $\AffSpc{1}$-invariant.
    For this, note that we have equivalences
    \begin{align*}
        L_p K(A, n-1)
        &\cong \tau_{\ge 1} \pLoop L_p \Sigma^{n-1} HA \\
        &\cong \tau_{\ge 1} \pLoop \limil{k} \Sigma^{n-1} HA \sslash p^k \\
        &\cong \tau_{\ge 1} \limil{k} \pLoop \Sigma^{n-1} HA \sslash p^k \\
        &\cong \tau_{\ge 1} \limil{k} \Fib{K(A, n) \xrightarrow{p^k} K(A, n)}. 
    \end{align*}
    Here we used \cite[Lemmas 3.17 and 2.5]{mattis2024unstable} in the first and second equivalence,
    that $\pLoop$ preserves limits as a right adjoint in the third equivalence,
    and that 
    \begin{equation*}
        \Sigma^{n-1} HA \sslash p^k \cong \Fib{\Sigma^{n} HA \xrightarrow{p^k} \Sigma^{n} HA}
    \end{equation*}
    in the last equivalence.
    Hence, as $\AffSpc{1}$-invariant sheaves are stable under limits and connected covers (cf.\ \cref{lem:etale-motives:connected-cover-invariant} 
    for the latter), it suffices to show that $K(A, n)$ is $\AffSpc{1}$-invariant,
    which holds by assumption.
\end{proof}

\begin{prop} \label{lem:strict:p-completely-strict-if-4-strict}
    Let $k$ be a perfect field with $\cd{}{k} < \infty$ and $p \neq \operatorname{char}(k)$ be a prime.
    Let $A \in \ShvTop{\et}{\Sm{k}, \Ab}$ be an étale sheaf of abelian groups.
    Suppose that $A$ is $m$-strictly étale $\AffSpc{1}$-invariant
    for some $m \ge 4$.
    Then $A$ is $p$-completely strictly étale $\AffSpc{1}$-invariant, and 
    $L_p HA$ is an $\AffSpc{1}$-invariant sheaf of spectra.
\end{prop}
\begin{proof}
    It suffices to show that $L_p HA$ is $\AffSpc{1}$-invariant.
    Indeed, then for any $n \ge 1$ we have 
    $L_p K(A, n) \cong \tau_{\ge 1} \pLoop \Sigma^n L_p HA$
    by \cite[Corollary 3.18]{mattis2024unstable}. Hence, it is $\AffSpc{1}$-invariant as both 
    $\pLoop$ and $\tau_{\ge 1}$ preserve $\AffSpc{1}$-invariant sheaves,
    cf.\ \cref{lem:etale-mot:SHone-adjunction,lem:etale-motives:connected-cover-invariant}.

    Now, it is enough to show that $\iota_! \iota^* HA \to HA$ is a $p$-equivalence,
    as then $L_p HA \cong L_p \iota_! \iota^* HA$ is $\AffSpc{1}$-invariant by the proof of \cite[Corollary 6.2]{Bachmann2021etalerigidity}.
    For this, it is clearly enough to show that $\Sigma^m \iota_! \iota^* HA \to \Sigma^m HA$ is a $p$-equivalence.
    Since $\pLoop$ commutes with both $\iota_!$ and $\iota^*$, and detects 
    $p$-equivalences between $1$-connective objects by \cref{lem:p-comp:loop-detects-peq},
    it thus suffices to show that
    $\iota_! \iota^* \pLoop \Sigma^m HA \to \pLoop \Sigma^m HA$ is a $p$-equivalence.
    Since $\pLoop \Sigma^m HA \cong K(A, m)$ and $m \ge 4$, this follows from \cref{lem:etale-nilpotent:4-connective-nilpotent}.
\end{proof}

\begin{prop} \label{lem:strict:rationally-strict-if-1-strict}
    Let $k$ be a perfect field.
    Let $A \in \ShvTop{\et}{\Sm{k}, \Ab}$ be an étale sheaf of abelian groups.
    Suppose that $A$ is rationally $1$-strictly étale $\AffSpc{1}$-invariant.
    Then $A$ is rationally strictly étale $\AffSpc{1}$-invariant.
\end{prop}
\begin{proof}
    First note that for every $n \ge 1$ we have 
    $L_\Q K(A, n) \cong K(A_\Q, n)$ using e.g.\ \cite[Proposition 3.12]{mattis2024fracture}.
    It follows from \cite[Proposition 5.27]{voevodsky2000cohomological} that $K_{\nis}(A_\Q, n) \cong K_{\et}(A_\Q, n)$.
    But if $K_{\nis}(A_\Q, 1)$ is $\AffSpc{1}$-invariant,
    then so is $K_{\nis}(A_\Q, n)$ by Morel's theorem \cite[Theorem 1.7]{bachmann2024stronglya1}.
\end{proof}

\begin{prop} \label{lem:strict:at-the-characteristic}
    Let $k$ be a perfect field with $p = \operatorname{char}(k) > 0$, and
    $A \in \ShvTop{\et}{\Sm{k}, \Ab}$ be an étale sheaf of abelian groups.
    Suppose that $A$ is $p$-completely $3$-strictly étale $\AffSpc{1}$-invariant.
    Then $L_p HA \cong 0$ and $L_p K(A, n) = *$ for all $n \ge 1$.
    In particular, $A$ is $p$-completely strictly étale $\AffSpc{1}$-invariant,
    and $L_p HA$ is $\AffSpc{1}$-invariant.
\end{prop}
\begin{proof}
    This is a reformulation of \cref{lem:at-char:EM-vanishing}.
\end{proof}

\begin{thm} \label{lem:strictly:main-thm}
    Let $k$ be a perfect field with $\cd{}{k} < \infty$.
    Let $A \in \ShvTop{\et}{\Sm{k}, \Ab}$ be an étale sheaf of abelian groups.
    Assume that $A$ is $4$-strictly étale $\AffSpc{1}$-invariant.
    Then $A$ is strictly étale $\AffSpc{1}$-invariant.
\end{thm}
\begin{proof}
    We have to show that $A$ is $m$-strictly étale $\AffSpc{1}$-invariant for every $m \ge 1$.
    From \cite[Theorem 8.7]{mattis2024fracture} we have a pullback square 
    \begin{center}
        \begin{tikzcd}
            K(A, m) \cartsymb \ar[r] \ar[d] & \tau_{\ge 1} \prod_p L_p K(A, m) \ar[d] \\
            L_\Q K(A, m) \ar[r] & L_\Q \tau_{\ge 1} \prod_p L_p K(A, m)\rlap{.} 
        \end{tikzcd}
    \end{center}
    We have to show that the top left object is $\AffSpc{1}$-invariant.
    As limits of $\AffSpc{1}$-invariant sheaves are $\AffSpc{1}$-invariant,
    it suffices to show that the other objects in the above diagram are $\AffSpc{1}$-invariant.
    That $L_p K(A, m)$ is $\AffSpc{1}$-invariant for every $p$ was shown in \cref{lem:strict:p-completely-strict-if-4-strict}
    if $p \neq \operatorname{char}(k)$, and in \cref{lem:strict:at-the-characteristic}
    if $p = \operatorname{char}(k)$, using that $L_p K(A, 3)$ is $\AffSpc{1}$-invariant by \cref{lem:strict:reduction}.
    Moreover, that $L_\Q K(A, m)$ is $\AffSpc{1}$-invariant is exactly \cref{lem:strict:rationally-strict-if-1-strict},
    again using \cref{lem:strict:reduction}.
    Now also $\tau_{\ge 1} \prod_p L_p K(A, m)$ is $\AffSpc{1}$-invariant as a 
    connected cover of a limit 
    of $\AffSpc{1}$-invariant sheaves, cf.\ \cref{lem:etale-motives:connected-cover-invariant}.
    For the object in the bottom right corner we compute 
    \begin{align*}
        L_\Q \tau_{\ge 1} {\prod}_p L_p K(A, m)
        &\cong L_\Q \tau_{\ge 1} {\prod}_p \tau_{\ge 1} \pLoop L_p \Sigma^m HA &\text{\cite[Corollary 3.18]{mattis2024unstable}} \\
        &\cong L_\Q \tau_{\ge 1} {\prod}_p \pLoop L_p \Sigma^m HA &\text{\cite[Lemma 4.2]{mattis2024fracture}}\\
        &\cong L_\Q \tau_{\ge 1} \pLoop {\prod}_p L_p \Sigma^m HA &(\pLoop \text{ is a right adjoint})\\
        &\cong \tau_{\ge 1} \pLoop L_\Q {\prod}_p L_p \Sigma^m HA &\text{\cite[Lemma 3.15]{mattis2024fracture}}.
    \end{align*}
    Now $L_p HA$ is $\AffSpc{1}$-invariant (again by \cref{lem:strict:p-completely-strict-if-4-strict}
    if $p \neq \operatorname{char}(k)$, and by \cref{lem:strict:at-the-characteristic} if $p = \operatorname{char}(k)$),
    and hence so is the product over all primes.
    Since $\pLoop$ and $\tau_{\ge 1}$ preserve $\AffSpc{1}$-invariant sheaves
    by \cref{lem:etale-mot:SHone-adjunction,lem:etale-motives:connected-cover-invariant}, 
    it thus suffices to show that $L_\Q$ on $\Sp(\ShvTop{\et}{\Sm{k}})$
    preserves $\AffSpc{1}$-invariant sheaves of spectra.
    This holds since $L_\Q E$ is given by a filtered colimit of a diagram involving only $E$ (cf.\ \cite[Lemma 2.6]{mattis2024fracture}),
    and $\AffSpc{1}$-invariant sheaves of spectra are closed under colimits, cf. \cref{lem:etale-mot:A1-stable-under-filtered}.
    This proves the theorem.
\end{proof}

%% file: nilpotent-morphisms.tex
\section{Nilpotent morphisms}
Let $\topos X$ be an $\infty$-topos.
In this section we discuss nilpotent morphisms in $\topos X$.
We essentially copy the contents from \cite{asok2022localization}
about nilpotent morphisms of motivic spaces
to the setting of an $\infty$-topos.

\begin{defn}\label{def:nilpotent:morphism-def}
    Let $f \colon E \to B$ be a morphism in $\topos X_*$.
    We say that $f$ is \emph{nilpotent} 
    if the fiber $\Fib{f}$ is connected,
    and the action of $\pi_1(E)$ on $\pi_n(\Fib{f})$ is nilpotent for every $n$.
\end{defn}

\begin{rmk} \label{rmk:nilpotent:morphism-to-ast}
    Let $X \in \topos X_*$.
    Then $X$ is nilpotent if and only if $X \to *$ is nilpotent.
\end{rmk}

\begin{lem} \label{lem:nilpotent:composition}
    Let $q \colon E_2 \to E_1$ and $p \colon E_1 \to E_0$
    be two morphisms in $\topos X_*$,
    such that $\Fib{q}$, $\Fib{p}$ and $\Fib{pq}$ are all connected.
    If $p$ and $q$ are nilpotent, then so is $pq$.
\end{lem}
\begin{proof}
    The proof of \cite[Proposition 3.3.2]{asok2022localization}
    can also be used for the $\infty$-topos case.
\end{proof}

\begin{cor} \label{lem:nilpotent:morphism-of-object}
    Let $f \colon E \to B$ be a morphism in $\topos X_*$.
    Suppose that $E$ and $B$ are nilpotent,
    and that $\Fib{f}$ is connected.
    Then also $f$ is nilpotent.
\end{cor}
\begin{proof}
    This is an application of \cref{lem:nilpotent:composition}
    for $q = f \colon E \to B$ and $p \colon B \to *$,
    using \cref{rmk:nilpotent:morphism-to-ast}
    and the assumption that $B$, $E$ and $\Fib{f}$ are connected.
\end{proof}

\begin{prop}[Principal Moore--Postnikov tower] \label{lem:nilpotent:refined-Post-Moore-tower}
    Let $f \colon E \to B$ be a morphism in $\topos X_*$.
    If $f$ is nilpotent, then there is a system of sheaves $E_i$ under $E$ and over $B$
    with $E_0 \cong B$,
    such that $E \cong \limil{i} E_i$,
    and such that there are fiber sequences 
    \begin{equation*}
        E_{i+1} \to E_i \to K(A_i, n_i)
    \end{equation*} 
    with $A_i \in \AbObj{\Disc{\topos X}}$ and $n_i \ge 2$,
    such that $n_i \to \infty$ as $i \to \infty$.
\end{prop}
\begin{proof}
    The proof of \cite[Corollary 4.2.4]{asok2022localization}
    can also be used for the $\infty$-topos case.
\end{proof}

\begin{lem} \label{lem:nilpotent:refined-Post-Moore-tower-lower-bound}
    Let $f \colon E \to B$ be a morphism in $\topos X_*$.
    If $f$ is nilpotent, and $\Fib{f}$ is $k$-connective for some $k \ge 2$,
    then we may assume that in the situation of \cref{lem:nilpotent:refined-Post-Moore-tower}
    the integers $n_i$ are $\ge k$.
\end{lem}
\begin{proof}
    This is immediate from the construction.
\end{proof}

%% file: main.bbl
\newcommand{\etalchar}[1]{$^{#1}$}
\begin{thebibliography}{AWW17}

\bibitem[ABH23]{asok2023freudenthal}
Aravind Asok, Tom Bachmann, and Michael~J Hopkins.
\newblock On $\mathbb{P}^1$-stabilization in unstable motivic homotopy theory.
\newblock {\em arXiv preprint arXiv:2306.04631}, 2023.

\bibitem[AD18]{ananyevskiy2018rigidity}
Alexey Ananyevskiy and Andrei Druzhinin.
\newblock Rigidity for linear framed presheaves and generalized motivic cohomology theories.
\newblock {\em Advances in Mathematics}, 333:423--462, 2018.

\bibitem[AFH22]{asok2022localization}
Aravind Asok, Jean Fasel, and Michael~J Hopkins.
\newblock Localization and nilpotent spaces in homotopy theory.
\newblock {\em Compositio Mathematica}, 158(3):654--720, 2022.

\bibitem[AWW17]{Asok2017simplicialsuspension}
Aravind Asok, Kirsten Wickelgren, and Ben Williams.
\newblock The simplicial suspension sequence in $\mathbb{A}^1$–homotopy.
\newblock {\em Geometry \& Topology}, 21(4):2093–2160, May 2017.

\bibitem[Ayo14]{ayoub2014realisation}
Joseph Ayoub.
\newblock La r{\'e}alisation {\'e}tale et les op{\'e}rations de {G}rothendieck.
\newblock In {\em Annales Scientifiques de l'Ecole Normale Superieure}, volume~47, pages 1--145, 2014.

\bibitem[Bac21a]{Bachmann2021etalerigidity}
Tom Bachmann.
\newblock Rigidity in \'etale motivic stable homotopy theory.
\newblock {\em Algebraic \& Geometric Topology}, 21(1):173–209, February 2021.

\bibitem[Bac21b]{bachmann2021zerothP1}
Tom Bachmann.
\newblock The zeroth -stable homotopy sheaf of a motivic space.
\newblock {\em Journal of the Institute of Mathematics of Jussieu}, 22(3):1293–1317, August 2021.

\bibitem[Bac24]{bachmann2024stronglya1}
Tom Bachmann.
\newblock Strongly $\mathbb{A}^1$-invariant sheaves (after {F}. {M}orel), 2024.

\bibitem[BH17]{bachmann2020norms}
Tom Bachmann and Marc Hoyois.
\newblock Norms in motivic homotopy theory.
\newblock {\em Ast{\'e}risque}, 2017.

\bibitem[BH21]{bachmann2021remarksetalemotivicstable}
Tom Bachmann and Marc Hoyois.
\newblock Remarks on \'etale motivic stable homotopy theory, 2021.

\bibitem[BT06]{bass2006milnor}
Hyman Bass and John Tate.
\newblock The {M}ilnor ring of a global field.
\newblock In {\em “Classical” Algebraic K-Theory, and Connections with Arithmetic: Proceedings of the Conference held at the Seattle Research Center of the Battelle Memorial Institute, from August 28 to September 8, 1972}, pages 347--446. Springer, 2006.

\bibitem[BY20]{bachmannyakerson}
Tom Bachmann and Maria Yakerson.
\newblock Towards conservativity of $\mathbf{G}_m$–stabilization.
\newblock {\em Geometry \& Topology}, 24(4):1969–2034, November 2020.

\bibitem[CD15]{Cisinski_2015}
Denis-Charles Cisinski and Frédéric Déglise.
\newblock {\'E}tale motives.
\newblock {\em Compositio Mathematica}, 152(3):556–666, September 2015.

\bibitem[CD19]{cisinski2019triangulated}
Denis-Charles Cisinski and Fr{\'e}d{\'e}ric D{\'e}glise.
\newblock {\em Triangulated categories of mixed motives}, volume~6.
\newblock Springer, 2019.

\bibitem[CM21]{Clausen_2021}
Dustin Clausen and Akhil Mathew.
\newblock Hyperdescent and étale {K}-theory.
\newblock {\em Inventiones mathematicae}, 225(3):981–1076, April 2021.

\bibitem[CS23]{cesnavicius2023purityflatcohomology}
Kestutis Cesnavicius and Peter Scholze.
\newblock Purity for flat cohomology, 2023.

\bibitem[D{\'e}g23]{deglise2023notes}
Fr{\'e}d{\'e}ric D{\'e}glise.
\newblock Notes on {M}ilnor-{W}itt {K}-theory.
\newblock {\em arXiv preprint arXiv:2305.18609}, 2023.

\bibitem[DH21]{DevalapurkarHaine}
Sanath Devalapurkar and Peter Haine.
\newblock On the {J}ames and {H}ilton-{M}ilnor splittings, and the metastable {EHP} sequence.
\newblock {\em Doc. Math.}, 26:1423--1464, 2021.

\bibitem[Fel21]{feld2021milnor}
Niels Feld.
\newblock Milnor-witt homotopy sheaves and morel generalized transfers.
\newblock {\em Advances in Mathematics}, 393:108094, 2021.

\bibitem[GAV71]{SGA4}
Alexander Grothendieck, Michael Artin, and Jean-Louis Verdier.
\newblock {\em Th{\'e}orie des {T}opos et {C}ohomologie {\'E}tale des {S}ch{\'e}mas {I}, {II}, {III}}, volume 269, 270, 305 of {\em Lecture Notes in Mathematics}.
\newblock Springer, 1971.

\bibitem[Hoy13]{hoyois2013hopkinsmorel}
Marc Hoyois.
\newblock From algebraic cobordism to motivic cohomology.
\newblock {\em Journal für die reine und angewandte Mathematik (Crelles Journal)}, 2015(702):173–226, June 2013.

\bibitem[Lur09]{highertopoi}
Jacob Lurie.
\newblock Higher {T}opos {T}heory, 2009.

\bibitem[Lur17]{higheralgebra}
Jacob Lurie.
\newblock Higher {A}lgebra.
\newblock Available online at \url{https://www.math.ias.edu/~lurie/papers/HA.pdf}, 2017.

\bibitem[Lur18]{sag}
Jacob Lurie.
\newblock Spectral {A}lgebraic {G}eometry.
\newblock Available online at \url{https://www.math.ias.edu/~lurie/papers/sag-rootfile.pdf}, 2018.

\bibitem[Mat24a]{mattis2024fracture}
Klaus Mattis.
\newblock Unstable arithmetic fracture squares in $\infty$-topoi.
\newblock {\em arXiv preprint}, \href{https://arxiv.org/abs/2404.18618}{arXiv:2404.18618}, 2024.

\bibitem[Mat24b]{mattis2024unstable}
Klaus Mattis.
\newblock Unstable $p$-completion in motivic homotopy theory.
\newblock {\em arXiv preprint}, \href{https://arxiv.org/abs/2401.17848}{arXiv:2401.17848}, 2024.

\bibitem[Mil80]{milne1980etale}
James~S Milne.
\newblock {\em \'Etale cohomology (PMS-33)}.
\newblock Princeton university press, 1980.

\bibitem[Mor03]{morel2003trieste}
F.~Morel.
\newblock An introduction to $\mathbb{A}^1$-homotopy theory.
\newblock 15:p. 357–441, Sep 2003.
\newblock Available online at \url{https://inis.iaea.org/records/7w14t-9xs10}.

\bibitem[Mor12]{morel2012a1}
Fabien Morel.
\newblock {\em $\mathbb{A}^1$-algebraic topology over a field}, volume 2052.
\newblock Springer, 2012.

\bibitem[MP11]{MCAT}
J.~P. May and K.~Ponto.
\newblock {\em More {C}oncise {A}lgebraic {T}opology: {L}ocalization, {C}ompletion, and {M}odel {C}ategories}.
\newblock University of Chicago Press, Chicago, 2011.

\bibitem[NPR24]{naumann2024symmetricmonoidalfracturesquare}
Niko Naumann, Luca Pol, and Maxime Ramzi.
\newblock A symmetric monoidal fracture square, 2024.

\bibitem[Pst22]{pstragowski2022syntheticspectracellularmotivic}
Piotr Pstrągowski.
\newblock Synthetic spectra and the cellular motivic category, 2022.

\bibitem[{Sta}23]{stacks-project}
The {Stacks project authors}.
\newblock The {S}tacks project.
\newblock \url{https://stacks.math.columbia.edu}, 2023.

\bibitem[SV96]{suslin1996singular}
Andrei Suslin and Vladimir Voevodsky.
\newblock Singular homology of abstract algebraic varieties.
\newblock {\em Inventiones mathematicae}, 123(1):61--94, 1996.

\bibitem[V{\etalchar{+}}00]{voevodsky2000cohomological}
Vladimir Voevodsky et~al.
\newblock Cohomological theory of presheaves with transfers.
\newblock {\em Cycles, transfers, and motivic homology theories}, 143:87--137, 2000.

\end{thebibliography}
